\documentclass[12pt,a4]{amsart}
\usepackage[a4paper, left=26mm, right=26mm, top=28mm, bottom=34mm]{geometry}

\usepackage{amsmath, amscd}
\usepackage{amssymb}
\usepackage[alphabetic]{amsrefs}
\usepackage[T1]{fontenc}
\usepackage[all]{xy}
\usepackage{mathtools}
\usepackage{footmisc}
\usepackage{stmaryrd}
\usepackage{mathrsfs}

\usepackage{tikz}
\usetikzlibrary{calc}
\usetikzlibrary{matrix,arrows,decorations.pathmorphing}
\usepackage{tikz-cd}

\usepackage{relsize} 
\usepackage[bbgreekl]{mathbbol} 
\usepackage{amsfonts} 
\usepackage[colorlinks=true, hyperindex, linkcolor=cyan, citecolor=, pagebackref=false, urlcolor=cyan, linktocpage]{hyperref}

\DeclareSymbolFontAlphabet{\mathbb}{AMSb} %to ensure that the meaning of \mathbb does not change
\DeclareSymbolFontAlphabet{\mathbbl}{bbold} 
\newcommand{\Prism}{{\mathlarger{\mathbbl{\Delta}}}}

\theoremstyle{definition}
\newtheorem{thm}{Theorem}[section]
\newtheorem{mainthm}{Theorem}

\newtheorem{lem}[thm]{Lemma}
\newtheorem{cor}[thm]{Corollary}
\newtheorem{prop}[thm]{Proposition}

\theoremstyle{definition}

\newtheorem{rem}[thm]{Remark}

\newtheorem{dfn}[thm]{Definition}

\newtheorem{construction}[thm]{Construction}
\newtheorem{setting}[thm]{Setting}
\newtheorem{notation}[thm]{Notation}

\newcommand{\bA}{\mathbb{A}}
\newcommand{\bB}{\mathbb{B}}

\newcommand{\bL}{\mathbb{L}}
\newcommand{\bN}{\mathbb{N}}
\newcommand{\bZ}{\mathbb{Z}}
\newcommand{\bF}{\mathbb{F}}

\newcommand{\bQ}{\mathbb{Q}}

\newcommand{\cD}{\mathcal{D}}
\newcommand{\cE}{\mathcal{E}}
\newcommand{\cF}{\mathcal{F}}

\newcommand{\cI}{\mathcal{I}}
\newcommand{\cJ}{\mathcal{J}}

\newcommand{\cM}{\mathcal{M}}
\newcommand{\cN}{\mathcal{N}}
\newcommand{\cO}{\mathcal{O}}
\newcommand{\cP}{\mathcal{P}}

\newcommand{\sC}{\mathscr{C}}

\newcommand{\fM}{\mathfrak{M}}
\newcommand{\fS}{\mathfrak{S}}
\newcommand{\fU}{\mathfrak{U}}
\newcommand{\fV}{\mathfrak{V}}

\newcommand{\fX}{\mathfrak{X}}
\newcommand{\fY}{\mathfrak{Y}}

\newcommand{\isom}{\stackrel{\sim}{\to}}

\newcommand{\et}{\text{\'{e}t}}

\newcommand{\Spec}{\mathrm{Spec}}
\newcommand{\Spf}{\mathrm{Spf}}

\title{Log prismatic $F$-crystals and realization functors}

\author{Kentaro Inoue}

\begin{document}

\begin{abstract}
    Log prismatic cohomology theory developed by Koshikawa-Yao involves coefficient objects, called log prismatic $F$-crystals. In this paper, we construct and study realization functors from the category of log prismatic $F$-crystals to the category of coefficient objects of other $p$-adic cohomology theories, in the setting where boundary divisors may involve horizontal components.
\end{abstract}

\maketitle

\tableofcontents

\section{Introduction}
The theory of prismatic cohomology is introduced in \cite{bs22}. This cohomology theory involves coefficient objects called prismatic $F$-crystals. We let $\mathrm{Vect}^{\varphi}(\fX_{\Prism})$ denote the category of prismatic $F$-crystals on a bounded $p$-adic formal scheme $\fX$. The prismatic cohomology theory specializes to other $p$-adic cohomology theories, and corresponding coefficient objects behave similarly. To be precise, the works \cites{gr24,iky23} give the following:
for a $p$-adic formal scheme which is smooth over $\cO_{K}$, an \'{e}tale realization functor
\[
T_{\et}\colon \mathrm{Vect}^{\varphi}(\fX_{\Prism})\to \mathrm{Loc}_{\bZ_{p}}(\fX_{\eta})
\]
and a crystalline realization functor
\[
T_{\mathrm{fisoc}}\colon \mathrm{Vect}^{\varphi}(\fX_{\Prism})\to \mathrm{FilIsoc}^{\varphi}(\fX)
\] 
are constructed, where $\mathrm{Loc}_{\bZ_{p}}(\fX_{\eta})$ is the category of $\bZ_{p}$-local systems on the generic fiber $\fX_{\eta}$ and $\mathrm{FilIsoc}^{\varphi}(\fX)$ is the category of filtered $F$-isorystals on $\fX$. Moreover, for a prismatic $F$-crystal $\cE$, the coefficient objects $T_{\et}(\cE)$ and $T_{\mathrm{fisoc}}(\cE)$ are associated in the sense of Faltings.

The log version of prismatic cohomology theory is established in \cites{kos22,ky23}. Log prismatic cohomology theory has coefficient objects called log prismatic $F$-crystals. We let $\mathrm{Vect}^{\varphi}((\fX,\cM_{\fX})_{\Prism})$ denote the category of log prismatic $F$-crystals on a bounded $p$-adic log formal scheme. Log prismatic cohomology theory enables us to treat $p$-adic Hodge theory on a smooth variety over a $p$-adic field $K$ admitting an open immersion to a proper semi-stable scheme over $\cO_{K}$ while (non-log) prismatic cohomology theory behaves well only for a proper smooth variety over $K$ with good reduction. We make a precise setup.

\begin{dfn}
Let $k$ be a perfect field, $K$ a totally ramified finite extension of $W(k)[1/p]$, and $\cO_{K}$ a valuation ring of $K$. We fix a uniformizer $\pi$ of $\cO_{K}$. Consider the prelog ring $(R^{0},\bN^{r})$ defined as follows. Let $l,m,n\in \bZ_{\geq 0}$, $r\coloneqq m+n$. We define $R^{0}$ as follows: 
    \[
    R^{0}\coloneqq
    \begin{cases}
    \cO_{K}\langle x^{\pm1}_{1},\dots,x^{\pm1}_{l},y_{1},\dots,y_{m}\rangle \ \ (n=0) \\
    \cO_{K}\langle x^{\pm1}_{1},\dots,x^{\pm1}_{l},y_{1},\dots,y_{m},z_{1},\dots,z_{n} \rangle/(\prod^{n}_{k=1}z_{k}-\pi) \ \ (n\geq 1).
    \end{cases}
    \]
    We define a prelog structure $\bN^{r}\to R^{0}$ by 
    \[
    e_{i}\mapsto 
    \begin{cases}
        y_{i} \ (1\leq i\leq m) \\
        z_{i-m} \ (m+1\leq i\leq r).
    \end{cases}
    \]
   A $p$-adic log formal scheme $(\mathrm{Spf}(R),\cM_{R})$ over $\cO_{K}$ is called \emph{small affine} if there exists a strict \'{e}tale map $(\mathrm{Spf}(R),\cM_{R})\to (\mathrm{Spf}(R^{0}),\bN^{r})^{a}$ for some $l,m,n$. Such a strict \'{e}tale map  is called a \emph{framing}. A $p$-adic log formal scheme $(\fX,\cM_{\fX})$ over $\cO_{K}$ is called \emph{semi-stable} if $(\fX,\cM_{\fX})$ has an \'{e}tale covering consisting of small affine log formal schemes.
\end{dfn}

The first goal of this paper is to generalize the construction of realization functors and the relationship between them to log prismatic $F$-crystals on semi-stable log formal schemes over $\cO_{K}$. 

\begin{mainthm}[see Theorem \ref{existence of filfisoc real}]\label{main thm fisoc real}
    For a semi-stable log formal scheme $(\fX,\cM_{\fX})$ over $\cO_{K}$, there exist an \'{e}tale realization functor
    \[
    T_{\et}\colon \mathrm{Vect}^{\varphi}((\fX,\cM_{\fX})_{\Prism})\to \mathrm{Loc}_{\bZ_{p}}((\fX,\cM_{\fX})_{\eta})
    \]
    and a crystalline realization functor
    \[
    T_{\mathrm{fisoc}}\colon \mathrm{Vect}^{\varphi}((\fX,\cM_{\fX})_{\Prism})\to \mathrm{FilIsoc}^{\varphi}(\fX,\cM_{\fX}),
    \]
    where $\mathrm{Loc}_{\bZ_{p}}((\fX,\cM_{\fX})_{\eta})$ denotes the category of Kummer \'{e}tale $\bZ_{p}$-local systems on the generic fiber $(\fX,\cM_{\fX})_{\eta}$ and $\mathrm{FilIsoc}^{\varphi}(\fX,\cM_{\fX})$ denotes the category of filtered $F$-isocrystals on $(\fX,\cM_{\fX})$. Moreover, for a log prismatic $F$-crystal $\cE$, coefficient objects $T_{\et}(\cE)$ and $T_{\mathrm{fisoc}}(\cE)$ are associated in the sense of Definition \ref{def of semist loc sys}(2) (cf. \cites{fal89,ai12,tsu}).
\end{mainthm}

\begin{rem}
    In fact, we work with the category of analytic log prismatic $F$-crystals, denoted by $\mathrm{Vect}^{\mathrm{an},\varphi}((\fX,\cM_{\fX})_{\Prism})$, which is larger than $\mathrm{Vect}^{\varphi}((\fX,\cM_{\fX})_{\Prism})$.
\end{rem}

\begin{rem}
    When $(\fX,\cM_{\fX})$ has  framings with $m=n=0$ (resp.~ $m=0$ and $n\geq 1$) \'{e}tale locally, this theorem is proved in \cites{dlms23,gr24} (resp.~ \cite{dlms24}). Moreover, the work \cite{iky24} gives the description of $T_{\mathrm{fisoc}}$ using Nygaard filtrations when $(\fX,\cM_{\fX})$ has a framing with $m=n=0$ \'{e}tale locally. We also generalize this description for arbitrary $m,n$.
\end{rem}

The next goal is to prove the fully faithfulness and bi-exactness of \'{e}tale realization functors.

\begin{mainthm}[see Theorem \ref{fully faithfulness of et real for anal pris} and Theorem \ref{biexactness of et real}]\label{mainthm for et real}
    Let $(\fX,\cM_{\fX})$ be a semi-stable log formal scheme over $\cO_{K}$. Consider the \'{e}tale realization functor 
    \[
    T_{\et}\colon \mathrm{Vect}^{\mathrm{an},\varphi}((\fX,\cM_{\fX})_{\Prism})\to \mathrm{Loc}_{\bZ_{p}}((\fX,\cM_{\fX})_{\eta}).
    \]
    \begin{enumerate}
        \item The functor $T_{\et}$ is fully faithful.
        \item When $(\fX,\cM_{\fX})$ admits framings with $n\leq 1$ \'{e}tale locally, $T_{\et}$ gives a bi-exact equivalence to the essential image. In other words, a sequence of analytic log prismatic $F$-crystals
    \[
    0\to \cE_{1}\to \cE_{2}\to \cE_{3}\to 0
    \]
    is exact if and only if the corresponding sequence of Kummer \'{e}tale local systems
    \[
    0\to T_{\et}(\cE_{1})\to T_{\et}(\cE_{2})\to T_{\et}(\cE_{3})\to 0
    \]
    is exact.
    \end{enumerate}
\end{mainthm}

\begin{rem}
    When $(\fX,\cM_{\fX})$ has  framings with $m=n=0$ (resp.~ $m=0$ and $n\geq 1$) \'{e}tale locally, the fully faithfulness of the \'{e}tale realization functor is proved in \cites{bs23,dlms23,gr24} (resp.~ \cite{dlms24}). While the detailed calculation concerning Breuil-Kisin rings is the core of the proof in \cites{dlms23,dlms24}, the works \cites{bs23,gr24} use the fully faithfulness of Fargues's functor. In this paper, we adopt the method using the fully faithfulness of Fargues's functor. 
\end{rem}

\begin{rem}
     When $(\fX,\cM_{\fX})$ has  framings with $m=n=0$, the bi-exactness is proved in \cite{iky24}.
\end{rem}

\begin{rem}
    It is an important problem to specify the essential image of the \'{e}tale realization functor
    \[
    T_{\et}\colon \mathrm{Vect}^{\mathrm{an},\varphi}((\fX,\cM_{\fX})_{\Prism})\to \mathrm{Loc}_{\bZ_{p}}((\fX,\cM_{\fX})_{\eta}).
    \]
    This problem is indeed solved in some cases: When $\fX=\mathrm{Spf}(\cO_{K})$ and $\cM_{\fX}$ is trivial, it is proved that the functor $T_{\et}$ gives the equivalence to the category of crystalline $\bZ_{p}$-representation of $\mathrm{Gal}(\overline{K}/K)$ in \cite{bs23}. More generally, when $\cM_{\fX}$ is trivial and $\fX$ is higher dimensional, the essential image of $T_{\et}$ is the full subcategory of crystalline $\bZ_{p}$-local systems. The work \cite{gr24} proves this by the method generalizing \cite{bs23}, and the work \cite{dlms23} proves this by using the detailed computation concerning Breuil-Kisin rings. When $\fX=\mathrm{Spf}(\cO_{K})$ and $\cM_{\fX}$ is the standard log structure, the work \cite{dl24} shows that the essential image of $T_{\et}$ is the full subcategory of semi-stable $\bZ_{p}$-representations of $\mathrm{Gal}(\overline{K}/K)$. When $(\fX,\cM_{\fX})$ has framings with $m=0$ and $n\geq 1$ \'{e}tale locally, it is shown that $T_{\et}$ gives the equivalences to the category of semi-stable $\bZ_{p}$-local systems in \cite{dlms24} by a similar approach to \cite{dlms23}. 

    In this paper, we do not treat this problem. Naive expectation is that the essential image of $T_{\et}$ is the full subcategory of semi-stable $\bZ_{p}$-local systems in the sense of Definition \ref{def of semist loc sys}.
\end{rem}

\begin{rem}
    In \cite{ino25}, the work in this paper is applied to the construction of log prismatic realization on toroidal compactifications of Shimura varieties of Hodge type.
\end{rem}

We briefly explain the structure of the article. In Section \ref{section setup}, we define semi-stable log formal schemes and arrange notation and terminology concerning them. When we refer to semi-stable log formal schemes, we use notation in this section without reference. In Section \ref{section isoc}, we recall basics on absolute log crystalline sites and $F$-isocrystals. One of targets of this section is Proposition \ref{Fisoc on log crys site and Fisoc with monodromy ope}, which generalizes \cite[Theorem 2.21]{gy24}. This allows us to describe $F$-isocrystals on the special fibers of semi-stable log formal schemes admitting framings with $n=1$ \'{e}tale locally as $F$-isocrystals equipped with monodromy operators on the special fibers of horizontally semi-stable log formal schemes. In Section \ref{section loc sys}, we introduce de Rham (resp.  crystalline)(resp. semi-stable) period sheaves and the notion of semi-stable $\bZ_{p}$-local systems on the generic fibers of semi-stable log formal schemes. A key idea to do this is to consider the notion of \emph{big log affinoid perfectoids} (Definition \ref{def of big log affinoid perfd}). Although de Rham period sheaves on log adic spaces are already defined in \cite{dllz23b}, we give another construction which is similar to one in \cite{sch16} in spirit. In Section \ref{section pris crys}, we review fundamental results on prismatic $F$-crystals and prove ``Kummer quasi-syntomic descent'' for morphisms of analytic prismatic $F$-crystals (Proposition \ref{kqsyn descent for anal pris crys}). This descent property enables us to make an analogous argument to \cites{bs22,gr24,iky24} in order to prove Theorem \ref{mainthm for et real}. In Section \ref{section realization functor}, we construct several realization functors and prove Theorem \ref{main thm fisoc real} and Theorem \ref{mainthm for et real}.

\subsection*{Acknowledgements}
The author is grateful to his advisor, Tetsushi Ito, for useful discussions and warm encouragement. A special thanks goes to Teruhisa Koshikawa for innumerable discussions and advice. Moreover, the author would like to thank Kai-Wen Lan, Shengkai Mao, Koji Shimizu, Peihang Wu, and Alex Youcis for helpful comments. This work was supported by JSPS KAKENHI Grant Number 23KJ1325.

\subsection*{Notation and conventions}\noindent

\begin{itemize}
\item The symbol $p$ always denotes a prime.
\item All rings and monoids are commutative.
\item Limits of categories mean $2$-limits. 
\item Formal schemes are assumed to admit a finitely generated ideal of definition Zariski locally. For a property $P$ of morphisms of schemes, an adic morphism of formal schemes is called $\cP$ if it is adically $P$.
\item For a monoid $P$ and an integer $n\geq 1$, let $P^{1/n}$ denote the monoid $P$ with $P\to P^{1/n}$ mapping $p$ to $p^{n}$. The colimit of $P^{1/n}$ with respect to $n\geq 1$ is denoted by $P_{\bQ_{\geq 0}}$.
\item The $i$-th standard basis of $\bN^{r}$ is denote by $e_{i}$ for $1\leq i\leq r$.
\item For $I$-adically complete ring $A$ and a monoid $M$, a ring $A\langle M\rangle$ denotes the $I$-adic completion of the monoid algebra $A[M]$.
\item For a prelog ring $(A,M)$, the associated log scheme is denoted by $(\mathrm{Spec}(A),M)^{a}$. 
\item Let $\fX$ be a formal scheme and $n\geq 1$ be an integer. Assume that $p$ is topologically nilpotent on $\fX$ when $n=1$. The category of finite and locally free group schemes (resp. truncated Barsotti-Tate groups of level $n$) (resp. $p$-divisible groups) on $\fX$ is denoted by $\mathrm{Fin}(\fX)$ (resp. $\mathrm{BT}_{n}(\fX)$) (resp. $\mathrm{BT}(\fX)$). 
\end{itemize}

\section{Setups for semi-stable log formal schemes}\label{section setup}\noindent

In this section, we summarize notations around small affine log formal schemes and framings on them. Throughout this paper, we use the notation in this subsection without mentioning it. 

Fix a uniformizer $\pi$ of $\cO_{K}$ and $E(t)\in W[t]$ be a corresponding Eisenstein polynomial. Let $l,m,n\in \bZ_{\geq 0}$, $r\coloneqq m+n$. Consider the prelog ring $(R^{0},\bN^{r})$ defined as follows: 
\[
R^{0}\coloneqq
    \begin{cases}
    \cO_{K}\langle x^{\pm1}_{1},\dots,x^{\pm1}_{l},y_{1},\dots,y_{m}\rangle \ \ (n=0) \\
    \cO_{K}\langle x^{\pm1}_{1},\dots,x^{\pm1}_{l},y_{1},\dots,y_{m},z_{1},\dots,z_{n} \rangle/(\prod^{n}_{k=1}z_{k}-\pi) \ \ (n\geq 1)
    \end{cases}
    \]
equipped with a prelog structure $\bN^{r}\to R^{0}$ given by 
\[
    e_{i}\mapsto 
    \begin{cases}
        y_{i} \ (1\leq i\leq m) \\
        z_{i-m} \ (m+1\leq i\leq r).
    \end{cases}
    \]

 We define an integral perfectoid $\cO_{C}$-algebra $\widetilde{R}^{0}_{\infty}$ as follows:
    \[
    \widetilde{R}^{0}_{\infty}\coloneqq
    \begin{cases}
        \cO_{C}\langle \{x_{i}^{\pm1/p^{\infty}}\},\{y_{j}^{\bQ_{\geq 0}}\}\rangle \ \ (n=0)\\
        \cO_{C}\langle \{x_{i}^{\pm1/p^{\infty}}\},\{y_{j}^{\bQ_{\geq 0}}\},\{z_{k}^{\bQ_{\geq 0}}\} \rangle/(\prod^{n}_{k=1}z_{k}^{1/s}-\pi^{1/s})_{s\geq 1} \ \ (n\geq 1).
    \end{cases}
    \]
    We define a prelog structure $\bQ_{\geq 0}^{r}\to \widetilde{R}^{0}_{\infty}$ by 
    \[
    se_{i}\mapsto 
    \begin{cases}
        y_{i}^{s} \ (1\leq i\leq m) \\
        z_{i-m}^{s} \ (m+1\leq i\leq r).
    \end{cases}
    \]
    for $s\in \bQ_{\geq 0}$. If we put $R^{0}_{\infty}\coloneqq R^{0}\widehat{\otimes}_{\bZ_{p}\langle \bN^{r}\rangle} \bZ_{p}\langle \bQ_{\geq 0}^{r}\rangle$, there is a sequence of $p$-complete prelog rings
    \begin{equation}
        (R^{0},\bN^{r})\to (R^{0}_{\infty},\bQ_{\geq 0}^{r})\to (\widetilde{R}^{0}_{\infty},\bQ_{\geq 0}^{r}). \tag{1}
    \end{equation}
    The sequence of $p$-adic fs log formal schemes associated with the sequence $(1)$ is denoted by
    \begin{equation}
        (\mathrm{Spf}(\widetilde{R}^{0}_{\infty}),\cM_{\widetilde{R}^{0}_{\infty}})\to (\mathrm{Spf}(R^{0}_{\infty}),\cM_{R^{0}_{\infty}})\to (\mathrm{Spf}(R),\cM_{R}). \tag{2}
    \end{equation}

    Let $\widetilde{U}^{0}_{\infty}$ be the object of $(\mathrm{Spf}(R^{0}),\bN^{r})^{a}_{\eta,\mathrm{prok\et}}$ written as the limit of finite Kummer \'{e}tale coverings
    \[
    (\mathrm{Spf}(\cO_{L}\langle \{x_{i}^{\pm 1/p^{m}}\},\{y_{j}^{1/m}\},\{z_{k}^{1/m}\} \rangle/(\prod_{k=1}^{n}z_{k}^{1/m}-\pi^{1/m})),\frac{1}{m}\bN^{r})^{a}_{\eta}\to (\mathrm{Spf}(R^{0}),\bN^{r})^{a}_{\eta},
    \]
    where $m\geq 1$ and $L$ is a finite extension of $K[\pi^{1/m}]$ contained in $C$. Then $\widetilde{U}^{0}_{\infty}$ is a log affinoid perfectoid, and its associated affinoid perfectoid space is isomorphic to $\mathrm{Spf}(\widetilde{R}^{0}_{\infty})_{\eta}$.

    Let 
    \[
    (x_{i}^{s})^{\flat}\coloneqq (x_{i}^{s},x_{i}^{s/p},x_{i}^{s/p^{2}},\dots)\in (\widetilde{R}^{0}_{\infty})^{\flat}
    \]
    for $s\in \bZ[1/p]$, and let 
    \begin{align*}
    (y_{j}^{s})^{\flat}&\coloneqq (y_{j}^{s},y_{j}^{s/p},y_{j}^{s/p^{2}},\dots)\in (\widetilde{R}^{0}_{\infty})^{\flat} \\(z_{k}^{s})^{\flat}&\coloneqq (z_{k}^{s},z_{k}^{s/p},z_{k}^{s/p^{2}},\dots)\in (\widetilde{R}^{0}_{\infty})^{\flat} 
    \end{align*}
    for $s\in \bQ_{\geq 0}$. When $n\geq 1$, we set
    \[
    (\pi^{s})^{\flat}\coloneqq \prod_{k=1}^{n}(z_{k}^{s})^{\flat}
    \]
    for $s\in \bQ_{\geq 0}$.

\begin{notation}\label{convenient notation}
    When we would like to ease the notation, we write 
    \begin{align*}
    t_{i}&\coloneqq 
    \begin{cases}
        x_{i} \ \ \ (1\leq i\leq l) \\
        y_{i-l} \ \ \ (l+1\leq i\leq l+m) \\
        z_{i-l-m} \ \ \ (l+m+1\leq i\leq l+r) 
    \end{cases}  \\
    s&\coloneqq 
    \begin{cases}
        l+m \ \ \ (n=0) \\
        l+m+n-1=l+r-1 \ \ \ (n\geq 1).
    \end{cases} 
    \end{align*}
\end{notation}

\begin{dfn}[Small affine log formal schemes]
A $p$-adic affine log formal scheme $(\mathrm{Spf}(R),\cM_{R})$ over $\cO_{K}$ is called \emph{small affine} if $(\mathrm{Spf}(R),\cM_{R})$ admits a strict \'{e}tale map to $(\mathrm{Spf}(R^{0}),\bN^{r})^{a}$ for some $l,m,n$, called a \emph{framing}. 
\end{dfn}

Let $(\mathrm{Spf}(R),\cM_{R})$ be a small affine log formal scheme over $\cO_{K}$ and fix a framing  $(\mathrm{Spf}(R),\cM_{R})\to (\mathrm{Spf}(R^{0}),\bN^{r})^{a}$. Let $\alpha\colon \bN^{r}\to \cM_{R}$ denote the induced chart. We set
\[
R_{\infty,\alpha}\coloneqq R^{0}_{\infty}\widehat{\otimes}_{R^{0}} R,\ \ \  \widetilde{R}_{\infty,\alpha}\coloneqq \widetilde{R}^{0}_{\infty}\widehat{\otimes}_{R^{0}} R.
\]
Taking $p$-completed base change of the sequence $(1)$ along $R^{0}\to R$, we obtain a sequence of $p$-complete prelog rings
\begin{equation}
    (R,\bN^{r})\to (R_{\infty,\alpha},\bQ_{\geq 0}^{r})\to (\widetilde{R}_{\infty,\alpha},\bQ_{\geq 0}^{r}). \tag{3}
\end{equation}
The sequence of $p$-adic fs log formal schemes associated with the sequence $(3)$ is denoted by
\begin{equation}
    (\mathrm{Spf}(\widetilde{R}_{\infty,\alpha}),\cM_{\widetilde{R}_{\infty,\alpha}})\to (\mathrm{Spf}(R_{\infty,\alpha}),\cM_{R_{\infty,\alpha}})\to (\mathrm{Spf}(R),\cM_{R}), \tag{4}
\end{equation}
which is also obtained as the base change of the sequence $(2)$ along $(\mathrm{Spf}(R),\cM_{R})\to (\mathrm{Spf}(R^{0}),\bN^{r})^{a}$.

    Let $\widetilde{U}_{\infty,\alpha}\coloneqq \widetilde{U}^{0}_{\infty}\times_{(\mathrm{Spf}(R^{0}),\bN^{r})^{a}_{\eta}} (\mathrm{Spf}(R),\cM_{R})_{\eta}$. Then $\widetilde{U}_{\infty,\alpha}$ is a log affinoid perfectoid in $(\mathrm{Spf}(R),\cM_{R})_{\eta,\mathrm{prok\et}}$ with the associated perfectoid space isomorphic to $\mathrm{Spf}(\widetilde{R}_{\infty,\alpha})_{\eta}$.

    The following lemma is used afterward.

\begin{lem}\label{perfectoid cover}
    The natural map $(\mathrm{Spf}(\widetilde{R}_{\infty,\alpha}),\cM_{\widetilde{R}_{\infty,\alpha}})\to (\mathrm{Spf}(R_{\infty,\alpha}),\cM_{R_{\infty,\alpha}})$ is a strict quasi-syntomic cover.
\end{lem}

\begin{proof}
It suffices to prove that $R^{0}_{\infty}\to \widetilde{R}^{0}_{\infty}$ is a quasi-syntomic cover, where $R^{0}_{\infty}\coloneqq R^{0}\widehat{\otimes}_{\bZ_{p}\langle \bN^{r} \rangle} \bZ_{p}\langle \bQ_{\geq 0}^{r} \rangle$. When $n=0$, the assertion is easy. Hence, we treat the case that $n\geq 1$. The ring map
\begin{align*}
R^{0}_{\infty}=\cO_{K}\langle \{x_{i}^{\pm 1}\}, &\{y_{j}^{\bQ_{\geq 0}}\}, \{z_{k}^{\bQ_{\geq 0}}\}\rangle /(\prod_{k=1}^{n}z_{k}-\pi) \\
&\to R'^{0}_{\infty}\coloneqq \cO_{K}\langle \{x_{i}^{\pm 1/p^{\infty}}\}, \{y_{j}^{\bQ_{\geq 0}}\}, \{z_{k}^{\bQ_{\geq 0}}\}\rangle /(\prod_{k=1}^{n}z_{k}-\pi)    
\end{align*}
is a quasi-syntomic cover. Let $\cO_{K_{\infty}}$ denote the $p$-adic completion of 
\[
\bigcup_{s\geq 1} \cO_{K}[\pi^{1/s}](\subset \cO_{C}),
\]
and consider a map $\cO_{K_{\infty}}\to R'^{0}_{\infty}$ mapping $\pi^{1/s}$ to $\prod_{k=1}^{n}z_{k}^{1/s}$. Then $R'^{0}_{\infty}\widehat{\otimes}_{\cO_{K_{\infty}}} \cO_{C}$ is naturally isomorphic to $\widetilde{R}^{0}_{\infty}$, and so $R'^{0}_{\infty}\to \widetilde{R}^{0}_{\infty}$ is a quasi-syntomic cover. This proves the claim.
\end{proof}

\begin{dfn}[Semi-stable log formal schemes]\label{def of semi-stable log formal schemes}
A $p$-adic log formal scheme $(\fX,\cM_{\fX})$ over $\cO_{K}$ is called \emph{horizontally semi-stable} (resp. \emph{vertically semi-stable}) if $(\fX,\cM_{\fX})$ admits a framing with $n=0$ (resp. $n\geq 1$) \'{e}tale locally, and $(\fX,\cM_{\fX})$ is called \emph{semi-stable} if it admits a framing \'{e}tale locally. When $(\fX,\cM_{\fX})$ is horizontally semi-stable (resp. vertically semi-stable), we let $\cM_{\cO_{K}}$ denote the trivial (resp. standard) log structure on $\mathrm{Spf}(\cO_{K})$. Then $(\fX,\cM_{\fX})\to (\mathrm{Spf}(\cO_{K}),\cM_{\cO_{K}})$ is log smooth.
\end{dfn}

\begin{dfn}[Unramified models]\label{def of unr model}
    Let $(\fX,\cM_{\fX})$ be a horizontally semi-stable log formal scheme over $\cO_{K}$. A horizontally semi-stable log formal scheme $(\fX_{0},\cM_{\fX_{0}})$ over $W$ with $(\fX_{0},\cM_{\fX_{0}})\otimes_{W} \cO_{K}\cong (\fX,\cM_{\fX})$ is called an \emph{unramified model} of $(\fX,\cM_{\fX})$.

    Let $(\mathrm{Spf}(R),\cM_{R})$ be a small affine log formal scheme over $\cO_{K}$ with a fixed framing such that $n=0$. Consider a prelog ring $(R^{0}_{0},\bN^{r})$ defined as follows:
    \[
    R^{0}_{0}\coloneqq W\langle x_{1}^{\pm 1},\dots,x_{l}^{\pm 1},y_{1},\dots,y_{m}\rangle
    \]
    equipped with a prelog structure $\bN^{r}\to R_{0}^{0}$ given by $e_{i}\mapsto y_{i}$. Note that now we have $m=r$ by $n=0$. The framing $(\mathrm{Spf}(R),\cM_{R})\to (\mathrm{Spf}(R^{0}),\bN^{r})^{a}$ uniquely descends to a strict \'{e}tale map $(\mathrm{Spf}(R_{0}),\cM_{R_{0}})\to (\mathrm{Spf}(R_{0}^{0}),\bN^{r})^{a}$. Then $(\mathrm{Spf}(R_{0}),\cM_{R_{0}})$ is an unramified model of $(\mathrm{Spf}(R),\cM_{R})$. We call it the \emph{unramified model defined by the framing}. The ring map $R_{0}^{0}\to R_{0}^{0}$ given by $x_{i}\mapsto x_{i}^{p}, y_{j}\mapsto y_{j}^{p}$ and the monoid map $\times p\colon \bN^{r}\to \bN^{r}$ gives a Frobenius lift on the log formal scheme $(\mathrm{Spf}(R^{0}_{0}),\bN^{r})^{a}$, which uniquely lifts to a Frobenius lift on $(\mathrm{R_{0}},\cM_{R_{0}})$.
\end{dfn}

\section{\texorpdfstring{$F$}--isocrystals on log crystalline sites}\label{section isoc}

\subsection{Log crystalline sites}

In this subsection, we review fundamental properties of crystals on absolute log crystalline sites.

\begin{dfn}[Big absolute log crystalline sites]
Let $(X,\cM_{X})$ be a log scheme over $\bF_{p}$. Let $(X,\cM_{X})_{\mathrm{crys}}$ be the site defined as follows: an object consists of
\begin{itemize}
 \item a log scheme $(T,\cM_{T})$ on which $p$ is nilpotent and a PD-thickening $U\hookrightarrow T$ such that the $PD$-structure on $\mathrm{Ker}(\cO_{T}\twoheadrightarrow \cO_{U})$ is compatible with the natural PD-structure on $(p)\subset \bZ_{p}$, 
 \item a map of log schemes $f\colon (U,\cM_{U})\to (X,\cM_{X})$, where $\cM_{U}$ is the inverse image log structure of $\cM_{T}$.
\end{itemize}    
Morphisms are obvious ones. We simply write $(U,T,\cM_{T})$ or $(T,\cM_{T})$ for an object of $(X,\cM_{X})_{\mathrm{crys}}$ when there exists no confusion. When $T=\mathrm{Spec}(R)$ and $U=\mathrm{Spec}(R/I)$, we write $(R\twoheadrightarrow R/I,\cM_{R})$ for a log PD-thickening $(U,T,\cM_{T})$. We say that a map $(U_{1},T_{1},\cM_{T_{1}})\to (U_{2},T_{2},\cM_{T_{2}})$ in $(X,\cM_{X})_{\mathrm{crys}}$ is \emph{Cartesian} if the natural map $U_{1}\to U_{2}\times_{T_{2}} T_{1}$ is an isomorphism. A family of Cartesian maps $\{(U_{i},T'_{i},\cM_{T'_{i}})\to (U,T,\cM_{T})\}_{i}$ is covering if  $\{(T'_{i},\cM_{T'_{i}})\to (T,\cM_{T})\}_{i}$ is a strict fpqc covering. For a morphism $f\colon (Y,\cM_{Y})\to (X,\cM_{X})$, there exists a natural morphism of sites
\[
(X,\cM_{X})_{\mathrm{crys}}\to (Y,\cM_{Y})_{\mathrm{crys}}.
\]
Let $\cO_{\mathrm{crys}}$ denote the ring sheaf given by $(T,\cM_{T})\mapsto \Gamma(T,\cO_{T})$.

Let $(X,\cM_{X})^{\mathrm{str}}_{\mathrm{crys}}$ denote the full subcategory of $(X,\cM_{X})_{\mathrm{crys}}$ consisting of an object $(U,T,\cM_{T})$ such that the structure map $(U,\cM_{U})\to (X,\cM_{X})$ is strict.  
\end{dfn}

\begin{rem}
    In many references, the site $(X,\cM_{X})^{\mathrm{str}}_{\mathrm{crys}}$ is called an (absolute) log crystalline site.
\end{rem}

\begin{lem}\label{big log crys site to str log crys site}
    For an integral log scheme $(X,\cM_{X})$ over $\bF_{p}$, the inclusion functor 
    \[ (X,\cM_{X})^{\mathrm{str}}_{\mathrm{crys}}\hookrightarrow (X,\cM_{X})_{\mathrm{crys}}
    \]
    has a right adjoint functor
    \[
    (X,\cM_{X})_{\mathrm{crys}}\to (X,\cM_{X})^{\mathrm{str}}_{\mathrm{crys}} \ \ \ ((U,T,\cM_{T})\mapsto (U,T,\cN_{T})).
    \]
\end{lem}

\begin{proof}
    Let $(U,T,\cM_{T})\in (X,\cM_{X})^{\mathrm{str}}_{\mathrm{crys}}$ and $f\colon (U,\cM_{U})\to (X,\cM_{X})$ denote the structure morphism. We define a log structure $\cN_{T}$ on $T$ by 
    \[
    \cN_{T}\coloneqq\cM_{T}\times_{\cM_{U}} f^{*}\cM_{X},
    \]
    under the identification $T_{\et}\simeq U_{\et}$. Since $(X,\cM_{X})$ is integral, the surjection $\cM_{T}\twoheadrightarrow \cM_{U}$ is $1+\cI$-torsor, where $\cI\coloneqq \mathrm{Ker}(\cO_{T}\twoheadrightarrow \cO_{U})$. Therefore, $(U,f^{*}\cM_{X})\hookrightarrow (T,\cN_{T})$ is strict, and $(U,T,\cN_{T})$ belongs to $(X,\cM_{X})^{\mathrm{str}}_{\mathrm{crys}}$. Then the functor sending $(U,T,\cM_{T})$ to $(U,T,\cN_{T})$ is the desired one.
\end{proof}

For a morphism of integral log schemes $(Y,\cM_{Y})\to (X,\cM_{X})$, we have a functor
\[
(Y,\cM_{Y})_{\mathrm{crys}}^{\mathrm{str}}\hookrightarrow (Y,\cM_{Y})_{\mathrm{crys}}\to (X,\cM_{X})_{\mathrm{crys}}\to (X,\cM_{X})_{\mathrm{crys}}^{\mathrm{str}},
\]
where the last functor is one given in Lemma \ref{big log crys site to str log crys site}. This gives a morphism of site
\[
(X,\cM_{X})_{\mathrm{crys}}^{\mathrm{str}}\to (Y,\cM_{Y})_{\mathrm{crys}}^{\mathrm{str}}.
\]

\begin{dfn}[Crystalline crystals on log schemes]\label{def of F-isoc} \noindent

Let $\mathrm{Crys}((X,\cM_{X})_{\mathrm{crys}})$ be the category of crystals of quasi-coherent sheaves of finite type on the log crystalline site $(X,\cM_{X})_{\mathrm{crys}}$. An object of $\mathrm{Crys}((X,\cM_{X})_{\mathrm{crys}})$ is called a \emph{crystalline crystal} on $(X,\cM_{X})$. The Frobenius morphism $F\colon (X,\cM_{X})\to (X,\cM_{X})$ induces a functor
\[
F^{*}\colon \mathrm{Crys}((X,\cM_{X})_{\mathrm{crys}})\to \mathrm{Crys}((X,\cM_{X})_{\mathrm{crys}}).
\]
Let $\mathrm{Isoc}((X,\cM_{X})_{\mathrm{crys}})$ denote the isogeny category $ \mathrm{Crys}((X,\cM_{X})_{\mathrm{crys}})\otimes_{\bZ_{p}} \bQ_{p}$. An object $\cE$ of $\mathrm{Isoc}((X,\cM_{X})_{\mathrm{crys}})$ is called an \emph{isocrystal}. For $\cE^{+}\in \mathrm{Crys}((X,\cM_{X})_{\mathrm{crys}})$, the image of $\cE^{+}$ by the natural functor
\[
\mathrm{Crys}((X,\cM_{X})_{\mathrm{crys}})\to \mathrm{Isoc}((X,\cM_{X})_{\mathrm{crys}})
\]
is denoted by $\cE^{+}[1/p]$. For a $p$-adic log PD-thickening $(\mathrm{Spec}(A/I),\mathrm{Spf}(A),\cM_{A})\in (X,\cM_{X})$, there is an evaluation functor
\[
\mathrm{Isoc}((X,\cM_{X})_{\mathrm{crys}})\to \mathrm{Mod}(A[1/p]).
\]
The image of an isocrystal $\cE$ via this functor is simply denoted by $\cE_{A}$. An isocrystal $\cE$ is called \emph{locally free} if $\cE_{A}$ is a finite projective $A[1/p]$-module for every $(\mathrm{Spec}(A/I),\mathrm{Spf}(A),\cM_{A})$. The full subcategory of $\mathrm{Isoc}((X,\cM_{X})_{\mathrm{crys}})$ consisting of locally free objects is denoted by $\mathrm{Isoc}_{\mathrm{lf}}((X,\cM_{X})_{\mathrm{crys}})$.

Let $\mathrm{Isoc}^{\varphi}((X,\cM_{X})_{\mathrm{crys}})$ be the category of pairs $(\cE,\varphi_{\cE})$ consisting of an isocrystal $\cE$ on $(X,\cM_{X})$ and an isomorphism $\varphi_{\cE}\colon F^{*}\cE\isom \cE$ of isocrystals. Morphisms are ones in the category of isocrystals that are compatible with Frobenius isomorphisms. When no confusion occurs, we simply write $\cE$ for $(\cE,\varphi_{\cE})$. An object of $\mathrm{Isoc}^{\varphi}((X,\cM_{X})_{\mathrm{crys}})$ is called an \emph{$F$-isocrystal} on $(X,\cM_{X})$. An $F$-isocrystal is called \emph{locally free} if its underlying isocrystal is locally free, and the full subcategory of $\mathrm{Isoc}^{\varphi}((X,\cM_{X})_{\mathrm{crys}})$ consisting of locally free objects is denoted by $\mathrm{Isoc}^{\varphi}_{\mathrm{lf}}((X,\cM_{X})_{\mathrm{crys}})$.

When $(X,\cM_{X})$ is integral, we can also define similar categories by replacing $(X,\cM_{X})_{\mathrm{crys}}$ with a strict site $(X,\cM_{X})_{\mathrm{crys}}^{\mathrm{str}}$. For the functoriality of strict sites, a morphism $(X,\cM_{X})\to (Y,\cM_{Y})$ induces a functor
\[
(X,\cM_{X})_{\mathrm{crys}}^{\mathrm{str}}\hookrightarrow (X,\cM_{X})_{\mathrm{crys}}\to (Y,\cM_{Y})_{\mathrm{crys}}\to (Y,\cM_{Y})_{\mathrm{crys}}^{\mathrm{str}},
\]
where the last functor is one given in Lemma \ref{big log crys site to str log crys site}. This allows us to define the functor
\[
F^{*}\colon \mathrm{Crys}((X,\cM_{X})^{\mathrm{str}}_{\mathrm{crys}})\to \mathrm{Crys}((X,\cM_{X})^{\mathrm{str}}_{\mathrm{crys}})
\]
and the category $\mathrm{Isoc}^{\varphi}((X,\cM_{X})_{\mathrm{crys}}^{\mathrm{str}})$.
\end{dfn}

\begin{rem}\label{crystalline crystal as lim}
    For $\star\in \{\mathrm{str},\emptyset \}$, we have canonical bi-exact equivalences
    \[
    \mathrm{Crys}((X,\cM_{X})^{\star}_{\mathrm{crys}})\simeq \varprojlim_{(T,\cM_{T})\in (X,\cM_{X})^{\star}_{\mathrm{crys}}} \mathrm{FQCoh}(T).
    \]
    Here, $\mathrm{FQCoh}(T)$ is the category of quasi-coherent sheaves of finite type on $T$.
\end{rem}

\begin{rem}\label{cat of crys crys is unchanged}
    Let $(X,\cM_{X})$ be an integral log scheme over $\bF_{p}$. For a crystalline crystal $\cE$ on $(X,\cM_{X})$ and a log PD-thickening $(U,T,\cM_{T})\in (X,\cM_{X})_{\mathrm{crys}}$, the crystal property gives an isomorphism
    \[
    \cE_{(U,T,\cM_{T})}\cong \cE_{(U,T,\cN_{T})},
    \]
    where $(U,T,\cN_{T})$ is the image of $(U,T,\cM_{T})$ via the functor in Lemma \ref{big log crys site to str log crys site}. Therefore, the restriction functors
    \begin{align*}
        \mathrm{Crys}((X,\cM_{X})_{\mathrm{crys}})&\to \mathrm{Crys}((X,\cM_{X})^{\mathrm{str}}_{\mathrm{crys}}) \\
        \mathrm{Isoc}^{\star}((X, \cM_{X})_{\mathrm{crys}})&\to \mathrm{Isoc}^{\star}((X, \cM_{X})^{\mathrm{str}}_{\mathrm{crys}})
    \end{align*}
    give bi-exact equivalences. 
\end{rem}

\begin{lem}[Dwork's trick]\label{dwork trick}
    Let $\iota\colon (X,\cM_{X})\hookrightarrow (X',\cM_{X'})$ be a strict nilpotent immersion of log schemes over $\bF_{p}$. 
    \begin{enumerate}
    \item The restriction functor
    \[
    \iota^{*}\colon \mathrm{Isoc}^{\varphi}((X',\cM_{X'})_{\mathrm{crys}})\to \mathrm{Isoc}^{\varphi}((X,\cM_{X})_{\mathrm{crys}})
    \]
    is an equivalence.
    \item Let $\cE$ be an $F$-isocrystal on $(X',\cM_{X'})$ and $(A\twoheadrightarrow A/I,\cM_{A})$ be a $p$-adic log PD-thickening in $(X',\cM_{X'})_{\mathrm{crys}}$. Suppose that $\iota$ has a retraction $r\colon (X',\cM_{X'})\to (X,\cM_{X})$. Let $(A\twoheadrightarrow A/I,\cM_{A})^{r}\coloneqq (ir)^{-1}(A\twoheadrightarrow A/I,\cM_{A})$.
    Then there exists a natural isomorphism
    \[
    \cE_{(A\twoheadrightarrow A/I,\cM_{A})^{r}}\cong \cE_{(A\twoheadrightarrow A/I,\cM_{A})}.
    \]
    \end{enumerate}
\end{lem}

\begin{proof}
(1): See \cite[Remark B.23]{dlms24}.

(2): Take an integer $n\geq 1$ such that the iterated Frobenius map $F^{n}\colon (X',\cM_{X'})\to (X',\cM_{X'})$ factors thorough $(X,\cM_{X})\hookrightarrow (X',\cM_{X'})$. Then, by $irF^{n}=F^{n}$, we have isomorphisms
    \[
    \cE_{(A\twoheadrightarrow A/I,\cM_{A})}\cong \cE_{(F^{n})^{-1}(A\twoheadrightarrow A/I,\cM_{A})}= \cE_{(F^{n})^{-1}(A\twoheadrightarrow A/I,\cM_{A})^{r}}\cong \cE_{(A\twoheadrightarrow A/I,\cM_{A})^{r}},
    \]
    where the first and third isomorphisms are induced by the Frobenius structure on $\cE$. We can check directly that the resulting isomorphism $\cE_{(A\twoheadrightarrow A/I,\cM_{A})}\cong \cE_{(A\twoheadrightarrow A/I,\cM_{A})^{r}}$ is independent of the choice of $n$.
\end{proof}

\begin{rem}\label{dwork trick and unr model}
    Let $(\fX,\cM_{\fX})$ be a horizontally semi-stable log formal scheme over $W$. We set $(X,\cM_{X})\coloneqq (\fX,\cM_{\fX})\otimes_{W} k$, $(X',\cM_{X'})\coloneqq (\fX,\cM_{\fX})\otimes_{W} \cO_{K}/p$. Then the natural map $k\to \cO_{K}/p$ induces a retraction $r\colon (X',\cM_{X'})\to (X,\cM_{X})$ of the natural inclusion $\iota\colon (X,\cM_{X})\to (X',\cM_{X'})$, and so the above lemma can be applied.
\end{rem}

The following notion plays an important role when we clarify the relationship between isocrystals and modules with connection.

\begin{dfn}[cf.~{\cite[p.2117]{es18}}]\label{def of torsion free crystals}
Let $\cE$ be an object of $\mathrm{Crys}((X,\cM_{X})_{\mathrm{crys}})$. We say that $\cE$ is \emph{$p$-torsion free} if, \'{e}tale locally on $X$, there exist a $p$-adic log PD-thickening $(\cD,\cM_{\cD})\in (X,\cM_{X})_{\mathrm{crys}}^{\mathrm{str}}$ satisfying the following conditions:
\begin{itemize}
    \item $\cD$ is $\bZ_{p}$-flat ;
    \item $(\cD,\cM_{\cD})$ covers the final object of the topos associated with $(X,\cM_{X})_{\mathrm{crys}}$;
    \item there exists an $n+1$--fold self product $(\cD^{(n)},\cM_{\cD^{(n)}})$ of $(\cD,\cM_{\cD})$ in $(X,\cM_{X})_{\mathrm{crys}}$ for each $n\geq 0$;
    \item $\cE_{\cD^{(n)}}$ is $p$-torsion free for each $n\geq 0$.
\end{itemize}
Let $\mathrm{Crys}_{\mathrm{tf}}((X,\cM_{X})_{\mathrm{crys}})$ be the full subcategory of $\mathrm{Crys}((X,\cM_{X})_{\mathrm{crys}})$ consisting of $p$-torsion free objects.
Let $\mathrm{Isoc}_{\mathrm{tf}}((X,\cM_{X})_{\mathrm{crys}})$ denote the essential image of the natural functor
\[
\mathrm{Crys}_{\mathrm{tf}}((X,\cM_{X})_{\mathrm{crys}})\hookrightarrow \mathrm{Crys}((X,\cM_{X})_{\mathrm{crys}})\to \mathrm{Isoc}((X,\cM_{X})_{\mathrm{crys}}).
\]
We say that an isocrystal $\cE$ is \emph{$p$-torsion free} if $\cE$ belongs to $\mathrm{Isoc}_{\mathrm{tf}}((X,\cM_{X})_{\mathrm{crys}})$ as well. Let $\mathrm{Isoc}_{\mathrm{lftf}}((X,\cM_{X})_{\mathrm{crys}})$ denote the full subcategory of $\mathrm{Isoc}((X,\cM_{X})_{\mathrm{crys}})$ consisting of locally free and $p$-torsion free objects and $\mathrm{Isoc}_{\mathrm{lftf}}^{\varphi}((X,\cM_{X})_{\mathrm{crys}})$ denote the category of $F$-isocrystals whose underlying isocrystals belong to $\mathrm{Isoc}_{\mathrm{lftf}}((X,\cM_{X})_{\mathrm{crys}})$.
\end{dfn}

\begin{lem}\label{autmatically tf}
Suppose that we can take $(\cD,\cM_{\cD})$ in Definition \ref{def of torsion free crystals} such that $\cD$ is noetherian. Then every isocrystal on $(X,\cM_{X})$ is $p$-torsion free. 
\end{lem}

\begin{proof}
Let $\cE$ be an isocrystal on $(X,\cM_{X})$, and take $\cE^{+}\in \mathrm{Crys}((X,\cM_{X})_{\mathrm{crys}})$ with $\cE\cong \cE^{+}[1/p]$. Then $\cE^{+}$ corresponds to a quasi-coherent $\cO_{\cD}$-module $\cM\coloneqq \cE^{+}_{A}$ equipped with HPD-stratification. Let $\overline{\cM}$ be the maximal $p$-torsion free quotient of $\cM$. Since $\cD^{(1)}\to \cD$ is flat, $\overline{\cM}$ has the induced HPD-stratification (see Lemma \ref{lem for tf quot and completion} below). The $\cO_{\cD}$-module $\overline{\cM}$ equipped with HPD-stratification defines an object $\overline{\cE}^{+}\in \mathrm{Crys}((X,\cM_{X})_{\mathrm{crys}})$. Since $A$ is noetherian, $M$ and $\overline{M}$ is isomorphic in the isogeny category $\mathrm{Mod}(\cO_{\cD})\otimes_{\bZ_{p}} \bQ_{p}$. Therefore, we get $\cE\cong \cE^{+}[1/p]\cong \overline{\cE}^{+}[1/p]$. This proves that $\cE$ is $p$-torsion free.
\end{proof}

\begin{rem}\label{log smooth imlies tf}
    By the above proposition, when $(X,\cM_{X})$ is log smooth over $k$, every isocrystal on $(X,\cM_{X})$ is $p$-torsion free. Indeed, a log smooth formal lift $(\cD,\cM_{\cD})$ of $(X,\cM_{X})$ is enough.
\end{rem}

\subsection{Isocrystals and connections}

In this subsection, we recall the relation between crystals and modules with integrable connections studied in \cite[Section 6]{kat89}.

Let $(\fS,\cM_{\fS})$ be a $p$-adic fine noetherian $\bZ_{p}$-flat log formal scheme with a PD-ideal $\cI$ containing $p$ that is compatible with the natural PD-structure on $(p)\subset \bZ_{p}$. Let $(S,\cM_{S})$ be the strict closed subscheme of $(\fS,\cM_{\fS})$ defined by $\cI$. Consider a fine log scheme $(X,\cM_{X})$ over $(S,\cM_{S})$. Suppose that there is a strict closed immersion $(X,\cM_{X})\hookrightarrow (\fY,\cM_{\fY})$, where $(\fY,\cM_{\fY})$ is a $p$-adic fine log formal scheme that is log smooth and integral over $(\fS,\cM_{\fS})$. Let $(\cD,\cM_{\cD})$ be the log PD-envelope of $(X,\cM_{X})\hookrightarrow (\fY,\cM_{\fY})$ and $(\cD^{(n)},\cM_{\cD^{(n)}})$ be the log PD-envelope of the diagonal map 
\[
(X,\cM_{X})\hookrightarrow \underbrace{(\fY,\cM_{\fY})\times_{(\fS,\cM_{\fS})}\dots \times_{(\fS,\cM_{\fS})} (\fY,\cM_{\fY})}_{n+1 \ \text{copies of} \ (\fY,\cM_{\fY})}
\]
for $n\geq 1$.

\begin{construction}\label{stratification to connection}
Let $\cJ$ denote the kernel of the surjection $\cO_{\cD^{(1)}}\twoheadrightarrow \cO_{\cD}$ induced from the diagonal map $(\fY,\cM_{\fY})\hookrightarrow (\fY,\cM_{\fY})\times_{(\fS,\cM_{\fS})} (\fY,\cM_{\fY})$. Suppose that we are given $t_{1},\dots,t_{s}\in \Gamma(\fY,\cM_{\fY})$ such that $\mathrm{dlog}(t_{1}),\dots,\mathrm{dlog}(t_{s})$ forms a basis of $\Omega^{1}_{(\fY,\cM_{\fY})/(\fS,\cM_{\fS})}$ (this is the case \'{e}tale locally on $\fY$). By \cite[Proposition 6.5]{kat89}, there is an isomorphism of $\cO_{\cD}$-algebras
\[
\cO_{\cD}\{\xi_{1},\dots,\xi_{s}\}\isom \cO_{\cD^{(1)}}
\]
given by $\xi_{i}\mapsto p_{1}^{*}(t_{i})p_{2}^{*}(t_{i})^{-1}-1$, where $\cO_{\cD^{(1)}}$ is viewed as an $\cO_{\cD}$-algebra via $p_{2}$ and $\cO_{\cD}\{\xi_{1},\dots,\xi_{s}\}$ is the $p$-completed PD-polynomial ring with variables $\xi_{1},\dots,\xi_{s}$ over $\cO_{\cD}$.
 
Let $\cF$ be a quasi-coherent $\cO_{\cD}$-module with a HPD-stratification $p_{1}^{*}\cF\isom p_{2}^{*}\cF$. Then, by cutting off the HPD-stratification on $\cF$ modulo $\cJ^{[\geq 2]}$, we obtain an integrable connection 
\[
\nabla\colon \cF\to \cF\otimes_{\cO_{\fY}} \Omega^{1}_{(\fY,\cM_{\fY})/(\fS,\cM_{\fS})}.
\]
 We define a $\cO_{\fS}$-linear map $\nabla_{t_{i}\frac{\partial}{\partial t_{i}}}\colon \cF\to \cF$ as follows:
\[
\nabla(x)=\sum_{i=1}^{s} \nabla_{t_{i}\frac{\partial}{\partial t_{i}}}(x)\otimes \mathrm{dlog}(t_{i}) \ \ \ (x\in \cF).
\]
More generally, let $\nabla_{t_{i}^{n}\frac{\partial^{n}}{\partial t_{i}^{n}}}\colon \cF\to \cF$ be the $\cO_{\fS}$-linear map defined by
\[
\nabla_{t_{i}^{n}\frac{\partial^{n}}{\partial t_{i}^{n}}}\coloneqq \prod_{j=0}^{n-1}(\nabla_{t_{i}\frac{\partial}{\partial t_{i}}}-j)
\]
for $1\leq i\leq s$ and $n\geq 1$.  Then, by the proof of \cite[Theorem 6.2]{kat89}, the $\cO_{\fS}$-linear map
\[
\eta\colon \cF\to p_{1}^{*}\cF\isom p_{2}^{*}\cF
\]
defined from the HPD-stratification is written as follows:
\[
\eta(x)=\sum_{\underline{n}=(n_{1},\dots,n_{s})\in \bN^{s}} \nabla_{\underline{n}}(x)\otimes \xi^{\underline{n}} \ \ \ (x\in \cF),
\]
where $\displaystyle \nabla_{\underline{n}}\coloneqq \prod_{1\leq i\leq s} \nabla_{t_{i}^{n_{i}}\frac{\partial^{n_{i}}}{\partial t_{i}^{n_{i}}}}$, and $\displaystyle \xi^{\underline{n}}\coloneqq \prod_{1\leq i\leq s}\xi_{i}^{[n_{i}]}$.
\end{construction}

This procedure gives the following equivalence.

\begin{prop}[cf.~{\cite[Th\'{e}or\`{e}me 4.2.11]{ber74}},{\cite[Theorem 6.2]{kat89}}]\label{eq of HPDstrat and conn}
    There is a natural equivalence between the following categories:
    \begin{itemize}
        \item the category of quasi-coherent $\cO_{\cD}$-modules with HPD-stratification;
        \item the category of quasi-coherent $\cO_{\cD}$-modules with topologically quasi-nilpotent integrable connection.
    \end{itemize}
    Here, an integrable connection  $\cF\to \cF\otimes_{\cO_{\fY}} \Omega^{1}_{(\fY,\cM_{\fY})/(\fS,\cM_{\fS})}$ is called \emph{topologically quasi-nilpotent} if, for some (or equivalently any) local sections $t_{1},\dots,t_{s}\in \Gamma(\fY,\cM_{\fY})$ as in Construction \ref{stratification to connection}, $x\in \cF$, and $1\leq i\leq s$, $\nabla_{t_{i}^{n}\frac{\partial^{n}}{\partial t_{i}^{n}}}(x)$ converges to $0$ when $n\to \infty$.
\end{prop}

\begin{dfn}
    We define the following categories.
    \begin{itemize}
        \item Let $\mathrm{StrFQCoh}(\cD)$ be the category of quasi-coherent $\cO_{\cD}$-modules of finite type equipped with HPD-stratification.
        \item Let $\mathrm{FQCoh}^{\nabla}(\cD)$ be the category of quasi-coherent $\cO_{\cD}$-modules $\cM$ of finite type equipped with topologically quasi-nilpotent integrable connection.
    \end{itemize}
\end{dfn}

Let $\cE\in \mathrm{Crys}((X,\cM_{X})_{\mathrm{crys}})$. By the crystal property of $\cE$, there are isomorphisms
\[
p_{1}^{*}\cE_{\cD}\cong \cE_{\cD^{(1)}}\cong p_{2}^{*}\cE_{\cD},
\]
where $p_{i}\colon (\cD^{(1)},\cM_{\cD^{(1)}})\to (\cD,\cM_{\cD})$ is the natural projection map for $i=1,2$. This isomorphism satisfies the cocycle condition on $\cD^{(2)}$ and gives a HPD-stratification on $\cE_{\cD}$. This construction and the equivalence in Proposition \ref{eq of HPDstrat and conn} gives the following proposition.

\begin{prop}\label{functor from crys to conn}
    There is a natural functor 
    \[
    \mathrm{Crys}((X,\cM_{X})_{\mathrm{crys}})\to \mathrm{StrFQCoh}(\cD) (\simeq \mathrm{FQCoh}^{\nabla}(\cD)).
    \]
\end{prop}

\begin{cor}\label{eq of crys crys and conn over W}
    Suppose that $(\fS,\cM_{\fS})=(\mathrm{Spf}(W),\cM_{W})$ is the formal spectrum of $W$ equipped with the trivial log structure and $\cI=(p)$. Then there are natural equivalences
    \[
    \mathrm{Crys}((X,\cM_{X})_{\mathrm{crys}})\simeq \mathrm{StrFQCoh}(\cD)\simeq \mathrm{FQCoh}(\cD).
    \]
\end{cor}

\begin{proof}
    Since $(\fY,\cM_{\fY})$ is log smooth over $W$, the $p$-adic log PD-thickening $(\cD,\cM_{\cD})$ covers the final object of the topos associated with $(X,\cM_{X})_{\mathrm{crys}}$. Therefore, the functor in Proposition \ref{functor from crys to conn} gives an equivalence.
\end{proof}

Next, we turn to studying isocrystals. We start with some lemmas on completions of modules.

\begin{lem}(cf.~{\cite[Lemma 3.24(4)]{dlms23}})\label{lemma for complete tensor and inverting p}
    Let $A\to B$ be a ring map from a  $p$-complete and $p$-torsion free ring $A$ to a $p$-complete ring $B$. Let $M$ be a finitely generated $A$-module. Suppose that $M[1/p]$ is finite projective over $A[1/p]$. Then there is a natural isomorphism
    \[
    M[1/p]\otimes_{A} B\isom (M\otimes_{A} B)^{\wedge}[1/p].
    \]
\end{lem}

\begin{proof}
    Since $M\otimes_{A} B$ is a finitely generated $B$-module, the natural map 
    \[
    M\otimes_{A} B\to (M\otimes_{A} B)^{\wedge}
    \]
    is surjective. It suffices to show that the induced surjection $M[1/p]\otimes_{A} B\to (M\otimes_{A} B)^{\wedge}[1/p]$ is injective.

    Since $M[1/p]$ is finite projective over $A[1/p]$, we can take a finite free $A$-module $N$ and an injection $M\to N$ such that $M[1/p]\to N[1/p]$ admits a retract. Then the map $M\otimes_{A} B\to N\otimes_{A} B$ factors as
    \[
    M\otimes_{A} B\to (M\otimes_{A} B)^{\wedge}\to N\otimes_{A} B.
    \]
    Since $M[1/p]\to N[1/p]$ admits a retract, $M[1/p]\otimes_{A} B\to N[1/p]\otimes_{A} B$ is also injective. Therefore, $M[1/p]\otimes_{A} B\to (M\otimes_{A} B)^{\wedge}[1/p]$ is also injective.
\end{proof}

\begin{lem}\label{completeness is automatically}
    Let $A$ be a $p$-complete and $p$-torsion free ring. Let $M$ be a finitely generated $p$-torsion free $A$-module. Suppose that $M[1/p]$ is finite projective over $A[1/p]$. Then $M$ is $p$-complete. 
\end{lem}

\begin{proof}
    Since $M$ is finitely generated, the natural map $M\to M^{\wedge}$ is surjective. To prove that this map is injective, it suffices to show that $M[1/p]\to M^{\wedge}[1/p]$ is an isomorphism due to the $p$-torsion freeness of $M$. The natural map 
    \[
    M^{\wedge}[1/p]=(\varprojlim_{n\geq 1} M/p^{n}M)[1/p]\to \varprojlim_{n\geq 1} M[1/p]/p^{n}M=(M[1/p])^{\wedge} 
    \]
    is an isomorphism. Since the topology on $M[1/p]$ is independent of the choice of a lattice, the assumption that $M[1/p]$ is finite projective over $A[1/p]$ implies that we have $(M[1/p])^{\wedge}\cong M[1/p]$. Therefore, we conclude that $M[1/p]\cong M^{\wedge}[1/p]$.
\end{proof}

\begin{dfn}\label{def of rat strvect}
We define the following categories:
    \begin{itemize}
        \item Let $\mathrm{StrVect}(\cD)$ be the category of  $p$-torsion free quasi-coherent $\cO_{\cD}$-module $\cM^{+}$ such that $\cM^{+}[1/p]$ is locally projective $\cO_{\cD}[1/p]$-module equipped with a $\cO_{\cD^{(1)}}$-linear isomorphism
        \[ 
        \epsilon\colon p_{1}^{*}\cM^{+}\isom p_{2}^{*}\cM^{+};
        \]
        \item Let $\mathrm{StrVect}(\cD[1/p])$ be the category of locally finite projective $\cO_{\cD}[1/p]$-modules $\cM$ with a $\cO_{\cD^{(1)}}[1/p]$-linear isomorphism 
        \[
        \epsilon\colon p_{1}^{*}\cM\isom p_{2}^{*}\cM
        \]
        satisfying the cocycle condition over $\cD^{(2)}$; 
        \item Let $\mathrm{Vect}^{\nabla}(\cD[1/p])$ be the category of locally finite projective $\cO_{\cD}[1/p]$-modules $\cM$ equipped with a topologically quasi-nilpotent integrable connection 
        \[
        \nabla_{\cM}\colon \cM\to \cM\otimes_{\cO_{\fY}} \Omega^{1}_{(\fY,\cM_{\fY})/(\fS,\cM_{\fS})}.
        \]
    \end{itemize}
\end{dfn}

\begin{lem}\label{eq of rational HPDstr and conn}
There is a natural equivalence
    \[
    \mathrm{StrVect}(\cD[1/p])\isom \mathrm{Vect}^{\nabla}(\cD[1/p]).
    \]
\end{lem}

\begin{proof}
    Both of the source and the target satisfies \'{e}tale descent by \cite[Proposition 2.6]{dri22}. We may assume that every formal scheme is affine and there are sections $t_{1},\dots,t_{s}\in \Gamma(\fY,\cM_{\fY})$ such that $\mathrm{dlog}(t_{1}),\dots,\mathrm{dlog}(t_{s})$ forms a basis of $\Omega^{1}_{(\fY,\cM_{\fY})/(\fS,\cM_{\fS})}$. Let $\fY=\mathrm{Spf}(B)$, $\cD=\mathrm{Spf}(D)$, and $\cD^{(n)}=\mathrm{Spf}(D^{(n)})$. 
    
    An object of the source category is described as a pair $(M,\epsilon)$ consisting of a finite projective $D[1/p]$-module $M$ and an $D^{(1)}[1/p]$-linear isomorphism
    \[
    \epsilon\colon M\otimes_{D,p_{1}} D^{(1)}\isom M\otimes_{D,p_{2}} D^{(1)}
    \]
    satisfying the cocycle condition over $D^{(2)}$. We have an exact sequence of $D$-modules
    \[
    0\to D\otimes_{B} \Omega^{1}_{(\fY,\cM_{\fY})/(\fS,\cM_{\fS})}\to D^{(1)}/J^{[\geq 2]}\to D\to 0.
    \]
    Hence, the map $M\to M\otimes_{D} D^{1}/J^{[\geq 2]}$ given by $x\mapsto \epsilon(x\otimes 1)-x\otimes 1$ defines a connection 
    \[
    \nabla_{M}\colon M\to M\otimes_{B} \Omega^{1}_{(\fY,\cM_{\fY})/(\fS,\cM_{\fS})}.
    \]
    We define a map $\nabla_{t_{i}\frac{\partial}{\partial t_{i}}}\colon M\to M$ by $\nabla_{M}(x)=\sum_{i=1}^{s} \nabla_{t_{i}\frac{\partial}{\partial t_{i}}}(x)\otimes \mathrm{dlog}(t_{i})$ for $1\leq i\leq s$.

    Since $M$ is finite projective over $D[1/p]$, there is an isomorphism
    \[
    M\otimes_{D,p_{2}} D^{(1)}\cong \widehat{\bigoplus_{\underline{n}=(n_{1},\dots.n_{s})\in \bN^{s}}} M\cdot \xi^{[\underline{n}]}
    \]
    (see Construction \ref{stratification to connection}). Then, in the same way as \cite[Th\'{e}or\`{e}me 4.2.11]{ber74} or \cite[Theorem 6.2]{kat89}, we conclude that the map
    \[
    M\to M\otimes_{D,p_{1}} D^{(1)}\stackrel{\epsilon}{\isom} M\otimes_{D,p_{2}} D^{(1)} \cong \widehat{\bigoplus_{\underline{n}=(n_{1},\dots.n_{s})\in \bN^{s}}} M\cdot \xi^{[\underline{n}]}
    \]
    is described as
    \[
    x\mapsto \sum_{\underline{n}=(n_{1},\dots,n_{s})\in \bN^{s}} \nabla_{\underline{n}}(x)\otimes \xi^{[\underline{n}]}
    \]
    under the same notation as Construction \ref{stratification to connection}.
    This description allows us to check that $(M,\nabla_{M})$ belongs to $\mathrm{Vect}^{\nabla}(\cD[1/p])$ and the functor $\mathrm{StrVect}(\cD[1/p])\to \mathrm{Vect}^{\nabla}(\cD[1/p])$ given by $(M,\epsilon)\mapsto (M,\nabla_{M})$ is an equivalence 
\end{proof}

\begin{lem}\label{rat strvect admit lattice}
    There is a natural equivalence
    \[
    \mathrm{StrVect}(\cD)\otimes_{\bZ_{p}} \bQ_{p}\isom \mathrm{StrVect}(\cD[1/p]),
    \]
    where the source category is the isogeny category.
\end{lem}

\begin{proof}
Both of the source and the target satisfies \'{e}tale descent by \cite[Proposition 2.6 and Corollary 2.5]{dri22}. We may assume that every formal scheme is affine and there are sections $t_{1},\dots,t_{s}\in \Gamma(\fY,\cM_{\fY})$ such that $\mathrm{dlog}(t_{1}),\dots,\mathrm{dlog}(t_{s})$ forms a basis of $\Omega^{1}_{(\fY,\cM_{\fY})/(\fS,\cM_{\fS})}$. Let $\fY=\mathrm{Spf}(B)$, $\cD=\mathrm{Spf}(D)$, and $\cD^{(n)}=\mathrm{Spf}(D^{(n)})$.

By Lemma \ref{lemma for complete tensor and inverting p}, inverting $p$ gives the functor in the statement by Lemma \ref{lemma for complete tensor and inverting p}. Note that, for a ring $A$ and finitely generated $p$-torsion free $A$-modules $M,N$, the natural map
    \[
    \mathrm{Hom}_{A}(M,N)[1/p]\to \mathrm{Hom}_{A[1/p]}(M[1/p],N[1/p])
    \]
    is an isomorphism. Since $p_{i}\colon D\to D^{(1)}$ is $p$-adically flat, $M^{+}\widehat{\otimes}_{D,p_{i}} D^{(1)}$ is $p$-torsion free for each object $(M^{+},\epsilon)\in \mathrm{StrVect}(D)$. Hence, the functor in the statement is fully faithful. 

    We use the notation in Construction \ref{stratification to connection} and the proof of Proposition \ref{eq of rational HPDstr and conn}. Let $(M,\epsilon)\in \mathrm{StrVect}(D[1/p])$. To show the essential surjectivity, it is enough to prove that there exists a finitely generated $D$-submodule $M^{+}$ of $M$ that is preserved by $\nabla_{t_{i}\frac{\partial}{\partial t_{i}}}$ for each $1\leq i\leq s$. Fix a finitely generated $D$-submodule $M^{+}_{0}$ of $M$. By the convergence condition, there exists an integer $N\geq 1$ such that $\nabla_{\underline{n}}(M^{+}_{0})$ is contained in $M^{+}_{0}$ for every $\underline{n}\in \bN^{s}\backslash [0,N]^{s}$. We put 
    \[
    M^{+}\coloneqq \sum_{\underline{n}\in [0,N]^{s}}\nabla_{\underline{n}}(M^{+}_{0}).
    \]
    The formula $D\cdot \nabla_{\underline{n}}(M^{+}_{0})\subset M^{+}_{0}+\nabla_{\underline{n}}(M^{+}_{0})$ for each $\underline{n}\in \bN^{s}$ coming from the Leibniz rule implies that $M^{+}$ is a finitely generated $D$-submodule. The formula 
    \[
    \nabla_{t_{i}\frac{\partial}{\partial t_{i}}}\circ \nabla_{\underline{n}}=\nabla_{(n_{1},\dots,n_{i}+1,\dots,n_{s})}+n_{i}\nabla_{\underline{n}} \ \ \ \ (\underline{n}\in \bN^{s}, 1\leq i\leq s)
    \]
    implies that $M^{+}$ is preserved by $\nabla_{t_{i}\frac{\partial}{\partial t_{i}}}$ for any $1\leq i\leq s$. Then the essential surjectivity follows.
\end{proof}

\begin{prop}\label{functor from isoc to rat conn}
There is a natural functor
    \[
    \mathrm{Isoc}_{\mathrm{lf}}((X,\cM_{X})_{\mathrm{crys}})\to \mathrm{StrVect}(\cD[1/p]) (\simeq \mathrm{Vect}^{\nabla}(\cD[1/p])).
    \]
\end{prop}

\begin{proof}
    This can be defined in the same way as Proposition \ref{functor from crys to conn}.
\end{proof}

\begin{prop}\label{eq isocrystal and conn}
Suppose that $(\fS,\cM_{\fS})$ is $\mathrm{Spf}(W)$ equipped with the trivial log structure and $\cI=(p)$. Then there are natural equivalences
    \[
    \mathrm{Isoc}_{\mathrm{lftf}}((X,\cM_{X})_{\mathrm{crys}})\simeq \mathrm{StrVect}(\cD[1/p])\simeq \mathrm{Vect}^{\nabla}(\cD[1/p]).
    \]
In particular, when $(X,\cM_{X})$ is log smooth over $k$, there are natural equivalences
    \[
    \mathrm{Isoc}_{\mathrm{lf}}((X,\cM_{X})_{\mathrm{crys}})\simeq \mathrm{StrVect}(\cD[1/p])\simeq \mathrm{Vect}^{\nabla}(\cD[1/p]).
    \]
\end{prop}

\begin{proof}
Since the latter assertion follows from the former and Remark \ref{log smooth imlies tf}, it is enough to prove the former one. The restriction of the functor in Proposition \ref{functor from isoc to rat conn} gives a functor
\begin{align}
    \mathrm{ev}_{\cD}\colon \mathrm{Isoc}_{\mathrm{lftf}}((X,\cM_{X})_{\mathrm{crys}})\to \mathrm{StrVect}(\cD[1/p]).
\end{align}
We consider the inverse direction. Corollary \ref{eq of crys crys and conn over W} gives a fully faithful functor
\[
\mathrm{StrVect}(\cD)\to \mathrm{Crys}_{\mathrm{tf}}((X,\cM_{X})_{\mathrm{crys}}).
\]
By passing to isogeny categories, we obtain a fully faithful functor
\[
\mathrm{StrVect}(\cD)\otimes_{\bZ_{p}} \bQ_{p}\to \mathrm{Isoc}_{\mathrm{lftf}}((X,\cM_{X})_{\mathrm{crys}}).
\]
Taking the composition with the equivalence in Lemma \ref{rat strvect admit lattice} gives a fully faithful functor
\begin{align}
    \mathrm{StrVect}(\cD[1/p])\to \mathrm{Isoc}_{\mathrm{lftf}}((X,\cM_{X})_{\mathrm{crys}}).
\end{align}

The composition
\[
\mathrm{StrVect}(\cD[1/p])\to \mathrm{Isoc}_{\mathrm{lftf}}((X,\cM_{X})_{\mathrm{crys}})\to \mathrm{StrVect}(\cD[1/p])
\]
is naturally isomorphic to the identity functor by definition. Hence, in order to show that functors (1) and (2) are inverse to each other, it suffices to verify that the functor (1) is fully faithful. Let $\cE_{1}^{+},\cE_{2}^{+}\in \mathrm{Crys}_{\mathrm{tf}}((X,\cM_{X})_{\mathrm{crys}})$ with $\cE_{1}^{+}[1/p]$ and $\cE_{2}^{+}[1/p]$ being locally free. Take a $p$-adic log PD-thickening $(\cD_{2},\cM_{\cD_{2}})\in (X,\cM_{X})_{\mathrm{crys}}$ satisfying conditions in Definition \ref{def of torsion free crystals} for a torsion free crystalline crystal $\cE_{2}^{+}$. We can define a category $\mathrm{StrVect}(\cD_{2}[1/p])$ and a functor $\mathrm{ev}_{\cD_{2}}$ in a similar way to $\mathrm{StrVect}(\cD[1/p])$ and $\mathrm{ev}_{\cD}$, and there is the following commutative diagram:
\[
\begin{tikzcd}
    \mathrm{Isoc}_{\mathrm{lftf}}((X,\cM_{X})_{\mathrm{crys}}) \ar[r,"\mathrm{ev}_{\cD}"] \ar[rd,"\mathrm{ev}_{\cD_{2}}"'] & \mathrm{StrVect}(\cD[1/p]) \ar[d,"\sim" sloped] \\
    & \mathrm{StrVect}(\cD_{2}[1/p]).
\end{tikzcd}
\]
Here, the right vertical equivalence is defined from the fact that both of $(\cD,\cM_{\cD})$ and $(\cD_{2},\cM_{\cD_{2}})$ cover the final object of the topos associated with $(X,\cM_{X})_{\mathrm{crys}}$. The natural map
\[
\mathrm{Hom}_{\mathrm{Isoc}((X,\cM_{X})_{\mathrm{crys}})}(\cE_{1}^{+}[1/p],\cE_{2}^{+}[1/p])\to \mathrm{Hom}_{\mathrm{StrVect}(\cD_{2}[1/p])}(\mathrm{ev}_{\cD_{2}}(\cE_{1}^{+}[1/p]),\mathrm{ev}_{\cD_{2}}(\cE_{2}^{+}[1/p]))
\]
is an isomorphism because of the following general claim: for a ring $A$, a finitely generated $A$-module $M$, and a $p$-torsion free $A$-module $N$, the natural map
\[
\mathrm{Hom}_{A}(M,N)[1/p]\to \mathrm{Hom}_{A[1/p]}(M[1/p],N[1/p])
\]
is an isomorphism. Therefore the natural map
\[
\mathrm{Hom}_{\mathrm{Isoc}((X,\cM_{X})_{\mathrm{crys}})}(\cE_{1}^{+}[1/p],\cE_{2}^{+}[1/p])\to \mathrm{Hom}_{\mathrm{StrVect}(\cD[1/p])}(\mathrm{ev}_{\cD}(\cE_{1}^{+}[1/p]),\mathrm{ev}_{\cD}(\cE_{2}^{+}[1/p]))
\]
is also an isomorphism. This proves the fully faithfulness of the functor (1).
\end{proof}

\begin{rem}
    Under the assumption of Proposition \ref{eq isocrystal and conn}, we have functors
    \[
    \mathrm{Isoc}_{\mathrm{lf}}((X,\cM_{X})_{\mathrm{crys}})\to \mathrm{StrVect}(\cD[1/p])
    \]
    given in Proposition \ref{functor from isoc to rat conn} and 
    \[
    \mathrm{StrVect}(\cD[1/p])\to \mathrm{Isoc}_{\mathrm{lf}}((X,\cM_{X})_{\mathrm{crys}})
    \]
    given in Proposition \ref{eq isocrystal and conn} (or its proof). Although the composition
    \[
    \mathrm{StrVect}(\cD[1/p])\to \mathrm{Isoc}_{\mathrm{lf}}((X,\cM_{X})_{\mathrm{crys}})\to \mathrm{StrVect}(\cD[1/p])
    \]
    is isomorphic to the identity functor, we do not know whether the composition
    \[
    \mathrm{Isoc}_{\mathrm{lf}}((X,\cM_{X})_{\mathrm{crys}})\to \mathrm{StrVect}(D[1/p])\to \mathrm{Isoc}_{\mathrm{lf}}((X,\cM_{X})_{\mathrm{crys}})
    \]
    is isomorphic to the identity functor. In particular, we do not know whether the category $\mathrm{Isoc}_{\mathrm{lf}}((X,\cM_{X})_{\mathrm{crys}})$ is indeed larger than $\mathrm{Isoc}_{\mathrm{lftf}}((X,\cM_{X})_{\mathrm{crys}})$.
\end{rem}

\begin{dfn}
    Let $(X,\cM_{X})$ be an fs log smooth log adic space over $K$. We have the sheaf of differentials $\Omega^{1}_{(X,\cM_{X})/K}$ defined in \cite[Definition 3.3.6]{dllz23b}. Let $\mathrm{MIC}(X,\cM_{X})$ denote the category of vector bundles with integrable connection on $(X,\cM_{X})$.
\end{dfn}

\begin{prop}\label{isoc to mic}
Let $\cM_{\cO_{K}}$ be a log structure on $\mathrm{Spf}(\cO_{K})$ which is either of the trivial log structure or the log structure associated with the prelog structure $\cO_{K}\backslash \{0\}\hookrightarrow \cO_{K}$. Let $(\fX,\cM_{\fX})$ be a log smooth log formal scheme over $(\mathrm{Spf}(\cO_{K}),\cM_{\cO_{K}})$ and $(X_{p=0},\cM_{X_{p=0}})$ be the mod-$p$ fiber of $(\fX,\cM_{\fX})$. Then there exists a natural functor
    \[
    \mathrm{Isoc}_{\mathrm{lftf}}((X_{p=0},\cM_{X_{p=0}})_{\mathrm{crys}})\to \mathrm{MIC}((\fX,\cM_{\fX})_{\eta}) \ \ \ (\cE\mapsto (E\coloneqq \cE_{(\fX,\cM_{\fX})},\nabla_{E})),
    \]
where $\cE_{(\fX,\cM_{\fX})}$ is the vector bundle on $\fX_{\eta}$ obtained as the evaluation of $\cE$ at the $p$-adic log PD-thickening $(X_{0},\fX,\cM_{\fX})$. Moreover, when $\cO_{K}=W$ and $\cM_{W}$ is trivial, this functor is fully faithful.
\end{prop}

\begin{proof}
    When $(\fX,\cM_{\fX})=(\mathrm{Spf}(R),\cM_{R})$ is affine, $\mathrm{MIC}((\fX,\cM_{\fX})_{\eta})$ is equivalent to the category of finite projective $R[1/p]$-modules $M$ equipped with an integrable connection
    \[
    M\to M\otimes_{R} \Omega^{1}_{(R,\cM_{R})/(\cO_{K},\cM_{\cO_{K}})}, 
    \]
    which include $\mathrm{Vect}^{\nabla}(R[1/p])$ as a full subcategory. Hence, the assertion follows from Proposition \ref{functor from isoc to rat conn} and Proposition \ref{eq isocrystal and conn}.
\end{proof}

\begin{dfn}[Filtered isocrystals]
    Let $(\fX,\cM_{\fX})$ be a semi-stable log formal scheme over $\cO_{K}$. Consider a pair $(\cE,\mathrm{Fil}^{\bullet}(E))$ consisting of an object $\cE\in \mathrm{Isoc}^{\varphi}_{\mathrm{lftf}}((X_{p=0},\cM_{X_{p=0}})_{\mathrm{crys}})$ and a descending filtration $\mathrm{Fil}^{\bullet}(E)$ on $E$, where $(E,\nabla_{E})$ is the object of $\mathrm{MIC}((\fX,\cM_{\fX})_{\eta})$ obtained via the functor in the Proposition \ref{isoc to mic}. We say that $(\cE,\mathrm{Fil}^{\bullet}(E))$ is a \emph{filtered $F$-isocrystal} on $(\fX,\cM_{\fX})$ if the following conditions are satisfied:
    \begin{itemize}
        \item $\mathrm{Fil}^{\bullet}(E)$ is a separated and exhausted filtration with locally free grading pieces;
        \item the triple $(E,\nabla_{E},\mathrm{Fil}^{\bullet}(E))$ satisfies the Griffiths transversality.
    \end{itemize}
\end{dfn}

\subsection{\texorpdfstring{$F$}--isocrystals and monodromy operators}

Let $k$ be a perfect field of characteristic $p>0$. Let $W\coloneqq W(k)$ be the Witt ring of $k$. Let $(X,\cM_{X})=(\mathrm{Spec}(R),\cM_{R})$ be a fine log scheme over $k$ admitting a strict \'{e}tale map $(X,\cM_{X})\to (\mathrm{Spec}(k[P]),P)^{a}$, where $P$ is an fs monoid such that the order of the torsion part of $P^{\mathrm{gp}}$ is prime to $p$.

The goal of this subsection is to describe $F$-isocrystals on $(X,\cM_{X}\oplus 0^{\bN})^{a}$ as $F$-isocrystals on $(X,\cM_{X})$ equipped with monodromy operators (Proposition \ref{Fisoc on log crys site and Fisoc with monodromy ope}). We follows the argument in \cite[\S 2]{gy24}.

Take a strict closed immersion
    \[
    (X,\cM_{X})\hookrightarrow (\mathrm{Spf}(B),\cM_{B}),
    \]
where $(\mathrm{Spf}(B),\cM_{B})$ is a fine log smooth log formal scheme over $W$ equipped with a Frobenius lift $\varphi_{B}$. Let $(\mathrm{Spf}(D),\cM_{D})$ be its $p$-completed log PD-envelope. Let $D\{u\}$ be the $p$-completed PD-polynomial ring over $D$ with a single variable $u$. Let $\cM_{B\langle u\rangle}\coloneqq (\cM_{B}\oplus u^{\bN})^{a}$ be the log structure on $\mathrm{Spf}(B\langle u\rangle)$ and $\cM_{D\{u\}}$ be the pullback log structure of $\cM_{B\langle u\rangle}$ to $\mathrm{Spf}(D\{u\})$. Then $(\mathrm{Spf}(D\{u\}),\cM_{D\{u\}})$ is the $p$-completed log PD-envelope of the strict closed immersion
\[
(X,\cM_{X}\oplus 0^{\bN})^{a}\hookrightarrow (\mathrm{Spf}(B\langle u\rangle),\cM_{B\langle u\rangle})
\]
given by $u\mapsto 0$. Let $\varphi_{B\langle u\rangle}$ denote the Frobenius lift on $(\mathrm{Spf}(B\langle u\rangle),\cM_{B\langle u\rangle})$ induced from $\varphi_{B}$ and $u\mapsto u^{p}$. Let $\varphi_{D}$ and $\varphi_{D\{u\}}$ be induced Frobenius lifts on respective log formal schemes.

\begin{dfn}
We define the following categories.
\begin{itemize}
    \item Let $\mathrm{Vect}^{\varphi,\nabla}(D[1/p])$ be the category of tuples $(M,\varphi_{M},\nabla_{M})$ consisting of an object $(M,\nabla_{M})$ of $\mathrm{Vect}^{\nabla}(\mathrm{Spf}(D)[1/p])$ and a $D[1/p]$-linear isomorphism $\varphi_{M}\colon \varphi_{D}^{*}M\isom M$ such that $\varphi_{M}$ is horizontal with respect to $\nabla_{M}$.
    \item Let $\mathrm{Vect}^{\varphi,\nabla}(D\{u\}[1/p])$ be the category of tuples $(M,\varphi_{M},\nabla_{M})$ consisting of an object $(M,\nabla_{M})$ of $\mathrm{Vect}^{\nabla}(\mathrm{Spf}(D\{u\})[1/p])$ and a $D\{u\}[1/p]$-linear isomorphism $\varphi_{M}\colon \varphi_{D\{u\}}^{*}M\isom M$ such that $\varphi_{M}$ is horizontal with respect to $\nabla_{M}$.
    \item Let $\mathrm{Vect}^{\varphi,\nabla,N}(D[1/p])$ be the category of tuples $(M,\varphi_{M},\nabla_{M},N_{M})$ consisting of an object $(M,\varphi_{M},\nabla_{M})$ of $\mathrm{Vect}^{\varphi,\nabla}(D[1/p])$ and a $D[1/p]$-linear endomorphism $N_{M}\colon M\to M$ such that the following conditions are satisfied:
\begin{itemize}
    \item $N_{M}$ is horizontal with respect to $\nabla_{M}$;
    \item $N_{M}\varphi_{M}=p\varphi_{M}(N_{\varphi_{D}^{*}M})$, where $N_{\varphi_{D}^{*}M}\colon \varphi_{D}^{*}M\to \varphi_{D}^{*}M$ is the base change of $N_{M}$ along $\varphi_{D}$.
\end{itemize}
\end{itemize}

\end{dfn}

\begin{construction}\label{construction from phinabla to phinablaN}
We  define a functor
\[
\mathrm{Vect}^{\varphi,\nabla}(D\{u\}[1/p])\to \mathrm{Vect}^{\varphi,\nabla,N}(D[1/p]) \ \ \ ((M,\varphi_{M},\nabla_{M})\mapsto (M_{0},\varphi_{M_{0}},\nabla_{M_{0}},N_{M_{0}}))
\]
as follows. Let $J\colon \mathrm{Ker}(D\{u\}[1/p]\stackrel{u\mapsto 0}{\twoheadrightarrow} D[1/p])$. Let $M_{0}\coloneqq M/JM$ and $\varphi_{M_{0}}$ is the induced map from $\varphi_{M}$. Since there is a decomposition
\begin{align*}
\Omega^{1}_{(B\langle u\rangle,\cM_{B\langle u\rangle})/W}&\cong (\Omega^{1}_{(B,\cM_{B})/W}\otimes_{B} B\langle u\rangle)\oplus (\Omega^{1}_{(W\langle u\rangle,u^{\bN})^{a}/W}\otimes_{W\langle u\rangle} B\langle u\rangle) \\
&\cong (\Omega^{1}_{(B,\cM_{B})/W}\otimes_{B} B\langle u\rangle)\oplus 
B\langle u\rangle\cdot \mathrm{dlog}(u),
\end{align*}
the connection $\nabla_{M}$ induces maps
\begin{align*}
    M &\twoheadrightarrow M\otimes_{B} \Omega^{1}_{(B,\cM_{B})/W} \\
    M &\twoheadrightarrow M\otimes_{W\langle u\rangle} \Omega^{1}_{(W\langle u\rangle,u^{\bN})^{a}/W}
\end{align*}
by taking the composition with natural projection maps. These are
denoted by $\nabla_{M}^{\mathrm{hol}}$ and $N_{M}$, respectively. By Leibniz rule, we have $\nabla_{M}(JM)\subset J\cdot M\otimes_{B\langle u\rangle} \Omega^{1}_{(B\langle u\rangle,\cM_{B\langle u\rangle})/W}$. We define $\nabla_{M_{0}}$ and $N_{M_{0}}$ as the mod-$J$ reduction of $\nabla_{M}^{\mathrm{hol}}$ and $N_{M}$, respectively. To check that $(M_{0},\varphi_{M_{0}},\nabla_{M_{0}},N_{M_{0}})$ is indeed an object of $\mathrm{Vect}^{\varphi,\nabla,N}(D[1/p])$, it is enough to show that $N_{M_{0}}\varphi_{M_{0}}=p\varphi_{M_{0}}(N_{\varphi_{D}^{*}M_{0}})$. This is because the fact that $\varphi_{M}$ is horizontal and $\varphi_{D\{u\}}(\mathrm{dlog}(u))=p\cdot \mathrm{dlog}(u)$ implies that $N_{M}\varphi_{M}=p\varphi_{M}(N_{\varphi_{D\{u\}}^{*}M})$.
\end{construction}

\begin{dfn}
Let $\mathrm{Isoc}^{\varphi,N}_{\mathrm{lf}}((X,\cM_{X})_{\mathrm{crys}})$ denote the category of triples $(\cE,\varphi_{\cE},N_{\cE})$ such that $(\cE,\varphi_{\cE})$ is a locally free $F$-isocrystal on $(X,\cM_{X})$ and $N_{\cE}\colon (\cE,\varphi_{\cE})\to (\cE,p\varphi_{\cE})$ is a map of $F$-isocrystals. 
\end{dfn}

\begin{lem}\label{N is nilp}
    For every $(\cE,\varphi_{\cE},N_{\cE})\in \mathrm{Isoc}^{\varphi,N}_{\mathrm{lf}}((X,\cM_{X})_{\mathrm{crys}})$, the endomorphism $N_{\cE}$ is nilpotent.
\end{lem}

\begin{proof}
    Let $U$ be the biggest open subset of $X$ where the log structure $\cM_{X}$ is trivial. Concretely, we can write $U=X\times_{\mathrm{Spec}(k[P])} \mathrm{Spec}(k[P^{\mathrm{gp}}])$. By \cite[Proposition 2.14]{gy24}, $N_{\cE}|_{U}$ is nilpotent. Take a formal log smooth lift $(\fX,\cM_{\fX})$ of $(X,\cM_{X})$ over $W$ and $\fU$ be the open subset of $\fX$ corresponding to $U\subset X$. There is the following commutative diagram:
    \[
    \begin{tikzcd}
        \mathrm{Isoc}((X,\cM_{X})_{\mathrm{crys}}) \ar[r] \ar[d] & \mathrm{MIC}((\fX,\cM_{\fX})_{\eta}) \ar[d] \\
        \mathrm{Isoc}(U_{\mathrm{crys}}) \ar[r] & \mathrm{MIC}(\fU_{\eta}).
    \end{tikzcd}
    \]
    Since $U$ is dense in $X$ and $\fX$ is reduced, the restriction functor 
    \[
    \mathrm{Vect}(\fX_{\eta})\to \mathrm{Vect}(\fU_{\eta})
    \]
    is faithful. In particular, the right vertical functor in the diagram is also faithful. By Proposition \ref{isoc to mic}, both horizontal functors are fully faithful. Therefore, the left vertical functor is faithful, and so $N_{\cE}$ is also nilpotent.
\end{proof}

\begin{construction}\label{construction of Res}
By Proposition \ref{eq isocrystal and conn}, there are natural equivalences
\begin{align*}
    \mathrm{Isoc}^{\varphi,N}_{\mathrm{lf}}((X,\cM_{X})_{\mathrm{crys}})&\simeq \mathrm{Vect}^{\varphi,\nabla,N}(D[1/p]) \\
    \mathrm{Isoc}^{\varphi}_{\mathrm{lftf}}((X,\cM_{X}\oplus 0^{\bN})^{a}_{\mathrm{crys}})&\simeq \mathrm{Vect}^{\varphi,\nabla}(D\{u\}[1/p])
\end{align*}
Combining with the functor constructed in Construction \ref{construction from phinabla to phinablaN}, we obtain a functor
\[
\mathrm{Isoc}^{\varphi}_{\mathrm{lftf}}((X,\cM_{X}\oplus 0^{\bN})^{a}_{\mathrm{crys}})\to \mathrm{Isoc}^{\varphi,N}_{\mathrm{lf}}((X,\cM_{X})_{\mathrm{crys}}),
\]
denoted by $\mathrm{Res}$. The functor $\mathrm{Res}$ is independent of the choice of an embedding $(X,\cM_{X})\hookrightarrow (\mathrm{Spf}(B),\cM_{B})$ because, if we are given two such embeddings, they are dominated by a common embedding (for example, the exactification of their product).
\end{construction}

\begin{rem}
We define a forgetting functor
\[
U\colon \mathrm{Isoc}_{\mathrm{lftf}}((X,\cM_{X}\oplus 0^{\bN})^{a}_{\mathrm{crys}})\to \mathrm{Isoc}_{\mathrm{lf}}((X,\cM_{X})_{\mathrm{crys}})
\]
by $U(\cE)(T,\cM_{T})\coloneqq \cE((T,\cM_{T}\oplus 0^{\bN})^{a})$ for each $\cE\in \mathrm{Isoc}((X,\cM_{X}\oplus 0^{\bN})^{a}_{\mathrm{crys}})$ and $(T,\cM_{T})\in (X,\cM_{X})_{\mathrm{crys}}$. Also, we can consider a similar functor for the category of $F$-isocrystals.

Let $\cE$ be an object of $\mathrm{Isoc}^{\varphi}_{\mathrm{lftf}}((X,\cM_{X}\oplus 0^{\bN})^{a}_{\mathrm{crys}})$. By construction of the functor $\mathrm{Res}$, the $D[1/p]$-module with an Frobenius isomorphism $\mathrm{Res}(\cE)_{D}$ is isomorphic to $\cE_{D\{u\}}\otimes_{D\{u\}} D$, which is isomorphic to $\cE_{(D,\cM_{D}\oplus 0^{\bN})^{a}}$ by the crystal property of $\cE$ for the map of $p$-adic log PD-thickenings $(D\{u\},\cM_{D\{u\}})\to (D,\cM_{D}\oplus 0^{\bN})^{a}$ given by $u\mapsto 0$. Therefore, the following triangle is commutative:
\[
\begin{tikzcd}
\mathrm{Isoc}^{\varphi}_{\mathrm{lf}}((X,\cM_{X}\oplus 0^{\bN})^{a}_{\mathrm{crys}}) \ar[r,"\mathrm{Res}"] \ar[rd,"U"'] & \mathrm{Isoc}^{\varphi,N}_{\mathrm{lf}}((X,\cM_{X})_{\mathrm{crys}}) \ar[d] \\
    & \mathrm{Isoc}^{\varphi}_{\mathrm{lf}}((X,\cM_{X})_{\mathrm{crys}}).
\end{tikzcd}
\]
Here, the right vertical arrow is the functor forgetting $N$.
\end{rem}

\begin{lem}\label{beilinson trick}
Let $A,A_{0}$ be $p$-complete and $p$-torsion free rings equipped with Frobenius lifts $\varphi_{A}$ and $\varphi_{A_{0}}$, respectively. Suppose that we are given a surjection $A\twoheadrightarrow A_{0}$ and its section $A_{0}\to A$ that are compatible with Frobenius lifts. We assume that the following condition holds:
\begin{itemize}
    \item if we set $I\coloneqq \mathrm{Ker}(A\twoheadrightarrow A_{0})$, for an arbitrary integer $n\geq 1$, there exists an integer $m\geq 1$ such that $\varphi_{A}^{m}(J)\subset p^{n}A$.
\end{itemize}
Let $\mathrm{Vect}^{\varphi}(A[1/p])$ (resp. $\mathrm{Vect}^{\varphi}(A_{0}[1/p])$) be the category of finite projective $A[1/p]$-modules (resp. $A_{0}$-modules) $M$ equipped with $A[1/p]$-linear isomorphism $\varphi_{M}\colon \varphi_{A}^{*}M\isom M$ (resp. $A_{0}[1/p]$-linear isomorphism $A_{0}$-linear isomorphism $\varphi_{A_{0}}^{*}M\to M$). Then the base change functor
    \[
    \mathrm{Vect}^{\varphi}(A[1/p])\to \mathrm{Vect}^{\varphi}(A_{0}[1/p])
    \]
    along $A\twoheadrightarrow A_{0}$ is an equivalence, and the inverse of this functor is given by
    \[
    (M,\varphi_{M})\mapsto (M\otimes_{A_{0}} A,\varphi_{M}\otimes \varphi_{A}).
    \]
\end{lem}

\begin{proof}
    This follows from the same argument as \cite[Theorem 2.21]{gy24}.
\end{proof}

\begin{lem}\label{equivalence from phinabla to phinablaN}
    The functor $\mathrm{Vect}^{\varphi,\nabla}(D\{u\}[1/p])\to \mathrm{Vect}^{\varphi,\nabla,N}(D[1/p])$ constructed in Construction \ref{construction from phinabla to phinablaN} is an equivalence, and the inverse sends $(M,\varphi_{M},\nabla_{M},N_{M})$ to a $D\{u\}[1/p]$-module $M'\coloneqq M\widehat{\otimes}_{W} W\{u\}$ equipped with additional structures
    \[ (\varphi_{M'},\nabla_{M'})\coloneqq (\varphi_{M}\otimes \varphi_{W\{u\}},\nabla_{M}\otimes \mathrm{id}_{W\{u\}}\oplus (N_{M}\otimes \mathrm{id}_{W\{u\}}+\mathrm{id}_{M}\otimes N_{W\{u\}})\cdot \mathrm{dlog}(u)),
    \]
    where we used the identification
    \[
    M'\otimes_{B\langle u\rangle} \Omega^{1}_{(B\langle u\rangle,\cM_{B\langle u\rangle})/W}\cong (M'\otimes_{B} \Omega^{1}_{(B,\cM_{B})/W}) \oplus M'\cdot \mathrm{dlog}(u).
    \]
\end{lem}

\begin{proof}
    First, we check that $(M',\varphi_{M'},\nabla_{M'})$ defined in the assertion indeed belongs to $\mathrm{Vect}^{\varphi,\nabla}(D\{u\}[1/p])$. The fact that $\varphi_{M}$ is horizontal and the equations $N_{M}\varphi_{M}=p\varphi_{M}N_{\varphi_{D}^{*}M}$ and $N_{W\{u\}}\varphi_{W\{u\}}=p\varphi_{W\{u\}}N_{W\{u\}}$ imply that $\varphi_{M'}$ is horizontal. To check the topologically quasi-nilpotence of $\nabla_{M'}$, it suffices to prove that, for $x\otimes u^{[n]}\in M'$, the element
    \begin{align*}
        &(\prod_{j=0}^{m}(N_{M}\otimes \mathrm{id}_{W\{u\}}+\mathrm{id}_{M}\otimes N_{W\{u\}}-j))(x\otimes u^{[n]}) \\
        =&(\prod_{j=0}^{m}(N_{M}\otimes \mathrm{id}_{W\{u\}}+(n-j)))(x\otimes u^{[n]})
    \end{align*}
    converges to $0$ when $m\to \infty$. By Lemma \ref{N is nilp}, $N_{M}$ is nilpotent. Take $r\in \bN$ such that $N_{M}^{r+1}=0$. Then the above element is written as a $\bZ$-linear combination of $x\otimes u^{[n]},N_{M}(x)\otimes u^{[n]},\dots,N_{M}^{r}(x)\otimes u^{[n]}$, and their coefficients converges to $0$ in the $p$-adic topology. This proves that $(M',\varphi_{M'},\nabla_{M'})$ belongs $\mathrm{Vect}^{\varphi,\nabla}(D\{u\}[1/p])$. 
    It follows from Lemma \ref{beilinson trick} and the same argument as \cite[Corollary 2.23]{gy24} that the functor $\mathrm{Vect}^{\varphi,\nabla,N}(D[1/p])\to \mathrm{Vect}^{\varphi,\nabla}(D\{u\}[1/p])$ defined in the assertion is indeed the inverse.
\end{proof}

\begin{prop}\label{Fisoc on log crys site and Fisoc with monodromy ope}
    The functor 
    \[
    \mathrm{Res}\colon \mathrm{Isoc}^{\varphi}_{\mathrm{lftf}}((X,\cM_{X}\oplus 0^{\bN})^{a}_{\mathrm{crys}})\to \mathrm{Isoc}^{\varphi,N}_{\mathrm{lf}}((X,\cM_{X})_{\mathrm{crys}})
    \]
    is an equivalence.
\end{prop}

\begin{proof}
This follows from the construction of the functor $\mathrm{Res}$ and Lemma \ref{equivalence from phinabla to phinablaN}.
\end{proof}

\begin{cor}\label{N=0 log isoc}
    The functor $\mathrm{Isoc}^{\varphi}_{\mathrm{lf}}((X,\cM_{X})_{\mathrm{crys}})\to \mathrm{Isoc}^{\varphi}_{\mathrm{lftf}}((X,\cM_{X}\oplus 0^{\bN})^{a}_{\mathrm{crys}})$ defined as the pullback along the natural map $(X,\cM_{X}\oplus 0^{\bN})^{a}\to (X,\cM_{X})$ is fully faithful, and the essential image of the functor
    \[ \mathrm{Isoc}^{\varphi}_{\mathrm{lf}}((X,\cM_{X})_{\mathrm{crys}})\to \mathrm{Isoc}^{\varphi}_{\mathrm{lftf}}((X,\cM_{X}\oplus 0^{\bN})^{a}_{\mathrm{crys}})\stackrel{\mathrm{Res}}{\isom} \mathrm{Isoc}^{\varphi,N}_{\mathrm{lf}}((X,\cM_{X})_{\mathrm{crys}})
    \]
    consists of objects with $N_{\cE}=0$.
\end{cor}

\begin{proof}
    We can directly check that there is a commutative diagram
    \[
    \begin{tikzcd} \mathrm{Isoc}^{\varphi}_{\mathrm{lf}}((X,\cM_{X})_{\mathrm{crys}}) \ar[r] \ar[d,"\sim", sloped] & \mathrm{Isoc}^{\varphi}_{\mathrm{lftf}}((X,\cM_{X}\oplus 0^{\bN})^{a}_{\mathrm{crys}}) \ar[d,"\sim", sloped] \\
        \mathrm{Vect}^{\varphi,\nabla}(D[1/p]) \ar[r] & \mathrm{Vect}^{\varphi,\nabla}(D\{u\}[1/p]),
    \end{tikzcd}
    \]
    where the bottom horizontal functor sends $(M,\varphi_{M},\nabla_{M})$ to 
    \[
    (M\widehat{\otimes}_{W} W\{u\},\varphi_{M}\otimes \varphi_{W\{u\}},\nabla_{M}\otimes \mathrm{id}_{W\{u\}}\oplus (\mathrm{id}_{M}\otimes N_{W\{u\}})\cdot \mathrm{dlog}(u)).
    \] 
    Then the assertion follows from Proposition \ref{Fisoc on log crys site and Fisoc with monodromy ope}.
    \end{proof}

\section{Local systems on log adic spaces}\label{section loc sys}

Let $k$ be a perfect field, $W\coloneqq W(k)$ the Witt ring of $k$, $K$ a totally ramified finite extension of $K_{0}\coloneqq W[1/p]$, and $\cO_{K}$ a valuation ring of $K$. Let $C$ be the completion of an algebraic closure of $K$ and $\cO_{C}$ be the valuation ring of $C$. Fix a compatible system of $p^{n}$-th power roots of unity for $n\geq 1$ denoted by $(\zeta_{p},\zeta_{p^{2}},\dots)$ and let $\epsilon\coloneqq (1,\zeta_{p},\zeta_{p^{2}},\dots)\in \cO_{C}^{\flat}$.

In this section, let $(X,\cM_{X})$ be a locally noetherian fs log adic space over $\mathrm{Spa}(K,\cO_{K})$ (\cite{dllz23b}).

\subsection{Horizontal period sheaves on log adic spaces}

We start with the definition of various period sheaves on log adic spaces.

\begin{dfn}
Let $S$ be a $p$-torsion free integral perfectoid $\cO_{C}$-algebra and $\displaystyle S^{\flat}\coloneqq \varprojlim_{x\mapsto x^{p}} S/p$ be the tilt of $S$.
\begin{itemize}
    \item Set $A_{\mathrm{inf}}(S)\coloneqq W(S^{\flat})$. The natural surjection $A_{\mathrm{inf}}(S)\twoheadrightarrow S^{\flat}\twoheadrightarrow S/p$ uniquely lifts to a surjection $\theta\colon A_{\mathrm{inf}}(S)\twoheadrightarrow S$, and the ideal $\mathrm{Ker}(\theta)$ is generated by a regular element. Let $\varphi_{A_{\mathrm{inf}}(S)}$ be the Witt vector Frobenius map. Let $q\coloneqq [\epsilon]\in A_{\mathrm{inf}}(\cO_{C})$ and $\mu\coloneqq q-1\in A_{\mathrm{inf}}(\cO_{C})$.
    \item  Set $B_{\mathrm{inf}}(S)\coloneqq A_{\mathrm{inf}}(S)[1/p]$. Let $\varphi_{B_{\mathrm{inf}}(S)}$ be the induced ring map from $\varphi_{A_{\mathrm{inf}}(S)}$.
    \item Let $B^{+}_{\mathrm{dR}}(S)$ be the $\mathrm{Ker}(\theta)$-adic completion of $B_{\mathrm{inf}}(S)$ equipped with the $\mathrm{Ker}(\theta)$-adic filtration $\mathrm{Fil}^{\bullet}(B^{+}_{\mathrm{dR}}(S))$. We put
    \[
    t\coloneqq \mathrm{log}(q)\coloneqq \sum_{n\geq 1}(-1)^{n+1}\frac{\mu^{n}}{n} \in B^{+}_{\mathrm{dR}}(\cO_{C}).
    \]
    \item Set 
    \[
    B_{\mathrm{dR}}(S)\coloneqq B^{+}_{\mathrm{dR}}(S)[1/\xi]=B^{+}_{\mathrm{dR}}(S)[1/\mu]=B^{+}_{\mathrm{dR}}(S)[1/t]
    \]
    and $\mathrm{Fil}^{\bullet}(B_{\mathrm{dR}}(S))$ be the filtration defined by $\mathrm{Fil}^{n}(B_{\mathrm{dR}}(S))\coloneqq \xi^{n}B^{+}_{\mathrm{dR}}(S)$ for $n\in \bZ$.
    \item Let $A^{0}_{\mathrm{crys}}(S)\coloneqq A_{\mathrm{inf}}(S)[\{\xi^{[n]}\}_{n\geq 1}]$ be the PD-envelope of the surjection $\theta$ equipped with the filtration $\mathrm{Fil}^{\bullet}(A^{0}_{\mathrm{crys}}(S))$ whose $n$-th filtration is the ideal generated by $\xi^{[i]}$ for $i\geq n$, where $\xi$ is a generator of $\mathrm{Ker}(\theta)\subset A_{\mathrm{inf}}(S)$. The ring map $A^{0}_{\mathrm{crys}}(S)\to A^{0}_{\mathrm{crys}}(S)$ induced from $\varphi_{A_{\mathrm{inf}}(S)}$ is denoted by $\varphi_{A^{0}_{\mathrm{crys}}(S)}$.
    \item Let $A_{\mathrm{crys}}(S)$ denote the $p$-completion of $A^{0}_{\mathrm{crys}}(S)$ with the filtration $\mathrm{Fil}^{\bullet}(A_{\mathrm{crys}}(S))$ whose $n$-th filtration is the closure of $\mathrm{Fil}^{n}(A^{0}_{\mathrm{crys}}(S))$. In the same way as \cite[Proposition 2.23]{tt19}, a natural map $A_{\mathrm{crys}}(S)\to B^{+}_{\mathrm{dR}}(S)$ is defined, and this map is injective by \cite[Proposition 5.36]{ck19} (see also Lemma \ref{genaral injectivity from Acrys to Bdr}). We put
    \[
    t\coloneqq \mathrm{log}(q)\coloneqq \sum_{n\geq 1}(-1)^{n+1}(n-1)!\mu^{[n]} \in A_{\mathrm{crys}}(\cO_{C}).
    \]
    This definition is compatible with the one defined in $\bB^{+}_{\mathrm{dR}}(S)$ via the above embedding. Let $\varphi_{A_{\mathrm{crys}}(S)}$ denote the Frobenius lift on $A_{\mathrm{crys}}(S)$ induced from $\varphi_{A^{0}_{\mathrm{crys}}(S)}$.
    \item Set $B^{+}_{\mathrm{crys}}(S)\coloneqq A_{\mathrm{crys}}(S)[1/p]$ and let $\mathrm{Fil}^{\bullet}(B^{+}_{\mathrm{crys}}(S))$ be the induced filtration. Let $\varphi_{B^{+}_{\mathrm{crys}}(S)}$ the ring map $B^{+}_{\mathrm{crys}}(S)\to B^{+}_{\mathrm{crys}}(S)$ induced from $\varphi_{A_{\mathrm{crys}}(S)}$.
    \item Set
    \[
    B_{\mathrm{crys}}(S)\coloneqq B^{+}_{\mathrm{crys}}(S)[1/t]=B^{+}_{\mathrm{crys}}(S)[1/\mu]=A_{\mathrm{crys}}(S)[1/\mu]
    \]
    let $\mathrm{Fil}^{\bullet}(B_{\mathrm{crys}}(S))$ be the filtration defined by
    \[
    \mathrm{Fil}^{n}(B_{\mathrm{crys}}(S))\coloneqq \sum_{i\in \bZ} t^{-i}\mathrm{Fil}^{i+n}(B^{+}_{\mathrm{crys}}(S))
    \]
    for $n\in \bZ$. Let $\varphi_{B_{\mathrm{crys}}(S)}$ the ring map $B_{\mathrm{crys}}(S)\to B_{\mathrm{crys}}(S)$ induced from $\varphi_{B^{+}_{\mathrm{crys}}(S)}$.
    \item Let $\cM_{A_{\mathrm{inf}}(S)}$ be the log structure on $\mathrm{Spf}(A_{\mathrm{inf}}(S))$ associated with a prelog structure $\cO_{C}^{\flat}\backslash \{0\}\to A_{\mathrm{inf}}(S)$ given by $x\mapsto [x]$ and $\cM_{S}$ be the pullback structure on $\mathrm{Spf}(S)$ of it along $\theta\colon A_{\mathrm{inf}}(S)\twoheadrightarrow S$. We define $(\mathrm{Spf}(\widehat{A}_{\mathrm{st}}(S)),\cM_{\widehat{A}_{\mathrm{st}}(S)})$ as the $p$-completed log PD-envelope of the closed immersion
    \[
    (\mathrm{Spf}(S),\cM_{S})\hookrightarrow (\mathrm{Spf}(A_{\mathrm{inf}}(S)\widehat{\otimes}_{W} W\langle u \rangle),\cM_{A_{\mathrm{inf}}(S)}\oplus u^{\bN})^{a}
    \]
    defined by $\theta$ and $u\mapsto p$.
    Explicitly, we can define an isomorphism 
    \[
    \widehat{A}_{\mathrm{st}}(S)\cong A_{\mathrm{crys}}(S)\{v\}\subset B^{+}_{\mathrm{dR}}(S)[[v]]
    \]
    sending $v$ to $\frac{u}{[p^{\flat}]}-1$
    by choosing $p^{\flat}=(p,p^{1/p},\dots)\in \cO_{C}^{\flat}$, where $A_{\mathrm{crys}}(S)\{v\}$ is the $p$-completed PD-polynomial ring with a single variable $v$ over $A_{\mathrm{crys}}(S)$.

    We define a monodromy operator $N_{\widehat{A}_{\mathrm{st}}(S)}\colon \widehat{A}_{\mathrm{st}}(S)\to \widehat{A}_{\mathrm{st}}(S)$ as follows.
    The $A_{\mathrm{crys}}(S)$-linear map $A_{\mathrm{crys}}(S)\{v\}\to A_{\mathrm{crys}}(S)\{v\}$ defined by $f(v)\mapsto (v+1)f'(v)$ induces the $A_{\mathrm{crys}}(S)$-linear map $\widehat{A}_{\mathrm{st}}(S)\to \widehat{A}_{\mathrm{st}}(S)$, denoted by $N_{\widehat{A}_{\mathrm{st}}(S)}$, via the above isomorphism. We can easily check that $N_{\widehat{A}_{\mathrm{st}}(S)}$ is independent of the choice of $p^{\flat}$. The $A_{\mathrm{crys}}(S)$-subalgebra $A_{\mathrm{crys}}(S)[\mathrm{log}(\frac{u}{[p^{\flat}]})]$ of $\widehat{A}_{\mathrm{st}}(S)(S)$ is the subset of nilpotent elements for $N_{\widehat{A}_{\mathrm{st}}(S)}$ as checked inductively.
    
    The Frobenius lift $\varphi_{A_{\mathrm{inf}}(S)}$ induces a Frobenius lift on $(\mathrm{Spf}(A_{\mathrm{inf}}(S)),\cM_{A_{\mathrm{inf}}(S)})$, and we define a Frobenius lift $\varphi_{W\{u\}}$ on $(\mathrm{Spf}(W\{u\}),u^{\bN})^{a}$ by $u\mapsto u^{p}$. Then $\varphi_{A_{\mathrm{inf}}(S)}$ and $\varphi_{W\{u\}}$ induces a Frobenius lift on $(\mathrm{Spf}(\widehat{A}_{\mathrm{st}}(S)),\cM_{\widehat{A}_{\mathrm{st}}(S)})$, denoted by $\varphi_{\widehat{A}_{\mathrm{st}}(S)}$. The underlying ring map $\widehat{A}_{\mathrm{st}}(S)\to \widehat{A}_{\mathrm{st}}(S)$ is also denoted by $\varphi_{\widehat{A}_{\mathrm{st}}(S)}$.
    \item Set $\widehat{B}^{+}_{\mathrm{st}}(S)\coloneqq \widehat{A}_{\mathrm{st}}(S)[1/p]$. Let $N_{\widehat{B}^{+}_{\mathrm{st}}(S)}\colon \widehat{B}^{+}_{\mathrm{st}}(S)\to \widehat{B}^{+}_{\mathrm{st}}(S)$ be the monodromy operator induced from $N_{\widehat{A}_{\mathrm{st}}(S)}$. Let $\varphi_{\widehat{B}^{+}_{\mathrm{st}}(S)}\colon \widehat{B}^{+}_{\mathrm{st}}(S)\to \widehat{B}^{+}_{\mathrm{st}}(S)$ be the Frobenius structure induced from $\varphi_{\widehat{A}_{\mathrm{st}}(S)}$.
    \item Set $\widehat{B}_{\mathrm{st}}(S)\coloneqq \widehat{B}^{+}_{\mathrm{st}}(S)[1/t]$. Let $N_{\widehat{B}_{\mathrm{st}}(S)}\colon \widehat{B}_{\mathrm{st}}(S)\to \widehat{B}_{\mathrm{st}}(S)$ be the monodromy operator induced from $N_{\widehat{B}^{+}_{\mathrm{st}}(S)}$. Let $\varphi_{\widehat{B}_{\mathrm{st}}(S)}\colon \widehat{B}_{\mathrm{st}}(S)\to \widehat{B}_{\mathrm{st}}(S)$ be the Frobenius structure induced from $\varphi_{\widehat{B}^{+}_{\mathrm{st}}(S)}$.
    \item Let $A_{\mathrm{st}}(S)\coloneqq A_{\mathrm{crys}}(S)[\mathrm{log}(\frac{p}{[p^{\flat}]})]$ be the $A_{\mathrm{crys}}(S)$-subalgebra of $B^{+}_{\mathrm{dR}}(S)$ generated by $\mathrm{log}(\frac{p}{[p^{\flat}]})$. This is independent of the choice of $p^{\flat}$ because $\mathrm{log}(\frac{p}{[\epsilon p^{\flat}]})=\mathrm{log}(\frac{p}{[p^{\flat}]})-\mathrm{log}([\epsilon])$ and $\mathrm{log}([\epsilon])\in A_{\mathrm{crys}}(S)$. By \cite[Corollary 2.28]{shi22}, $\mathrm{log}(\frac{p}{[p^{\flat}]})$ is transcendental over $A_{\mathrm{crys}}(S)$. Let $N_{A_{\mathrm{st}}(S)}\colon A_{\mathrm{st}}(S)\to A_{\mathrm{st}}(S)$ be the monodromy operator defined as the $A_{\mathrm{crys}}(S)$-linear map $A_{\mathrm{st}}(S)\to A_{\mathrm{st}}(S)$ given by $(\mathrm{log}(\frac{p}{[p^{\flat}]}))^{n}\mapsto n(\mathrm{log}(\frac{p}{[p^{\flat}]}))^{n-1}$ for each $n\geq 1$. Let $\varphi_{A_{\mathrm{st}}(S)}\colon A_{\mathrm{st}}(S)\to A_{\mathrm{st}}(S)$ be the ring map induced from $\varphi_{A_{\mathrm{crys}}(S)}$ and $\mathrm{log}(\frac{p}{[p^{\flat}]})\mapsto p\mathrm{log}(\frac{p}{[p^{\flat}]})$. The maps $N_{A_{\mathrm{st}}(S)}$ and $\varphi_{A_{\mathrm{crys}}(S)}$ are independent of a choice of $p^{\flat}$. The map of $A_{\mathrm{crys}}(S)$-algebras $A_{\mathrm{st}}(S)\to \widehat{A}_{\mathrm{st}}(S)$ mapping $\mathrm{log}(\frac{p}{[p^{\flat}]})$ to $\mathrm{log}(\frac{u}{[p^{\flat}]})$ is an injective map which is independent of a choice of $p^{\flat}$. This injection is compatible with Frobenius structures and monodromy operators and identifies $A_{\mathrm{st}}(S)$ with the subset of $\widehat{A}_{\mathrm{st}}(S)$ consisting of all nilpotent elements for $N_{\widehat{A}_{\mathrm{st}}(S)}$.
    \item Set $B^{+}_{\mathrm{st}}(S)\coloneqq A_{\mathrm{st}}[1/p]$. Let $N_{B^{+}_{\mathrm{st}}(S)}\colon B^{+}_{\mathrm{st}}(S)\to B^{+}_{\mathrm{st}}(S)$ be the monodromy operator induced from $N_{A_{\mathrm{st}}(S)}$, and let $\varphi_{B^{+}_{\mathrm{st}}(S)}\colon B^{+}_{\mathrm{st}}(S)\to B^{+}_{\mathrm{st}}(S)$ be the Frobenius structure induced from $\varphi_{A_{\mathrm{st}}(S)}$. 
    \item Set $B_{\mathrm{st}}(S)\coloneqq B^{+}_{\mathrm{st}}(S)[1/t]$. Let $N_{B_{\mathrm{st}}(S)}\colon B_{\mathrm{st}}(S)\to B_{\mathrm{st}}(S)$ be the monodromy operator induced from $N_{B^{+}_{\mathrm{st}}(S)}$, and let $\varphi_{B_{\mathrm{st}}(S)}\colon B_{\mathrm{st}}(S)\to B_{\mathrm{st}}(S)$ be the Frobenius structure induced from $\varphi_{B^{+}_{\mathrm{st}}(S)}$. .
\end{itemize}  
\end{dfn}

Let $(X,\cM_{X})_{\mathrm{prok\et}}$ be the pro-Kummer-\'{e}tale site of $(X,\cM_{X})$ (\cite[Definition 5.1.2]{dllz23b}) and $(\widehat{\cO},\widehat{\cO}^{+})$ be completed structure sheaves (\cite[Definition 5.4.1]{dllz23b}). For a log affinoid perfectoid $U$ with the associated affinoid perfectoid space $\widehat{U}=\mathrm{Spa}(S,S^{+})$ (\cite[Definition 5.3.1 and Remark 5.3.5]{dllz23b}), there exists a natural morphism of sites
\[
(\mathrm{Spa}(S,S^{+}))_{\mathrm{pro\et}}\to {(X,\cM_{X})_{\mathrm{prok\et}}}_{/U}
\]
inducing an equivalence of the associated topoi (\cite[Lemma 5.3.8]{dllz23b}). The log affinoid perfectoid objects form a basis of $(X,\cM_{X})_{\mathrm{prok\et}}$ by \cite[Proposition 5.3.12]{dllz23b}. 

\begin{lem}
Let $(X,\cM_{X})$ be a locally noetherian fs log adic space over $\mathrm{Spa}(K,\cO_{K})$ and $B\in \{A_{\mathrm{inf}},B_{\mathrm{inf}},B_{\mathrm{st}},\widehat{B}_{\mathrm{st}}\}$ (resp. $B\in \{B^{+}_{\mathrm{dR}},B_{\mathrm{dR}},B_{\mathrm{crys}}\}$). Then 
\[
U\mapsto (B(S^{+})) \ \ (\text{resp.} \  (B(S^{+}),\mathrm{Fil}^{\bullet}(B(S^{+}))))
\]
defines a sheaf (resp. filtered sheaf) on $(X,\cM_{X})_{\mathrm{prok\et}}$, where $U$ is a log affinoid perfectoid in $(X,\cM_{X})_{C,\mathrm{prok\et}}$ and $\widehat{U}=\mathrm{Spa}(S,S^{+})$ is the associated affinoid perfectoid space (in the sense of \cite[Definition 5.3.1 and Remark 5.3.5]{dllz23b}). 
\end{lem}

\begin{proof}
    For the case that $B=A_{\mathrm{inf}},B_{\mathrm{inf}},B^{+}_{\mathrm{dR}},B_{\mathrm{dR}}$, see \cite[Proposition 2.2.4 (1)]{dllz23a}. 
    
    For the case that $B=B_{\mathrm{crys}},\widehat{B}_{\mathrm{st}}$, the same argument as the proof of \cite[Corollary 2.8 (2)]{tt19} works. In their proof, \cite[Lemma 4.10]{sch13} is used as a key input, which we replace with \cite[Theorem 5.4.3]{dllz23b}. In the case that $B=B_{\mathrm{st}}(S)$, the assertion follows from the result for $B_{\mathrm{crys}}(S)$
\end{proof}

\begin{dfn}
    The (filtered) sheaves defined in the above lemma are denoted by $\bA_{\mathrm{inf}}$, $\bB_{\mathrm{inf}}$, $\bB^{+}_{\mathrm{dR}}$, $\bB_{\mathrm{dR}}$, $\bB_{\mathrm{crys}}$, $\bB_{\mathrm{st}}$,$\widehat{\bB}_{\mathrm{st}}$ respectively. 
\end{dfn}

\subsection{Period sheaves with connections for semi-stable log formal schemes}

Let $(\fX,\cM_{\fX})$ be a semi-stable log formal scheme over $\cO_{K}$.

\begin{dfn}\label{def of big log affinoid perfd}
    Let $U$ be a log affinoid perfectoid in $(\fX,\cM_{\fX})_{\eta,\mathrm{prok\et}}$ with the associated affinoid perfectoid space $\widehat{U}\coloneqq\mathrm{Spa}(S,S^{+})$. We say that $U$ is \emph{big} (with respect to the formal model $(\fX,\cM_{\fX})$) if $U\to (\fX,\cM_{\fX})_{\eta}$ factors as
    \[
    U\to (\mathrm{Spf}(\widetilde{R}_{\infty,\alpha}),\cM_{\widetilde{R}_{\infty,\alpha}})_{\eta}\to (\mathrm{Spf}(R),\cM_{R})_{\eta} \to (\fX,\cM_{\fX})_{\eta},
    \]
    where $(\mathrm{Spf}(R),\cM_{R})$ is a small affine log formal scheme with a fixed framing and the morphism $(\mathrm{Spf}(R),\cM_{R})\to (\fX,\cM_{\fX})$ is strict \'{e}tale. In this situation, we define a log structure on $\cM_{S^{+}}$ as the pullback log structure of $\cM_{\widetilde{R}_{\infty,\alpha}}$ by $\mathrm{Spf}(S^{+})\to \mathrm{Spf}(\widetilde{R}_{\infty,\alpha})$. This is independent of the choice of the above factorization. To check this, take another factorization
    \[
    U\to (\mathrm{Spf}(\widetilde{R'}_{\infty,\alpha'}),\cM_{\widetilde{R'}_{\infty,\alpha'}})_{\eta}\to (\mathrm{Spf}(R'),\cM_{R'})_{\eta} \to (\fX,\cM_{\fX})_{\eta},
    \]
    and set
    \[
    (\fV,\cM_{\fV})\coloneqq (\mathrm{Spf}(\widetilde{R}_{\infty,\alpha}),\cM_{\widetilde{R}_{\infty,\alpha}})\times_{(\fX,\cM_{\fX})}^{\mathrm{sat}} (\mathrm{Spf}(\widetilde{R'}_{\infty,\alpha'}),\cM_{\widetilde{R'}_{\infty,\alpha'}}).
    \]
    We have the induced map
    \[
    U\to (\mathrm{Spf}(\widetilde{R}_{\infty,\alpha}),\cM_{\widetilde{R}_{\infty,\alpha}})_{\eta}\times_{(\fX,\cM_{\fX})_{\eta}}^{\mathrm{sat}} (\mathrm{Spf}(\widetilde{R'}_{\infty,\alpha'}),\cM_{\widetilde{R'}_{\infty,\alpha'}})_{\eta} \cong (\fV,\cM_{\fV})_{\eta},
    \]
    and so the natural map $\mathrm{Spf}(S^{+})\to \fX$ factors through $\fV$. Since both of two projection maps from $(\fV,\cM_{\fV})$ are strict by \cite[Lemma A.2]{ino25}, the log structure $\cM_{S^{+}}$ is independent of the choice of a factorization.
    
    By Lemma \ref{lift of div log str}, $\cM_{S^{+}}$ lifts uniquely to a log structure $\cM_{A_{\mathrm{inf}}(S^{+})}$ on $\mathrm{Spf}(A_{\mathrm{inf}}(S^{+}))$, and $\cM_{A_{\mathrm{inf}}(S^{+})}$ is isomorphic to the associated log structure with the prelog ring $\bQ_{\geq 0}\to A_{\mathrm{inf}}(S^{+})$ defined by
    \[
    se_{i}\mapsto 
    \begin{cases}
        [(y_{i}^{s})^{\flat}] \ \ (1\leq i\leq m) \\
        [(z_{i-m}^{s})^{\flat}] \ \ (m+1\leq i\leq r)
    \end{cases}
    \]
    for $s\in \bQ_{\geq 0}$.
    Let $\cM_{A_{\mathrm{crys}}(S^{+})}$ denote the log structure on $\mathrm{Spf}(A_{\mathrm{crys}}(S^{+}))$ obtained as the pullback of $\cM_{A_{\mathrm{inf}}(S^{+})}$. We define a log PD-thickening $(A_{\mathrm{crys},K}(S^{+})\twoheadrightarrow S^{+}/p,\cM_{A_{\mathrm{crys},K}(S^{+})})$ in $ (X_{0},\cM_{X_{0}})_{\mathrm{crys}}$ as the $p$-adic log PD-envelope of the closed immersion 
    \[
    (\mathrm{Spec}(S^{+}/p),\cM_{S^{+}/p})\hookrightarrow (\mathrm{Spf}(A_{\mathrm{crys}}(S^{+})\otimes_{W} \cO_{K}),(\cM_{A_{\mathrm{crys}}(S^{+})}\oplus \cM_{\cO_{K}})^{a}),
    \]
    where $\cM_{S^{+}/p}$ is the restriction of $\cM_{S^{+}}$. When $(\fX,\cM_{\fX})$ is horizontally semi-stable (resp. vertically semi-stable), we have the explicit description
    \begin{align*}
    A_{\mathrm{crys},K}(S^{+})&=(A_{\mathrm{crys}}(S^{+})\otimes_{W} \cO_{K})[\{\frac{([\pi^{\flat}]-\pi)^{n}}{n!}\}_{n\geq 1}]^{\wedge}_{p} \\
    (\text{resp.}~A_{\mathrm{crys},K}(S^{+}) &=(A_{\mathrm{crys}}(S^{+})\otimes_{W} \cO_{K})[\{\frac{(([\pi^{\flat}]/\pi)-1)^{n}}{n!}\}_{n\geq 1}]^{\wedge}_{p}).
    \end{align*}
    Hence, we can define a natural map
    \[
    A_{\mathrm{crys},K}(S^{+})\to B^{+}_{\mathrm{dR}}(S^{+})
    \]
    in the same way as the map $A_{\mathrm{crys}}(S^{+})\to B^{+}_{\mathrm{dR}}(S^{+})$
\end{dfn}

De Rham period sheaves with connections on log adic spaces are introduced in \cite{dllz23a} and denoted by $\cO\bB^{+}_{\mathrm{dR,log}}$ there. These objects are regarded as the log version of de Rham period sheaves defined in \cite{sch13}. 

We give another construction of $\cO\bB^{+}_{\mathrm{dR,log}}$ which is similar to the definition in \cite{sch16} under the setting that a formal model is fixed. Our construction is used in Proposition \ref{two conn over obdr is eq} afterward.

\begin{construction}\label{construction of obdr}
Let $U$ be a big log affinoid perfectoid in $(\fX,\cM_{\fX})_{\eta,\mathrm{prok\et}}$ with the associated affinoid perfectoid space $\widehat{U}=\mathrm{Spa}(S,S^{+})$. Choose a factorization $U\to (\mathrm{Spf}(R),\cM_{R})_{\eta}\to (\fX,\cM_{\fX})_{\eta}$, where $(\mathrm{Spf}(R),\cM_{R})\to (\fX,\cM_{\fX})$ is strict \'{e}tale. The natural surjection $\theta\colon A_{\mathrm{inf}}(S^{+})\twoheadrightarrow S^{+}$ induces a closed immersion 
\[
(\mathrm{Spf}(S^{+}),\cM_{S^{+}})\hookrightarrow (\mathrm{Spf}(R\widehat{\otimes}_{W} A_{\mathrm{inf}}(S^{+})),\cM_{R}\oplus \cM_{A_{\mathrm{inf}}(S^{+})})^{a}.
\]
Let $(\mathrm{Spf}(\cO\bA_{\mathrm{inf}}(U)),\cM_{\cO\bA_{\mathrm{inf}}(U)})$ be its exactification. The surjection $\cO\bA_{\mathrm{inf}}(U)[1/p]\twoheadrightarrow S$ is also denoted by $\theta$. Let $\cO\bB^{+}_{\mathrm{dR}}(U)$ be the $\mathrm{Ker}(\theta)$-adic completion of $\cO\bA_{\mathrm{inf}}(U)[1/p]$. Then $\cO\bB^{+}_{\mathrm{dR}}(U)$ is independent of the choice of $(\mathrm{Spf}(R),\cM_{R})$. We have an induced map $\theta\colon \cO\bB^{+}_{\mathrm{dR}}(U)\twoheadrightarrow S$ and define a filtration $\mathrm{Fil}^{\bullet}(\cO\bB^{+}_{\mathrm{dR}}(U))$ as the $\mathrm{Ker}(\theta)$-adic filtration. Let $(\cO\bB^{+}_{\mathrm{dR}},\mathrm{Fil}^{\bullet}(\cO\bB^{+}_{\mathrm{dR}}))$ be the filtered sheaf associated with the presheaf given by $U\mapsto (\cO\bB^{+}_{\mathrm{dR}}(U),\mathrm{Fil}^{\bullet}(\cO\bB^{+}_{\mathrm{dR}}(U)))$.

Next, we define a connection on $\cO\bB^{+}_{\mathrm{dR}}$. We define a connection
\[
\nabla_{\cO\bB^{+}_{\mathrm{dR}}}\colon \cO\bB^{+}_{\mathrm{dR}}\to \cO\bB^{+}_{\mathrm{dR}}\otimes_{\cO_{\fX_{\eta}}} \Omega^{1}_{(\fX,\cM_{\fX})_{\eta}/K} 
\]
as a unique $\bB^{+}_{\mathrm{dR}}$-linear continuous connection extending the derivation $\cO_{\fX_{\eta}}\to \Omega^{1}_{(\fX,\cM_{\fX})_{\eta}/K}$ such that $\nabla_{\cO\bB^{+}_{\mathrm{dR}}}(m/m')=(m/m')\otimes \mathrm{dlog}(m)$ for each $m\in \Gamma(\fX,\cM_{\fX})$ and $m'\in \Gamma(\mathrm{Spf}(A_{\mathrm{inf}}(S^{+})),\cM_{A_{\mathrm{inf}}(S^{+})})$ with the common image in $\Gamma(\mathrm{Spf}(S^{+}),\cM_{S^{+}})$. By the Leibniz rule, we have
    \[
    \nabla_{\cO\bB^{+}_{\mathrm{dR}}}(\mathrm{Fil}^{n}(\cO\bB^{+}_{\mathrm{dR}}))\subset \mathrm{Fil}^{n-1}(\cO\bB^{+}_{\mathrm{dR}})\otimes_{\cO_{\fX_{\eta}}} \Omega^{1}_{(\fX,\cM_{\fX})_{\eta}/K}
    \]
for $n\geq 1$.

Suppose that $(\fX,\cM_{\fX})=(\mathrm{Spf}(R),\cM_{R})$ is small affine log formal scheme over $\cO_{K}$ with a fixed framing and that there is a factorization $U\to \widetilde{U}_{\infty,\alpha}\to (\fX,\cM_{\fX})_{\eta}$. Then we have an isomorphism $\cO\bA_{\mathrm{inf}}(U)\cong (R\widehat{\otimes}_{W} A_{\mathrm{inf}}(S^{+}))[\{([y_{j}^{\flat}]/y_{j})^{\pm 1}\},\{([z_{k}^{\flat}]/z_{k})^{\pm 1}\}]^{\wedge}_{(p,\xi)}$. The connection
\[
\nabla_{\cO\bB^{+}_{\mathrm{dR}}(U)}\colon \cO\bB^{+}_{\mathrm{dR}}(U)\to \cO\bB^{+}_{\mathrm{dR}}(U)\otimes_{R[1/p]} \Omega^{1}_{(\fX,\cM_{\fX})_{\eta}/K} 
\]
is characterized as a unique $B^{+}_{\mathrm{dR}}(S^{+})$-linear continuous connection extending the derivation $R[1/p]\to \Omega^{1}_{(\fX,\cM_{\fX})_{\eta}/K}$ such that 
\[
\nabla_{\cO\bB^{+}_{\mathrm{dR}}(U)}(y_{j}/[y_{j}^{\flat}])=(y_{j}/[y_{j}^{\flat}])\otimes \mathrm{dlog}(y_{j}), \ \  \nabla_{\cO\bB^{+}_{\mathrm{dR}}(U)}(z_{k}/[z_{k}^{\flat}])=(z_{k}/[z_{k}^{\flat}])\otimes \mathrm{dlog}(z_{k})
\]
for each $j,k$.

Consider the sheaf $\cO\bB^{+}_{\mathrm{dR}}[1/t]$ equipped with a filtration defined by
\[
\mathrm{Fil}^{n}(\cO\bB^{+}_{\mathrm{dR}}[1/t])\coloneqq \sum_{m\geq 0} t^{n-m}\mathrm{Fil}^{m}(\cO\bB^{+}_{\mathrm{dR}})
\]
for $n\in \bZ$. Let $\cO\bB_{\mathrm{dR}}$ denote the completion of $\cO\bB_{\mathrm{dR}}[1/t]$ with respect to the filtration. We define a filtration on $\cO\bB_{\mathrm{dR}}$ by
\[
\mathrm{Fil}^{n}(\cO\bB_{\mathrm{dR}})\coloneqq \varprojlim_{m\geq 0} \mathrm{Fil}^{n}(\cO\bB^{+}_{\mathrm{dR}}[1/t])/\mathrm{Fil}^{n+m}(\cO\bB^{+}_{\mathrm{dR}}[1/t])
\]
for $n\in \bZ$. The connection $\nabla_{\cO\bB^{+}_{\mathrm{dR}}}$ induces a connection
\[
\nabla_{\cO\bB_{\mathrm{dR}}}\colon \cO\bB_{\mathrm{dR}}\to \cO\bB_{\mathrm{dR}}\otimes_{\cO_{\fX_{\eta}}} \Omega^{1}_{(\fX,\cM_{\fX})_{\eta}/K}
\]
satisfying $\nabla_{\cO\bB_{\mathrm{dR}}}(\mathrm{Fil}^{n}(\cO\bB_{\mathrm{dR}}))\subset \mathrm{Fil}^{n-1}(\cO\bB_{\mathrm{dR}})$ for $n\in \bZ$.
\end{construction}

We shall give a local description of $\cO\bB^{+}_{\mathrm{dR}}$. Suppose that $(\fX,\cM_{\fX})=(\mathrm{Spf}(R),\cM_{R})$ is small affine log formal scheme over $\cO_{K}$ with a fixed framing. Consider a surjection $\theta\colon \bB^{+}_{\mathrm{dR}}[X_{1},\dots,X_{s}]\twoheadrightarrow \cO^{\flat}$ induced from the natural surjection $\theta\colon \bB^{+}_{\mathrm{dR}}\twoheadrightarrow \cO^{\flat}$ and $X_{i}\mapsto 1$. Let $\bB^{+}_{\mathrm{dR}}[[X_{1}-1,\dots,X_{s}-1]]$ be the $\mathrm{Ker}(\theta)$-adic completion of $\bB^{+}_{\mathrm{dR}}[X_{1},\dots,X_{s}]$. We have the induced surjection $\bB^{+}_{\mathrm{dR}}[[X_{1}-1,\dots,X_{s}-1]]\twoheadrightarrow \cO^{\flat}$, which is also denoted by $\theta$. Let $\mathrm{Fil}^{\bullet}(\bB^{+}_{\mathrm{dR}}[[X_{1}-1,\dots,X_{s}-1]])$ be the $\mathrm{Ker}(\theta)$-adic filtration of $\bB^{+}_{\mathrm{dR}}[X_{1},\dots,X_{s}]$.

\begin{lem}\label{local description of our obdr}
Let $(\fX,\cM_{\fX})=(\mathrm{Spf}(R),\cM_{R})$ be a small affine log formal scheme over $\cO_{K}$ with a fixed framing. Then there is an isomorphism of $\bB^{+}_{\mathrm{dR}}$-algebras
    \[ \cO\bB^{+}_{\mathrm{dR}}|_{\widetilde{U}_{\infty,\alpha}}\cong \bB^{+}_{\mathrm{dR}}[[X_{1}-1,\dots,X_{s}-1]]|_{\widetilde{U}_{\infty,\alpha}}
    \]
    that is compatible with filtrations. 
\end{lem}

\begin{proof}
    We follow the argument of \cite{sch16}. Define a map of $\bB^{+}_{\mathrm{dR}}|_{\widetilde{U}_{\infty,\alpha}}$-algebras
    \[
    \bB^{+}_{\mathrm{dR}}[[X_{1}-1,\dots,X_{s}-1]]|_{\widetilde{U}_{\infty,\alpha}}\to \cO\bB^{+}_{\mathrm{dR}}|_{\widetilde{U}_{\infty,\alpha}}
    \]
    sending $X_{i}$ to $t_{i}/[t_{i}^{\flat}]$ for each $i$. We shall construct the inverse map. Let $U$ be a log affinoid perfectoid over $\widetilde{U}_{\infty,\alpha}$ with the associated affinoid perfectoid space $\widehat{U}=\mathrm{Spa}(S,S^{+})$. Consider a map of $\cO_{K}$-algebras $R^{0}\to B^{+}_{\mathrm{dR}}(S^{+})[[X_{1}-1,\dots,X_{s}-1]]$ given by
    \[
    x_{i}\mapsto [x_{i}^{\flat}]X_{i} \ (1\leq i\leq l), \ \ \ y_{j}\mapsto [y_{j}^{\flat}]X_{l+j} \ (1\leq j\leq m)
    \]
    \begin{align*}
        z_{k}\mapsto 
    \begin{cases}
        [z_{k}^{\flat}]X_{l+m+k} \ \ &(1\leq k\leq n-1) \\
        \pi\cdot (\prod_{k=1}^{n-1} [z_{k}^{\flat}]X_{l+m+k})^{-1} \ \ &(k=n).
    \end{cases}
    \end{align*}
    Since $R^{0}\to R$ is \'{e}tale, there exists a unique map $R\to B^{+}_{\mathrm{dR}}(S^{+})[[X_{1}-1,\dots,X_{s}-1]]$ commuting the following diagram:
    \[
    \begin{tikzcd}
        R^{0} \ar[d] \ar[r] &  B^{+}_{\mathrm{dR}}(S^{+})[[X_{1}-1,\dots,X_{s}-1]] \ar[d,twoheadrightarrow,"\theta"] \\
        R \ar[r] \ar[ur,dotted] & S.
    \end{tikzcd}
    \]
    This map and the natural map $A_{\mathrm{inf}}(S^{+})\to B^{+}_{\mathrm{dR}}(S^{+})[[X_{1}-1,\dots,X_{s}-1]]$ induce a desired map $\cO\bB^{+}_{\mathrm{dR}}|_{\widetilde{U}_{\infty,\alpha}}\to B^{+}_{\mathrm{dR}}[[X_{1}-1,\dots,X_{s}-1]]|_{\widetilde{U}_{\infty,\alpha}}$.

    It is enough to show that the map constructed above is indeed an inverse map. The composition
    \[
    B^{+}_{\mathrm{dR}}[[X_{1}-1,\dots,X_{s}-1]]|_{\widetilde{U}_{\infty,\alpha}}\to \cO\bB^{+}_{\mathrm{dR}}|_{\widetilde{U}_{\infty,\alpha}}\to B^{+}_{\mathrm{dR}}[[X_{1}-1,\dots,X_{s}-1]]|_{\widetilde{U}_{\infty,\alpha}}
    \]
    is the identity map by definition. Hence, it suffices to show that
    \[
    (B^{+}_{\mathrm{dR}}(S^{+})/\xi^{n})[X_{1},\dots,X_{s}]/I^{n}\to \cO\bB^{+}_{\mathrm{dR}}(U)/\mathrm{Ker}(\theta)^{n}
    \]
    is surjective for $n\geq 1$, where $I$ is the ideal generated by $X_{1}-1,\dots,X_{s}-1$. Let $(\mathrm{Spf}(E_{n}),\cM_{E_{n}})$ denote the $n$-th infinitesimal neighborhood of $(\fX,\cM_{\fX})$ over $\mathrm{Spf}(W)$. Since $(\fX,\cM_{\fX})\to \mathrm{Spf}(W)$ is log smooth after passing to generic fibers, the ring map
    \[
    f\colon R[X_{1},\dots,X_{s}]/I^{n}\to E_{n}
    \]
    defined by $f(X_{i})=(t_{i}\otimes 1)(1\otimes t_{i})^{-1}$ is an isomorphism after inverting $p$. 
    Since $R$ is topologically of finite type over $\cO_{K}$ and the image of the ring map
    \[
    R\to B^{+}_{\mathrm{dR}}(S^{+})[[X_{1}-1,\dots,X_{s}-1]]\twoheadrightarrow S
    \]
    is contained in $S^{+}$, the image of the ring map 
    \[
    R\to B^{+}_{\mathrm{dR}}(S^{+})[[X_{1}-1,\dots,X_{s}-1]]\twoheadrightarrow B^{+}_{\mathrm{dR}}(S^{+})/\xi^{n}
    \]
    is contained in a $p$-complete ring $A_{k}\coloneqq (A_{\mathrm{inf}}(S^{+})/\xi^{n})[\xi/p^{k}](\subset B^{+}_{\mathrm{dR}}(S^{+})/\xi^{n})$ for some $k\geq 1$.
    Taking the $p$-complete base change of $f$ along $R\to A_{k}$, we get a map
    \[
    A_{k}[X_{1},\dots,X_{s}]/I^{n}\to E_{n}\widehat{\otimes}_{R} A_{k}
    \]
    which is an isomorphism up to bounded $p$-power torsions. By composing a natural surjection $E_{n}\widehat{\otimes}_{R} A_{k}\twoheadrightarrow \cO\bA_{\mathrm{inf}}(U)/\mathrm{Ker}(\theta)^{n}$ and inverting $p$, we conclude that the map
    \[
    (B^{+}_{\mathrm{dR}}(S^{+})/\xi^{n})[X_{1},\dots,X_{s}]/I^{n}\to \cO\bB^{+}_{\mathrm{dR}}(U)/\mathrm{Ker}(\theta)^{n}
    \]
    is surjective. 
\end{proof}

In \cite[Definition 2.2.10]{dllz23a}, they introduced de Rham period sheaves $\cO\bB^{+}_{\mathrm{dR,log}}$ on locally noetherian fs log adic space. The following corollary shows that the sheaf $\cO\bB^{+}_{\mathrm{dR,log}}$ on the generic fiber of a semi-stable log formal scheme  coincides with a de Rham period sheaf defined above by using a fixed formal model.

\begin{cor}\label{two obdr}
    For a semi-stable log formal scheme $(\fX,\cM_{\fX})$ over $\cO_{K}$, there exists a natural isomorphism of sheaves on $(\fX,\cM_{\fX})_{\eta,\mathrm{prok\et}}$
    \[
    \cO\bB^{+}_{\mathrm{dR,log}}\cong \cO\bB^{+}_{\mathrm{dR}}
    \]
    that is compatible with filtrations and connections.
\end{cor}

\begin{proof}
    By working \'{e}tale locally on $\fX$, we may assume that $(\fX,\cM_{\fX})$ is a small affine log formal scheme with a fixed framing. Lemma \ref{local description of our obdr} and \cite[Proposition 2.3.15]{dllz23a} give an isomorphism that is compatible with filtrations and connections
    \[ \cO\bB^{+}_{\mathrm{dR,log}}|_{\widetilde{U}_{\infty,\alpha}}\cong \cO\bB^{+}_{\mathrm{dR}}|_{\widetilde{U}_{\infty,\alpha}},
    \]
    which descents to an isomorphism
    \[
    \cO\bB^{+}_{\mathrm{dR,log}}\cong \cO\bB^{+}_{\mathrm{dR}}
    \]
    that is compatible with filtrations and connections. 
\end{proof}

Let $\nu\colon (\fX,\cM_{\fX})_{\eta,\mathrm{prok\et}}\to \fX_{\eta,\mathrm{an}}$ be the natural map of sites. 

\begin{prop}\label{zeroth coh of obdr}
    There is a natural isomorphism
    \[
    \cO_{\fX_{\eta}}\cong \nu_{*}\cO\bB_{\mathrm{dR}}
    \]
    that is compatible with connections.
\end{prop}

\begin{proof}
    Let $D_{\mathrm{dR,log}}$ be the log $p$-adic Riemann-Hilbert functor defined in \cite[(3.2.6) and Theorem 3.2.7 (1)]{dllz23a}. Then we have isomorphisms
    \[
    \cO_{\fX_{\eta}}\cong D_{\mathrm{dR,log}}(\widehat{\bQ}_{p})\cong  \nu_{*}\cO\bB_{\mathrm{dR}}
    \]
    by Corollary \ref{two obdr}.
\end{proof}

Next, we turn to defining crystalline and semi-stable period sheaves with connections. Let $(\fX,\cM_{\fX})$ be a horizontally semi-stable log formal scheme over $W$. 

\begin{construction}[Crystalline period sheaves with connections]\label{construction of crystalline period sheaf}
Let $U$ be a big log affinoid perfectoid in $(\fX,\cM_{\fX})_{\eta,\mathrm{prok\et}}$ with the associated affinoid perfectoid space $\widehat{U}=\mathrm{Spa}(S,S^{+})$. Choose a factorization $U\to (\mathrm{Spf}(R),\cM_{R})_{\eta}\to (\fX,\cM_{\fX})_{\eta}$, where the map $(\mathrm{Spf}(R),\cM_{R})\to (\fX,\cM_{\fX})$ is strict \'{e}tale (this indeed exists because $U$ is big). The natural surjection $\theta\colon A_{\mathrm{inf}}(S^{+})\twoheadrightarrow S^{+}$ induces a closed immersion
        \[
        (\mathrm{Spf}(S^{+}),\cM_{S^{+}})\hookrightarrow (\mathrm{Spf}(R\widehat{\otimes}_{W} A_{\mathrm{inf}}(S^{+})),\cM_{R}\oplus \cM_{A_{\mathrm{inf}}(S^{+})})^{a}.
        \]
Let $(\mathrm{Spf}(\cO\bA_{\mathrm{crys}}(U)),\cM_{\cO\bA_{\mathrm{crys}}(U)})$ be its $p$-adic log PD-envelope. This is independent of the choice of a factorization. We set 
        \begin{align*}
            \cO\bB^{+}_{\mathrm{crys}}(U)&\coloneqq \cO\bA_{\mathrm{crys}}(U)[1/p] \\
            \cO\bB_{\mathrm{crys}}(U)&\coloneqq \cO\bB^{+}_{\mathrm{crys}}(U)[1/\mu].
        \end{align*}
        
Suppose that $(\fX,\cM_{\fX})$ is equipped with a Frobenius lift $\varphi_{(\fX,\cM_{\fX})}$. Let $\varphi_{R}$ denote a unique Frobenius lift on $(\mathrm{Spf}(R),\cM_{R})$ such that the strict \'{e}tale map $(\mathrm{Spf}(R),\cM_{R})\to (\fX,\cM_{\fX})$ is compatible with Frobenius lifts on both sides. We let $\varphi_{\cO\bA_{\mathrm{crys}}(U)}$ denote a Frobenius lift on $(\mathrm{Spf}(\cO\bA_{\mathrm{crys}}(U)),\cM_{\cO\bA_{\mathrm{crys}}(U)})$ induced from the product Frobenius $\varphi_{R}\times \varphi_{A_{\mathrm{crys}}(S^{+})}$ on $(\mathrm{Spf}(R\widehat{\otimes}_{W} A_{\mathrm{crys}}(S^{+})),\cM_{R}\oplus \cM_{A_{\mathrm{crys}}(S^{+})})^{a}$. Let $\varphi_{\cO\bB^{+}_{\mathrm{crys}}(U)}$ and $\varphi_{\cO\bB_{\mathrm{crys}}(U)}$ be the induced ring maps on respective rings. 
    
We define a connection 
        \[
        \nabla_{\cO\bA_{\mathrm{crys}}(U)}\colon \cO\bA_{\mathrm{crys}}(U)\to \cO\bA_{\mathrm{crys}}(U)\otimes_{R} \Omega^{1}_{(\mathrm{Spf}(R),\cM_{R})/W}
        \]
(note that the right hand side is independent of the choice of $(\mathrm{Spf}(R),\cM_{R})$) as a unique $A_{\mathrm{crys}}(S^{+})$-linear continuous connection extending the derivation 
    \[
    R\to \Omega^{1}_{(\mathrm{Spf}(R),\cM_{R})/W}
    \]
    satisfying the following condition:
    \begin{itemize}
        \item Suppose that we have a factorization $U\to (\mathrm{Spf}(R_{1}),\cM_{R_{1}})_{\eta}\to (\fX,\cM_{\fX})_{\eta}$, where $(\mathrm{Spf}(R_{1}),\cM_{R_{1}})\to (\fX,\cM_{\fX})$ is strict \'{e}tale. Then the equation 
        \[
        \nabla_{\cO\bA_{\mathrm{crys}}(U)}(m/m')=(m/m')\otimes \mathrm{dlog}(m)
        \]
        holds for each $m\in \Gamma(\mathrm{Spf}(R_{1}),\cM_{R_{1}})$ and $m'\in \Gamma(\mathrm{Spf}(A_{\mathrm{crys}}(S^{+})),\cM_{A_{\mathrm{crys}}(S^{+})})$ with the common image in $\Gamma(\mathrm{Spf}(S^{+}),\cM_{S^{+}})$. Here, we used a natural identification
        \begin{align*}
            &\cO\bA_{\mathrm{crys}}(U)\otimes_{R} \Omega^{1}_{(\mathrm{Spf}(R),\cM_{R})/(\mathrm{Spf}(W),\cM_{W})} \\
            \cong & \cO\bA_{\mathrm{crys}}(U)\otimes_{R_{1}} \Omega^{1}_{(\mathrm{Spf}(R_{1}),\cM_{R_{1}})/(\mathrm{Spf}(W),\cM_{W})}.
        \end{align*}
    \end{itemize}
    We also have induced connections $\nabla_{\cO\bB^{+}_{\mathrm{crys}}(U)}$ and $\nabla_{\cO\bB_{\mathrm{crys}}(U)}$ on respective rings.

    Let $\cO\bB_{\mathrm{crys}}$ be the  sheaf on $(\fX,\cM_{\fX})_{\eta,\mathrm{prok\et}}$ associated with the presheaf given by $U\mapsto \cO\bB_{\mathrm{crys}}(U)$. Let $\varphi_{\cO\bB_{\mathrm{crys}}}$ be a ring map on $\cO\bB_{\mathrm{crys}}$ induced from $\varphi_{\cO\bB_{\mathrm{crys}}(U)}$ for each $U$. Connections $\nabla_{\cO\bB_{\mathrm{crys}}(U)}$ for each $U$ induces a connection
        \[
        \nabla_{\cO\bB_{\mathrm{crys}}}\colon \cO\bB_{\mathrm{crys}}\to \cO\bB_{\mathrm{crys}}\otimes_{\cO_{\fX_{\eta}}} \Omega^{1}_{(\fX,\cM_{\fX})_{\eta}/K_{0}}.
        \]
\end{construction}

\begin{construction}[Semi-stable period sheaves with connections]\label{construction of semistable period sheaf}

Let $U$ be a big log affinoid perfectoid with the associated affinoid perfectoid space $\widehat{U}=\mathrm{Spa}(S,S^{+})$. We set
        \begin{align*}
            \cO\bA_{\mathrm{st}}(U)&\coloneqq \cO\bA_{\mathrm{crys}}(U)\otimes_{A_{\mathrm{crys}}(S^{+})} A_{\mathrm{st}}(S^{+}) \\
            \cO\bB^{+}_{\mathrm{st}}(U)&\coloneqq \cO\bA_{\mathrm{st}}(U)[1/p] \\
            \cO\bB_{\mathrm{st}}(U)&\coloneqq \cO\bB^{+}_{\mathrm{st}}(U)[1/t].
        \end{align*}

When $(\fX,\cM_{\fX})$ is equipped with a Frobenius lift $\varphi_{(\fX,\cM_{\fX})}$, we define a Frobenius structure $\varphi_{\cO\bA_{\mathrm{st}}(U)}$ on $\cO\bA_{\mathrm{st}}(U)$ by
\[
\varphi_{\cO\bA_{\mathrm{st}}(U)}\coloneqq \varphi_{\cO\bA_{\mathrm{crys}}(U)}\otimes \varphi_{A_{\mathrm{st}}(S^{+})}.
\]
Similarly, we can define $\varphi_{\cO\bB^{+}_{\mathrm{st}}(U)}$ and $\varphi_{\cO\bB_{\mathrm{st}}(U)}$.

Let $N_{\cO\bA_{\mathrm{st}}(U)}$ be a monodromy operator $\mathrm{id}_{\cO\bA_{\mathrm{crys}}(U)}\otimes N_{A_{\mathrm{st}}(S^{+})}$ on $\cO\bA_{\mathrm{st}}(U)$. Let $\nabla_{\cO\bA_{\mathrm{st}}(U)}$ be a connection $\nabla_{\cO\bA_{\mathrm{crys}}(U)}\otimes \mathrm{id}_{A_{\mathrm{st}}(S^{+})}$. Similarly, we can define $N_{\cO\bB^{+}_{\mathrm{st}}(U)}$, $N_{\cO\bB_{\mathrm{st}}(U)}$, $\nabla_{\cO\bB^{+}_{\mathrm{st}}(U)}$, and $\nabla_{\cO\bB_{\mathrm{st}}(U)}$.
        
Let $\cO\bB_{\mathrm{st}}$ be the sheaf on $(\fX,\cM_{\fX})_{\eta,\mathrm{prok\et}}$ associated with the presheaf given by $U\mapsto \cO\bB_{\mathrm{st}}(U)$, which is equipped with a Frobenius structure $\varphi_{\cO\bB_{\mathrm{st}}}$ and a monodromy operator $N_{\cO\bB_{\mathrm{st}}}$ on $\cO\bB_{\mathrm{st}}$ induced from $\varphi_{\cO\bB_{\mathrm{st}}(U)}$ and $N_{\cO\bB_{\mathrm{st}}(U)}$. We define a connection 
\[
\nabla_{\cO\bB_{\mathrm{st}}}\colon \cO\bB_{\mathrm{st}}\to \cO\bB_{\mathrm{st}}\otimes_{R[1/p]} \Omega^{1}_{(\fX,\cM_{\fX})_{\eta}/K_{0}}
\]
by $\nabla_{\cO\bB_{\mathrm{st}}}\coloneqq \nabla_{\cO\bB_{\mathrm{crys}}}\otimes_{\bB_{\mathrm{crys}}} \mathrm{id}_{\bB_{\mathrm{st}}}$.
\end{construction}

\begin{construction}[cf.~{\cite[2.3.4]{ai12}}]\label{hat semistable period sheaf}
Let $U\in (\fX,\cM_{\fX})_{\eta,\mathrm{prok\et}}$ be a big log affinoid perfectoid with the associated affinoid perfectoid space $\widehat{U}=\mathrm{Spa}(S,S^{+})$. Choose a factorization $U\to (\mathrm{Spf}(R),\cM_{R})_{\eta}\to (\fX,\cM_{\fX})_{\eta}$, where $(\mathrm{Spf}(R),\cM_{R})\to (\fX,\cM_{\fX})$ is strict \'{e}tale. The object $U$ is also big with respect to a semi-stable formal model $(\fX,\cM_{\fX}\oplus p^{\bN})^{a}$ by definition. We consider the log structure $\cM_{S^{+}}$ and $\cM_{A_{\mathrm{crys}}(S^{+})}$ defined by the bigness with respect to $(\fX,\cM_{\fX}\oplus p^{\bN})^{a}$. There is a natural strict closed immersion
\[
(\mathrm{Spf}(S^{+}),\cM_{S^{+}})\hookrightarrow (\mathrm{Spf}(A_{\mathrm{crys}}(S^{+})),\cM_{A_{\mathrm{crys}}(S^{+})}).
\]
We set $(\mathrm{Spf}(R\langle u\rangle),\cM_{R\langle u\rangle})\coloneqq (\mathrm{Spf}(R\langle u\rangle),\cM_{R}\oplus u^{\bN})^{a}$. We have a map
\[
(\mathrm{Spf}(S^{+}),\cM_{S^{+}})\to (\mathrm{Spf}(R\langle u\rangle),\cM_{R\langle u\rangle})
\]
given by the structure map $R\to S^{+}$ and $u\mapsto p$. Then these two maps induce a closed immersion  
\[
(\mathrm{Spf}(S^{+}),\cM_{S^{+}})\hookrightarrow (\mathrm{Spf}(A_{\mathrm{crys}}(S^{+})\widehat{\otimes}_{W} R\langle u\rangle),\cM_{A_{\mathrm{crys}}(S^{+})}\oplus \cM_{R\langle u\rangle})^{a}.
\]
Let $(\mathrm{Spf}(\cO\widehat{\bA}_{\mathrm{st}}(U)),\cM_{\cO\widehat{\bA}_{\mathrm{st}}(U)})$ denote its $p$-completed log PD-envelope, which is independent of the choice of $(\mathrm{Spf}(R),\cM_{R})$. 

By the universal property of PD-envelopes, there is a natural isomorphism
\[
\cO\widehat{\bA}_{\mathrm{st}}(U)\cong \cO\bA_{\mathrm{crys}}(U)\widehat{\otimes}_{A_{\mathrm{crys}}(S^{+})} \widehat{A}_{\mathrm{st}}(S^{+}).
\]
When $(\fX,\cM_{\fX})$ is equipped with a Frobenius lift $\varphi_{(\fX,\cM_{\fX})}$, we define a Frobenius lift $\varphi_{\cO\widehat{\bA}_{\mathrm{st}}(U)}$ on $\cO\widehat{\bA}_{\mathrm{st}}(U)$ as the product Frobenius $\varphi_{\cO\bA_{\mathrm{crys}}(U)}\otimes \varphi_{\widehat{A}_{\mathrm{st}}(U)}$. Let $N_{\cO\widehat{\bA}_{\mathrm{st}}(U)}$ be the monodromy operator on $\cO\widehat{\bA}_{\mathrm{st}}(U)$ given by $\mathrm{id}_{\cO\bA_{\mathrm{crys}}}\otimes N_{\widehat{A}_{\mathrm{st}}(U)}$.

We set 
\begin{align*}
    \cO\widehat{\bB}^{+}_{\mathrm{st}}(U)&\coloneqq \cO\widehat{\bA}_{\mathrm{st}}(U)[1/p] \\
    \cO\widehat{\bB}_{\mathrm{st}}(U)&\coloneqq \cO\widehat{\bB}^{+}_{\mathrm{st}}(U)[1/t].
\end{align*}
Let $\varphi_{\cO\widehat{\bB}^{+}_{\mathrm{st}}(U)}$, $\varphi_{\cO\widehat{\bB}_{\mathrm{st}}(U)}$, $N_{\cO\widehat{\bB}^{+}_{\mathrm{st}}(U)}$, $N_{\cO\widehat{\bB}_{\mathrm{st}}(U)}$ be the induced Frobenius structures and monodromy operators on respective rings.

Let $\cO\widehat{\bB}_{\mathrm{st}}$ be the sheaf on $(\fX,\cM_{\fX})_{\eta,\mathrm{prok\et}}$ associated with the presheaf given by $U\to \cO\widehat{\bB}_{\mathrm{st}}(U)$. This sheaf is equipped with the induced Frobenius structure and the monodromy operator, denoted by $\varphi_{\cO\widehat{\bB}_{\mathrm{st}}}$ and $N_{\cO\widehat{\bB}_{\mathrm{st}}}$.
\end{construction}

We shall give a local description of period sheaves $\cO\bB_{\mathrm{crys}}$, $\cO\bB_{\mathrm{st}}$, and $\cO\widehat{\bB}_{\mathrm{st}}$. We define a sheaf $\bB_{\mathrm{crys}}\{X_{1}-1,\dots,X_{s}-1\}$ on $(\fX,\cM_{\fX})_{\eta,\mathrm{prok\et}}$ by
\[
U\mapsto A_{\mathrm{crys}}(S^{+})\{X_{1}-1,\dots,X_{s}-1\}[1/pt]
\]
for a log affinoid perfectoid $U$ with the associated affinoid perfectoid $\widehat{U}=\mathrm{Spa}(S,S^{+})$. We define $\bB_{\mathrm{crys}}\{X_{1}-1,\dots,X_{s}-1,Y-1\}$ in the same way. Furthermore, we define the sheaf $\bB_{\mathrm{crys}}\{X_{1}-1,\dots,X_{s}-1\}[\mathrm{log}(Y)]$ on  $(\fX,\cM_{\fX})_{\eta,\mathrm{prok\et}}$ by
\[
U\mapsto A_{\mathrm{crys}}(S^{+})\{X_{1}-1,\dots,X_{s}-1\}[\mathrm{log}(Y)][1/pt],
\]
where $\mathrm{log}(Y)$ is regarded as an independent variable.

\begin{lem}\label{local description of obcrys}
Let $(\fX,\cM_{\fX})=(\mathrm{Spf}(R),\cM_{R})$ be a small affine log formal scheme over $W$ with a fixed framing such that $n=0$. Then there is the following commutative diagram of sheaves on ${(\fX,\cM_{\fX})_{\eta,\mathrm{prok\et}}}_{\slash \widetilde{U}_{\infty,\alpha}}$: 
\[
\begin{tikzcd}
\cO\bB_{\mathrm{crys}}|_{\widetilde{U}_{\infty,\alpha}} \ar[r,"\sim"] \ar[d,hook] & \bB_{\mathrm{crys}}\{X_{1}-1,\dots,X_{s}-1\}|_{\widetilde{U}_{\infty,\alpha}} \ar[d,hook] \\
\cO\bB_{\mathrm{st}}|_{\widetilde{U}_{\infty,\alpha}} \ar[r,"\sim"] \ar[d,hook] & \bB_{\mathrm{crys}}\{X_{1}-1,\dots,X_{s}-1\}[\mathrm{log}(Y)]|_{\widetilde{U}_{\infty,\alpha}} \ar[d,hook] \\
\cO\widehat{\bB}_{\mathrm{st}}|_{\widetilde{U}_{\infty,\alpha}} \ar[r,"\sim"] & \bB_{\mathrm{crys}}\{X_{1}-1,\dots,X_{s}-1,Y-1\}|_{\widetilde{U}_{\infty,\alpha}}.
\end{tikzcd}
\]
\end{lem}

\begin{proof}
    Since we have isomorphisms 
    \[ 
    \cO\bB_{\mathrm{st}}\cong \cO\bB_{\mathrm{crys}}\otimes_{\bB_{\mathrm{crys}}} \bB_{\mathrm{st}}, \ \ \ \cO\widehat{\bB}_{\mathrm{st}}\cong \cO\bB_{\mathrm{crys}}\otimes_{\bB_{\mathrm{crys}}} \widehat{\bB}_{\mathrm{st}},
    \]
    it suffices to construct the upper horizontal isomorphism. Let $U$ be a big log affinoid perfectoid in ${(\fX,\cM_{\fX})_{\eta,\mathrm{prok\et}}}_{\slash \widetilde{U}_{\infty,\alpha}}$ with the associated affinoid perfectoid space $\widehat{U}=\mathrm{Spa}(S,S^{+})$. Define a map of $A_{\mathrm{crys}}(S^{+})$-algebras
    \begin{equation}
         A_{\mathrm{crys}}(S^{+})\{X_{1}-1,\dots,X_{s}-1\}\to \cO\bA_{\mathrm{crys}}(U) \tag{1}
    \end{equation}
    sending $X_{i}$ to $t_{i}/[t_{i}^{\flat}]$ for each $i$. We shall construct the inverse map. Consider a map of $W$-algebras $R^{0}\to A_{\mathrm{crys}}(S^{+})\{X_{1}-1,\dots,X_{s}-1\}$ given by $t_{i}\mapsto [t_{i}^{\flat}]X_{i}$ for each $i$.
    
    Since $R^{0}\to R$ is $p$-completely \'{e}tale, there exists a unique ring map $R\to A_{\mathrm{crys}}(S^{+})\{X_{1}-1,\dots,X_{s}-1\}$ fitting into the following commutative diagram:
    \[
    \begin{tikzcd}
        R^{0} \ar[r] \ar[d] & A_{\mathrm{crys}}(S^{+})\{X_{1}-1,\dots,X_{s}-1\} \ar[d] \\
        R \ar[r] \ar[ur,dotted] & S^{+}.
    \end{tikzcd}
    \]
    This map induces a map of $A_{\mathrm{crys}}(S^{+})$-algebras 
    \[
    R\widehat{\otimes}_{W} A_{\mathrm{crys}}(S^{+})\to  A_{\mathrm{crys}}(S^{+})\{X_{1}-1,\dots,X_{s}-1\},
    \]
    which induces a ring map
    \begin{equation}
        (R\widehat{\otimes}_{W} A_{\mathrm{crys}}(S^{+}))[t_{1}/[t_{1}^{\flat}],\dots,t_{s}/[t_{s}^{\flat}]]^{\wedge}_{p}\to  A_{\mathrm{crys}}(S^{+})\{X_{1}-1,\dots,X_{s}-1\} \tag{2}
    \end{equation}
    sending $t_{i}/[t_{i}^{\flat}]$ to $X_{i}$ for each $i$. Notice that the left hand side of the map $(2)$ is nothing but the exactification of the closed immersion
    \[
    (\mathrm{Spf}(S^{+}),\cM_{S^{+}})\hookrightarrow (\mathrm{Spf}(R\widehat{\otimes}_{W} A_{\mathrm{crys}}(S^{+})),\cM_{R}\oplus \cM_{A_{\mathrm{crys}}(S^{+})})^{a}.
    \]
    Since the ring map $(2)$ is compatible with surjections to $S^{+}$, we obtain a ring map
    \begin{equation}
        \cO\bA_{\mathrm{crys}}(U)\to A_{\mathrm{crys}}(S^{+})\{X_{1}-1,\dots,X_{s}-1\} \tag{3}
    \end{equation}
    by taking PD-envelopes.

    It is enough to prove the map $(3)$ is indeed the inverse of the map $(1)$. The composition
    \[
    A_{\mathrm{crys}}(S^{+})\{X_{1}-1,\dots,X_{s}-1\}\to \cO\bA_{\mathrm{crys}}(U)\to A_{\mathrm{crys}}(S^{+})\{X_{1}-1,\dots,X_{s}-1\}
    \]
    is the identity map by definition. Hence, it suffices to show that the map $\cO\bA_{\mathrm{crys}}(U)\to A_{\mathrm{crys}}(S^{+})\{X_{1}-1,\dots,X_{s}-1\}$ is injective. Combining with isomorphisms in Lemma \ref{local description of our obdr}, we get the following commutative diagram.
    \[
    \begin{tikzcd}
        A_{\mathrm{crys}}(S^{+})\{X_{1}-1,\dots,X_{s}-1\} \ar[r] \ar[d,hook] & \cO\bA_{\mathrm{crys}}(U) \ar[r] \ar[d] & A_{\mathrm{crys}}(S^{+})\{X_{1}-1,\dots,X_{s}-1\} \ar[d,hook] \\
        B^{+}_{\mathrm{dR}}(S^{+})[[X_{1}-1,\dots,X_{s}-1]] \ar[r,"\sim"] & \cO\bB^{+}_{\mathrm{dR}}(U) \ar[r,"\sim"] & B^{+}_{\mathrm{dR}}(S^{+})[[X_{1}-1,\dots,X_{s}-1]]
    \end{tikzcd}
    \]
    The middle vertical map is an isomorphism by Lemma \ref{genaral injectivity from Acrys to Bdr} below. Therefore, we conclude that the upper right map is injective, and so the map
    \[
    A_{\mathrm{crys}}(S^{+})\{X_{1}-1,\dots,X_{s}-1\}\to \cO\bA_{\mathrm{crys}}(U)
    \]
    is an isomorphism. By inverting $\mu$, we obtain the upper horizontal isomorphism in the diagram in the assertion.
\end{proof}

\begin{lem}\label{genaral injectivity from Acrys to Bdr}
    Let $A$ be a $p$-torsion free $p$-complete ring and $I=(f_{1},\dots,f_{r})$ be an ideal of $A$. Suppose that $(\bar{f_{1}},\dots,\bar{f_{r}})$ is a Koszul-regular sequence in $A/p$ and that $A/I$ is $p$-torsion free. Let $D$ denote the $p$-competed PD-envelope of $A\twoheadrightarrow A/I$ and $B$ denote the $I$-adic completion of $A[1/p]$. Then there is a natural injective map $D\to B$. 
\end{lem}

\begin{proof}
We follow the argument of \cite[Proposition 5.36]{ck19}. Let $C\coloneqq A[\{f_{i}^{[n]}=\frac{f_{i}^{n}}{n!}\}_{1\leq i\leq r,n\geq 0}]\subset A[1/p]$. By \cite[Lemma 2.38]{bs22}, we have $D\cong C^{\wedge}$, where $\wedge$ means the (classical) $p$-adic completion. Since a Koszul-regular sequence is quasi-regular and $A/I$ is $p$-torsion free, the $I^{n}/I^{n+1}$ is also $p$-torsion free for each $n\geq 0$. Hence, $A/I^{n}$ is also $p$-torsion free. Since the image of the natural map $C\hookrightarrow A[1/p]\twoheadrightarrow (A/I^{n})[1/p]$ is contained an open bounded subring, this extends to a map $D\to (A/I^{n})[1/p]$. Varying $n\geq 0$, we get a map $f\colon D\to B$.

We consider filtrations on each ring. Let $\mathrm{Fil}^{\bullet}(C)$ be the PD-adic filtration on $C$. In other words, $\mathrm{Fil}^{n}(C)$ is the ideal generated by $f_{i}^{[m]}$ for $1\leq i\leq r$ and $m\geq n$. We equip $D$ with the filtration $\mathrm{Fil}^{\bullet}(D)$ such that $\mathrm{Fil}^{n}(D)$ is the closure of $\mathrm{Fil}^{n}(C)$. Let $\mathrm{Fil}^{\bullet}(B)$ be the $IB$-adic filtration on $B$. Suppose that $\displaystyle \sum_{\substack{\underline{m}=(m_{1},\dots,m_{r}) \\ m_{1}+\dots+m_{r}=n}} a_{\underline{m}}f^{[\underline{m}]}\in \mathrm{Fil}^{n+1}(C)$, where $a_{\underline{m}}\in A$ and $f^{[\underline{m}]}\coloneqq \prod_{i=1}^{r}f_{i}^{[m_{i}]}$. Since $(f_{1},\dots,f_{r})$ is a quasi-regular sequence in $A[1/p]$, we have $a_{\underline{m}}\in I[1/p]\cap A=I$ for each $\underline{m}$. Therefore, we see that, for each $n\geq 0$,  there is a natural isomorphism
\[
\bigoplus_{\substack{\underline{m}=(m_{1},\dots,m_{r}) \\ m_{1}+\dots+m_{r}=n}} (A/I)\cdot f^{[\underline{m}]}\cong \mathrm{gr}^{n}(C)
\]
and the natural map $\mathrm{gr}^{n}(C)\to \mathrm{gr}^{n}(B)$ is injective. In particular, $\mathrm{gr}^{n}(C)$ and $C/\mathrm{Fil}^{n}(C)$ are $p$-torsion free and $p$-complete. Hence, the natural map $\mathrm{Fil}^{n}(C)/p^{m}\to C/p^{m}$ is injective for each $n,m\geq 0$. Taking limits with respect to $m\geq 0$, we get an injection $\mathrm{Fil}^{n}(C)^{\wedge}\to D$, which gives $\mathrm{Fil}^{n}(C)^{\wedge}\isom \mathrm{Fil}^{n}(D)$. Furthermore, by taking the $p$-completion of the exact sequence
\[
0\to \mathrm{Fil}^{n+1}(C)\to \mathrm{Fil}^{n}(C)\to \mathrm{gr}^{n}(C)\to 0,
\]
we obtain an exact sequence
\[
0\to \mathrm{Fil}^{n+1}(D)\to \mathrm{Fil}^{n}(D)\to \mathrm{gr}^{n}(C)\to 0.
\]
In other words, the natural map $\mathrm{gr}^{n}(C)\to \mathrm{gr}^{n}(D)$ is an isomorphism. Since $\mathrm{gr}^{n}(C)\to \mathrm{gr}^{n}(B)$ is injective, the map $\mathrm{gr}^{n}(D)\to \mathrm{gr}^{n}(B)$ is also injective.

Since $\mathrm{Fil}^{\bullet}(B)$ is separated, it is enough to prove that $\mathrm{Fil}^{\bullet}(D)$ is separated. By taking the mod-$p$ reduction of the pushout square in \cite[Lemma 2.38]{bs22}, we get an isomorphism
\[
C/p=D/p\cong A/(p,f_{1}^{p},\dots,f_{r}^{p})[\{X_{i,m}\}_{1\leq i\leq r,m\geq 1}]/(X_{i,m}^{p})_{1\leq i\leq r,m\geq 1}
\]
uniquely determined by $f_{i}^{[p^{m}]}\mapsto X_{i,m}$. From this, we see that $\mathrm{Fil}^{\bullet}(D)/p$ is a separated filtration on $D/p$, which implies that $\mathrm{Fil}^{\bullet}(D)$ is separated due to the $p$-torsion freeness of $D/\mathrm{Fil}^{n}(D)$. This proves the claim. \end{proof}

\begin{cor}\label{property of crys and semist period sheaf}
For period sheaves on $(\fX,\cM_{\fX})_{\eta,\mathrm{prok\et}}$, the following statements hold.
    \begin{enumerate}
        \item There is the following commutative diagram:
        \[
        \begin{tikzcd}
        \bB_{\mathrm{crys}} \ar[r,"\sim"] \ar[d,hook] & \cO\bB_{\mathrm{crys}}^{\nabla=0} \ar[d,hook] \\
        \bB_{\mathrm{st}} \ar[r,"\sim"] \ar[d,hook] & \cO\bB_{\mathrm{st}}^{\nabla=0}  \ar[d,hook] \\
        \widehat{\bB}_{\mathrm{st}} \ar[r,"\sim"] & \cO\widehat{\bB}_{\mathrm{st}}^{\nabla=0}.
        \end{tikzcd}
        \]
        \item There is a natural isomorphism
        \[
        \cO\bB_{\mathrm{crys}}\isom \cO\bB_{\mathrm{st}}^{N=0}.
        \]
        \item The sheaf $\cO\bB_{\mathrm{st}}$ is the subsheaf of $\cO\widehat{\bB}_{\mathrm{st}}$ consisting of nilpotent sections for $N_{\cO\widehat{\bB}_{\mathrm{st}}}$.
    \end{enumerate}
\end{cor}

\begin{proof}
    Since all assertions can be checked locally on $(\fX,\cM_{\fX})_{\mathrm{prok\et}}$, they follow from Lemma \ref{local description of obcrys} and its proof.
\end{proof}

We can also consider crystalline and semi-stable period sheaves with connections for horizontally semi-stable log formal scheme over a not absolutely unramified base by using a fixed model over $W$.

\begin{dfn}
Suppose that $(\fX,\cM_{\fX})$ admit a fixed unramified model $(\fX_{0},\cM_{\fX_{0}})$. We define period sheaves $\cO\bB_{\mathrm{crys},(\fX_{0},\cM_{\fX_{0}})}$ (resp. $\cO\bB_{\mathrm{st},(\fX_{0},\cM_{\fX_{0}})}$) on $(\fX,\cM_{\fX})_{\eta,\mathrm{prok\et}}$ equipped connections $\nabla_{\cO\bB_{\mathrm{crys},(\fX_{0},\cM_{\fX_{0}})}}$ (resp. $\nabla_{\cO\bB_{\mathrm{st},(\fX_{0},\cM_{\fX_{0}})}}$) as the pullback of $(\cO\bB_{\mathrm{crys}},\nabla_{\cO\bB_{\mathrm{crys}}})$ (resp. $(\cO\bB_{\mathrm{st}},\nabla_{\cO\bB_{\mathrm{st}}})$) along a strict \'{e}tale map $(\fX,\cM_{\fX})_{\eta}\to (\fX_{0},\cM_{\fX_{0}})_{\eta}$. When $(\fX_{0},\cM_{\fX_{0}})$ is equipped with a Frobenius lift, $\cO\bB_{\mathrm{crys},(\fX_{0},\cM_{\fX_{0}})}$ and $\cO\bB_{\mathrm{st},(\fX_{0},\cM_{\fX_{0}})}$ are also equipped induced Frobenius structures, denoted by $\varphi_{\cO\bB_{\mathrm{crys}}}$ and $\varphi_{\cO\bB_{\mathrm{st}}}$ respectively.
\end{dfn}

\begin{prop}\label{zeroth coh of obcrys}
Let $(\fX,\cM_{\fX})$ be a horizontally semi-stable log formal scheme over $\cO_{K}$ with an unramified model $(\fX_{0},\cM_{\fX_{0}})$. Let $\mu\colon (\fX,\cM_{\fX})_{\eta,\mathrm{prok\et}}\to \fX_{0,\eta,\et}$ be the natural maps of sites. Then there are natural isomorphisms
    \begin{align*}
        \cO_{\fX_{0}}[1/p]\cong \mu_{*}\cO\bB_{\mathrm{crys},(\fX_{0},\cM_{\fX_{0}})}\cong \mu_{*}\cO\bB_{\mathrm{st},(\fX_{0},\cM_{\fX_{0}})}
    \end{align*}
    that are compatible with connections. Furthermore, when $(\fX_{0},\cM_{\fX_{0}})$ is equipped with a Frobenius lift, the above isomorphisms are compatible with Frobenius structures.
\end{prop}

\begin{proof}
    By Lemma \ref{local description of our obdr}, Proposition \ref{zeroth coh of obdr}, and Lemma \ref{local description of obcrys}, we obtain injective maps
    \[
    \cO_{\fX_{0}}[1/p]\hookrightarrow \mu_{*}\cO\bB_{\mathrm{crys},(\fX_{0},\cM_{\fX_{0}})}\hookrightarrow \mu_{*}\cO\bB_{\mathrm{st},(\fX_{0},\cM_{\fX_{0}})}\hookrightarrow \mu_{*}\cO\bB_{\mathrm{dR}}.
    \]
    Taking tensor products with $K$, we get maps
    \begin{equation}
        \cO_{\fX_{0}}\otimes_{W} K \hookrightarrow \mu_{*}\cO\bB_{\mathrm{crys},(\fX_{0},\cM_{\fX_{0}})}\otimes_{K_{0}} K \hookrightarrow \mu_{*}\cO\bB_{\mathrm{st},(\fX_{0},\cM_{\fX_{0}})}\otimes_{K_{0}} K \to \mu_{*}\cO\bB_{\mathrm{dR}} \tag{1}
    \end{equation}
    Since the natural map $\bB_{\mathrm{st}}\otimes_{K_{0}} K\to \bB_{\mathrm{dR}}$ is injective (for example, \cite[Proposition 2.27]{shi22}), it follows from Lemma \ref{local description of our obdr} and Lemma \ref{local description of obcrys} that the map
    \[
    \cO\bB_{\mathrm{st},(\fX_{0},\cM_{\fX_{0}})}\otimes_{K_{0}} K\to \cO\bB_{\mathrm{dR}}
    \]
    is injective. Hence, the map
    \[
    \mu_{*}\cO\bB_{\mathrm{st},(\fX_{0},\cM_{\fX_{0}})}\otimes_{K_{0}} K\to \mu_{*}\cO\bB_{\mathrm{dR}}
    \]
    is also injective. Therefore, Proposition \ref{zeroth coh of obdr} implies that every map in the diagram $(1)$ is an isomorphism. This proves the assertion.
\end{proof}

\subsection{Association for local systems on log adic spaces}

In this subsection, we define the notion of semi-stable $\bZ_{p}$-local systems on $(\fX,\cM_{\fX})_{\eta}$ for a semi-stable log formal scheme $(\fX,\cM_{\fX})$ and study their properties. The most part of this subsection follows the argument of  \cite{gr24}.

\begin{construction}\label{bcrys mod ass to isocrystal}
    Let $\bB\in\{\bB_{\mathrm{crys}},\bB_{\mathrm{dR}}^{+},\bB_{\mathrm{dR}}\}$. For an isocrystal $\cE$ on $(X_{0},\cM_{X_{0}})$, we define $\bB$-module $\bB(\cE)$ on $(\fX,\cM_{\fX})_{\eta,\mathrm{prok\et}}$ by
    \[
    \bB(\cE)(U)\coloneqq \cE(A_{\mathrm{crys}}(S^{+})\twoheadrightarrow S^{+}/p,\cM_{A_{\mathrm{crys}}(S^{+})})\otimes_{B^{+}_{\mathrm{crys}}(S^{+})} \bB(U)
    \]
    for a big log affinoid perfectoid $U$ in $(\fX,\cM_{\fX})_{\eta,\mathrm{prok\et}}$ with the associated affinoid perfectoid space $\widehat{U}=\mathrm{Spa}(S,S^{+})$. When $\bB\in\{\bB_{\mathrm{dR}}^{+},\bB_{\mathrm{dR}}\}$, there exists a natural isomorphism
    \[
    \bB(\cE)(U)\cong \cE(A_{\mathrm{crys},K}(S^{+})\twoheadrightarrow S^{+}/p,\cM_{A_{\mathrm{crys},K}(S^{+})})\otimes_{A_{\mathrm{crys},K}(S^{+})} \bB(U)
    \]
    by the crystal property of $\cE$.

    Let $\bB\in\{\bB_{\mathrm{dR}}^{+},\bB_{\mathrm{dR}}\}$. When $\cE$ is enhanced to a filtered isocrystal $(\cE,\mathrm{Fil}^{\bullet}(E))$, we can define a descending filtration on $\bB(\cE)$ as follows. Let $U$ be a big log affinoid perfectoid in $(\fX,\cM_{\fX})_{\mathrm{prok\et}}$ over $(\mathrm{Spf}(\widetilde{R}_{\infty,\alpha}),\cM_{\widetilde{R}_{\infty,\alpha}})_{\eta}$ and $\widehat{U}=\mathrm{Spa}(S,S^{+})$ be the associated affinoid perfectoid space. Choose a morphism $s\colon (\mathrm{Spf}(A_{\mathrm{crys},K}(S^{+})),\cM_{A_{\mathrm{crys},K}(S^{+})})\to (\fX,\cM_{\fX})$ commuting the following diagram:
    \[
    \begin{tikzcd}
    (\mathrm{Spf}(S^{+}),\cM_{S^{+}}) \ar[r,hook] \ar[d] & (\mathrm{Spf}(A_{\mathrm{crys},K}(S^{+})),\cM_{A_{\mathrm{crys},K}(S^{+})}) \ar[ld,dashed,"s"] \\
    (\fX,\cM_{\fX}) & .    
    \end{tikzcd}
    \]
    By the crystal property of $\cE$, the map of log PD-thickenings 
    \[
    s\colon (R\twoheadrightarrow R/p,\cM_{R})\to (A_{\mathrm{crys},K}(S^{+})\twoheadrightarrow S^{+}/p,\cM_{A_{\mathrm{crys},K}(S^{+})})
    \]
    gives an isomorphism
    \[
    \bB(\cE)(U)\cong E\otimes_{R[1/p],s} \bB(U).
    \]
    The right hand side is equipped with the product filtration $\mathrm{Fil}^{\bullet}(E)\otimes \mathrm{Fil}^{\bullet}(\bB(U))$, which induced a filtration on the left side. 

    To justify this procedure, we need to check that such a map $s$ indeed exists and that the resulting filtration on $\bB(\cE)(U)$ is independent of the choice of $s$. First, we shall prove the existence of $s$. Take a factorization
    \[
    U\to (\mathrm{Spf}(\widetilde{R}_{\infty,\alpha}),\cM_{\widetilde{R}_{\infty,\alpha}})_{\eta}\to (\mathrm{Spf}(R),\cM_{R})_{\eta} \to (\fX,\cM_{\fX})_{\eta}
    \]
    as in Definition \ref{def of big log affinoid perfd}. Let $(A_{\mathrm{crys},K}(S^{+}),\bQ_{\geq 0}^{r}\oplus \bZ\cdot e)$ be a prelog ring defined by $\bQ_{\geq 0}^{r}\to A_{\mathrm{inf}}(S^{+})\to A_{\mathrm{crys},K}(S^{+})$ and $e\mapsto \pi/[\pi^{\flat}]$. Note that we have a natural isomorphism
    \[
    (\mathrm{Spf}(A_{\mathrm{crys},K}(S^{+})),\bQ_{\geq 0}^{r}\oplus \bZ\cdot e)^{a}\cong (\mathrm{Spf}(A_{\mathrm{crys},K}(S^{+})),\cM_{A_{\mathrm{crys},K}(S^{+})}).
    \]
    Consider a map of prelog rings 
    \[
    (R^{0},\bN^{r})\to (A_{\mathrm{crys},K}(S^{+}),\bQ_{\geq 0}^{r}\oplus \bZ\cdot e)
    \]
    defined by 
    \begin{align*}
    x_{i}&\mapsto [x_{i}^{\flat}] \ \ (1\leq i\leq l) \\
    y_{j}&\mapsto [y_{j}^{\flat}] \ \ (1\leq j\leq m) \\
    z_{k}&\mapsto [z_{k}^{\flat}] \ \ (1\leq k\leq n-1) \\
    z_{n}&\mapsto \pi(\prod_{k=1}^{n-1}[z_{k}^{\flat}])^{-1}
    \end{align*}
    and the monoid map $\bN^{r}\to \bQ_{\geq 0}^{r}\oplus \bZ\cdot e$ defined by
    \[
    e_{i}\mapsto
    \begin{cases}
        e_{i} \ \ (1\leq i\leq r-1) \\
        e_{n}+e \ \ (i=n).
    \end{cases}
    \]
    Then, passing to the associated log structures, we obtain a desired map $s$. 

    To check that the resulting filtration on $\bB(\cE)(U)$ is independent of the choice of $s$, it suffices to prove that, for another choice
    \[
    s'\colon (\mathrm{Spf}(A_{\mathrm{crys},K}(S^{+})),\cM_{A_{\mathrm{crys},K}(S^{+})})\to (\fX,\cM_{\fX}),
    \]
    the isomorphism
    \[
    E\otimes_{\cO_{\fX_{\eta}},s} \bB(U)\cong \bB(\cE)(U)\cong  E\otimes_{\cO_{\fX_{\eta}},s'} \bB(U)
    \]
    is a filtered isomorphism. This isomorphism can be described by using the connection $\nabla_{E}$ as in Construction \ref{stratification to connection}, and the claim follows from this description and Griffiths transversality (cf. \cite[Construction 2.28]{gr24},\cite[Construction 3.44]{dlms24}).
\end{construction}

\begin{dfn}\label{def of semist loc sys}
    Let $\bL$ be a Kummer \'{e}tale $\bZ_{p}$-local system on $(\fX,\cM_{\fX})_{\eta}$.
    \begin{enumerate}
        \item We say that $\bL$ is \emph{weakly semi-stable} if, for some locally free $F$-isocrystal $\cE$ on $(X,\cM_{X})$, there exists an isomorphism
        \[
        \vartheta\colon \bL\otimes_{\bZ_{p}} \bB_{\mathrm{crys}}\cong \bB_{\mathrm{crys}}(\cE)
        \]
        which is compatible with Frobenius isomorphisms. In this case, we say that $\bL$ and $\cE$ are \emph{associated}.
        \item We say that $\bL$ is \emph{semi-stable} if, for some filtered $F$-isocrystal $(\cE,\mathrm{Fil}^{\bullet}(E))$ on $(\fX,\cM_{\fX})$, there exists an isomorphism
        \[
        \vartheta\colon \bL\otimes_{\bZ_{p}} \bB_{\mathrm{crys}}\cong \bB_{\mathrm{crys}}(\cE)
        \]
        which is compatible with Frobenius isomorphisms and filtered after taking the base change along $\bB_{\mathrm{crys}}\to \bB_{\mathrm{dR}}$. Here, the filtration on $\bB_{\mathrm{dR}}(\cE)$ is the one constructed in Construction \ref{bcrys mod ass to isocrystal}. In this case, we say that $\bL$ and $(\cE,\mathrm{Fil}^{\bullet}(E))$ are \emph{associated}.
    \end{enumerate}
\end{dfn}

\begin{prop}\label{two conn over obdr is eq}
    Let $(\fX,\cM_{\fX})$ be a semi-stable log formal scheme over $\cO_{K}$ and let $\cE$ be an isocrystal on $(X_{0},\cM_{X_{0}})$. Let $(E,\nabla_{E})$ be the corresponding vector bundle with an integrable connection on $\fX_{\eta}$.
    \begin{enumerate}
        \item Let $\nabla'$ be the product connection $\nabla_{E}$ and $\nabla_{\cO\bB^{+}_{\mathrm{dR}}}$ on $\cO\bB^{+}_{\mathrm{dR}}$-vector bundle $E\otimes_{\cO_{\fX_{\eta}}} \cO\bB^{+}_{\mathrm{dR}}$. Then $(E\otimes_{\cO_{\fX_{\eta}}} \cO\bB^{+}_{\mathrm{dR}},\nabla')$ is naturally isomorphic to $(\bB^{+}_{\mathrm{dR}}(\cE)\otimes_{\bB^{+}_{\mathrm{dR}}} \cO\bB^{+}_{\mathrm{dR}},\text{id}\otimes \nabla_{\cO\bB^{+}_{\mathrm{dR}}})$.
        \item Assume that $\cE$ underlies a filtered isocrystal $(\cE,\mathrm{Fil}^{\bullet}(E))$, and let $\mathrm{Fil}^{\bullet}\bB^{+}_{\mathrm{dR}}(\cE)$ be the induced filtration on $\bB^{+}_{\mathrm{dR}}(\cE)$. Then the product filtration $\mathrm{Fil}^{\bullet}E\otimes_{\cO_{\fX_{\eta}}} \mathrm{Fil}^{\bullet}\cO\bB^{+}_{\mathrm{dR,log}}$ is naturally isomorphic to the product filtration $\mathrm{Fil}^{\bullet}\bB^{+}_{\mathrm{dR}}(\cE)\otimes_{\bB^{+}_{\mathrm{dR}}} \mathrm{Fil}^{\bullet}\cO\bB^{+}_{\mathrm{dR,log}}$.
    \end{enumerate}
\end{prop}

\begin{proof}
  We follow the argument of \cite[Proposition 2.36]{gr24}. By working \'{e}tale locally on $\fX$, we may assume that $(\fX,\cM_{\fX})=(\mathrm{Spf}(R),\cM_{R})$ is a small affine log formal scheme with a fixed framing. Let $U$ be a big log affinoid perfectoid in $(\fX,\cM_{\fX})_{\eta,\mathrm{prok\et}}$ with the associated affinoid perfectoid space $\widehat{U}=\mathrm{Spa}(S,S^{+})$. We define $\cO\bA_{\mathrm{crys}}(S^{+})$, $\cO\bA_{\mathrm{crys}}(1)(S^{+})$, and $D^{+}(1)$ to be the $p$-completed log PD-envelope of
  \begin{align*}
      (R\widehat{\otimes}_{W} A_{\mathrm{inf}}(S^{+}),\cM_{R}\oplus \cM_{A_{\mathrm{inf}(S^{+})}})^{a}&\to (S^{+},\cM_{S^{+}}), \\
      (R\widehat{\otimes}_{\cO_{K}} R\widehat{\otimes}_{W} A_{\mathrm{inf}}(S^{+}),\cM_{R}\oplus \cM_{R}\oplus \cM_{A_{\mathrm{inf}(S^{+})}})^{a}&\to (S^{+},\cM_{S^{+}}), \\
      (R\widehat{\otimes}_{\cO_{K}} R,\cM_{R}\oplus \cM_{R})^{a}&\to (R,\cM_{R}),
  \end{align*}
  respectively. They fit into the following diagram of $p$-adic PD-thickenings in $(X_{0},\cM_{X_{0}})_{\mathrm{crys}}$:
  \[
  \begin{tikzcd}
      D^{+}(1) \ar[r] & \cO\bA_{\mathrm{crys}}(1)(S^{+}) \\
      R \ar[u,shift left=0.8ex,"p_{1}"] \ar[u,shift right=0.8ex,"p_{2}"'] \ar[r] & \cO\bA_{\mathrm{crys}}(S^{+}) \ar[u,shift left=0.8ex,"q_{1}"] \ar[u,shift right=0.8ex,"q_{2}"'] \\
      & A_{\mathrm{crys}}(S^{+}) \ar[u]. 
  \end{tikzcd}
  \]
  Fix a factorization $U\to (\mathrm{Spf}(\widetilde{R}_{\infty,\alpha}),\cM_{\widetilde{R}_{\infty,\alpha}})_{\eta}\to (\fX,\cM_{\fX})_{\eta}$, which induces charts $\bQ_{\geq 0}^{r}\to \cM_{S^{+}}$ and $\bQ_{\geq 0}^{r}\to \cM_{A_{\mathrm{inf}}(S^{+})}$. We use the notation in Remark \ref{convenient notation} and Construction \ref{stratification to connection}. Analogously, the above diagram admits maps to the following diagram:
  \[
  \begin{tikzcd}
      R[1/p][[\xi_{1},\dots,\xi_{s}]] \cong D(1) \ar[r] & \cO\bB^{+}_{\mathrm{dR}}(1)(S^{+}) \cong \cO\bB^{+}_{\mathrm{dR}}(S^{+})[[\xi_{1},\dots,\xi_{s}]] \\
       R[1/p] \ar[u,shift left=0.8ex,"p_{1}"] \ar[u,shift right=0.8ex,"p_{2}"'] \ar[r] & \cO\bB^{+}_{\mathrm{dR}}(S^{+}) \ar[u,shift left=0.8ex,"q_{1}"] \ar[u,shift right=0.8ex,"q_{2}"'] \\
       & B^{+}_{\mathrm{dR}}(S^{+}) \ar[u],
  \end{tikzcd}
  \]
  where $\cO\bB^{+}_{\mathrm{dR}}(S^{+})$, $\cO\bB^{+}_{\mathrm{dR}}(1)(S^{+})$, and $D(1)$ is the formal completion of the exactification of surjections of prelog rings
  \begin{align*}
      ((R\widehat{\otimes}_{W} A_{\mathrm{inf}}(S^{+}))[1/p],\bN^{r}\oplus \bQ_{\geq 0}^{r})&\to (S^{+},\bQ_{\geq 0}^{r}), \\
      ((R\widehat{\otimes}_{\cO_{K}} R\widehat{\otimes}_{W} A_{\mathrm{inf}}(S^{+}))[1/p],\bN^{r}\oplus \bN^{r}\oplus \bQ_{\geq 0}^{r})&\to (S^{+},\bQ_{\geq 0}^{r}), \\
      ((R\widehat{\otimes}_{\cO_{K}} R)[1/p],\bN^{r}\oplus \bN^{r})\to (R,\bN^{r}).
  \end{align*}
  Note that $\cO\bB^{+}_{\mathrm{dR}}(S^{+})$ coincides with $\cO\bB^{+}_{\mathrm{dR}}(U)$ by Construction \ref{construction of obdr}. By extension of scalars, we can evaluate $\cE$ at each term in the above diagram. 

$(1)$: The crystal property of $\cE$ for the middle horizontal map and the right bottom vertical map induces a natural isomorphism 
  \[
  (E\otimes_{R[1/p]} \cO\bB^{+}_{\mathrm{dR}})(U)\cong (\bB^{+}_{\mathrm{dR}}(\cE)\otimes_{\bB^{+}_{\mathrm{dR}}} \cO\bB^{+}_{\mathrm{dR}})(U).
  \]
  It suffices to prove that this isomorphism is compatible with connections. The maps $p_{i}\colon R[1/p]\to D(1)$ and $q_{i}\colon \cO\bB^{+}_{\mathrm{dR}}(S^{+})\to \cO\bB^{+}_{\mathrm{dR}}(1)(S^{+})$ induce $p_{i}\colon E\to \cE(D(1))$ and $q_{i}\colon \cE(\cO\bB^{+}_{\mathrm{dR}}(S^{+}))\to \cE(\cO\bB^{+}_{\mathrm{dR}}(1)(S^{+}))$ for $i=1,2$, and we have equations
  \begin{align*}
      &q_{1}(e\otimes f)-q_{2}(e\otimes f) \\
      &= q_{1}(f)p_{1}(e)-q_{2}(f)p_{2}(e) \\
      &=q_{1}(f)(p_{1}(e)-p_{2}(e))+(q_{1}(f)-q_{2}(f))p_{2}(e) \\
      &=q_{1}(f)(\sum_{(n_{1},\dots,n_{s})\in \bN^{s}\backslash \{0\}}(\prod_{i=1}^{s}\nabla_{E, t_{i}^{n_{i}}\frac{\partial^{n_{i}}}{\partial t_{i}^{n_{i}}}})(e)(\prod_{i=1}^{s} \frac{\xi_{i}^{n_{i}}}{n_{i}!})) \\
      &+(\sum_{(n_{1},\dots,n_{s})\in \bN^{s}\backslash \{0\}}(\prod_{i=1}^{s}\nabla_{\cO\bB^{+}_{\mathrm{dR,log}}, t_{i}^{n_{i}}\frac{\partial^{n_{i}}}{\partial t_{i}^{n_{i}}}})(f)(\prod_{i=1}^{s} \frac{\xi_{i}^{n_{i}}}{n_{i}!}))p_{2}(e)
  \end{align*}
  for $e\otimes f\in E\otimes_{R[1/p]} \cO\bB^{+}_{\mathrm{dR}}(S^{+})\cong \cE(\cO\bB^{+}_{\mathrm{dR}}(S^{+}))$. Hence, the coefficient of $\xi_{i}$ in the Taylor expansion of $q_{1}(x)-q_{2}(x)$ is $\nabla'_{t_{i}\frac{\partial}{\partial t_{i}}}(x)$ for every $x\in \cE(\cO\bB^{+}_{\mathrm{dR}}(S^{+}))$. Therefore, applying the crystal property of $\cE$ to the right column in the above diagram, we see that $\nabla'$ is horizontal on $\cE(B^{+}_{\mathrm{dR}}(S^{+}))$. This proves the assertion.

$(2)$: By Construction \ref{bcrys mod ass to isocrystal}, the filtration on $\bB^{+}_{\mathrm{dR}}(\cE)\otimes_{\bB^{+}_{\mathrm{dR}}} \cO\bB^{+}_{\mathrm{dR}}$ is computed by choosing a section 
  \[
  s\colon (R,\cM_{R})\to (A_{\mathrm{crys},K}(S^{+}),\cM_{A_{\mathrm{crys},K}(S^{+})}).
  \]
  There exists a unique map $\beta$ of $p$-adic log PD-thickenings fitting into the following diagram:
  \[
  \begin{tikzcd}
      (R,\cM_{R}) \ar[r,yshift=0.8ex,"p_{1}"] \ar[r,yshift=-0.8ex,"p_{2}"'] & (D^{+}(1),\cM_{D^{+}(1)}) \ar[r,"\beta"] & (\cO\bA_{\mathrm{crys}}(S^{+}),\cM_{\cO\bA_{\mathrm{crys}}(S^{+})})
  \end{tikzcd}
  \]
  such that $\beta\circ p_{1}$ is a natural map and $\beta\circ p_{2}$ is the composition map
  \[
  (R,\cM_{R})\stackrel{s}{\to} (A_{\mathrm{crys},K}(S^{+}),\cM_{A_{\mathrm{crys},K}(S^{+})})\to (\cO\bA_{\mathrm{crys}}(S^{+}),\cM_{\cO\bA_{\mathrm{crys}}(S^{+})}).
  \]
  The above diagram admits a map to the following diagram:
  \[
  \begin{tikzcd}
      R[1/p] \ar[r,yshift=0.8ex,"p_{1}"] \ar[r,yshift=-0.8ex,"p_{2}"'] & D(1) \ar[r,"\beta"] & \cO\bB_{\mathrm{dR}}(S^{+}).
  \end{tikzcd}
  \]
  Therefore, the isomorphism
  \[
  E\otimes_{R[1/p]} \cO\bB^{+}_{\mathrm{dR}}(S^{+})\isom (\bB^{+}_{\mathrm{dR}}(\cE)\otimes_{\bB^{+}_{\mathrm{dR}}} \cO\bB^{+}_{\mathrm{dR}})(U)\isom E\otimes_{R[1/p],\beta\circ p_{2}} \cO\bB^{+}_{\mathrm{dR}}(S^{+})
  \]
  is given by 
  \[
  e\otimes 1\mapsto \sum_{(n_{1},\dots,n_{s})\in \bN^{s}}(\prod_{i=1}^{s}\nabla_{E, t_{i}^{n_{i}}\frac{\partial^{n_{i}}}{\partial t_{i}^{n_{i}}}})(e)\otimes (\prod_{i=1}^{s} \frac{\beta(\xi_{i})^{n_{i}}}{n_{i}!})
  \]
  for $e\in E$, and the Griffiths transversality implies that this is a filtered isomorphism.
\end{proof}

\begin{prop}\label{weakly semi-stable implies semi-stable}

Let $(\fX,\cM_{\fX})$ be a semi-stable log formal scheme over $\cO_{K}$. Let $\bL$ be a weakly semi-stable Kummer \'{e}tale local system on $(\fX,\cM_{\fX})_{\eta}$, and take a $F$-isocrystal $\cE$ on $(X,\cM_{X})$ and an isomorphism $\vartheta\colon \bL\otimes_{\bZ_{p}} \bB_{\mathrm{crys}}\cong \bB_{\mathrm{crys}}(\cE)$ as in Definition \ref{def of semist loc sys} (1). Then there exists a unique filtration $\mathrm{Fil}^{\bullet}_{\mathrm{dR}}(E)$ on $E\coloneqq \cE_{(\fX,\cM_{\fX})}$ satisfying the following conditions:
\begin{enumerate}
    \item $(\cE,\mathrm{Fil}^{\bullet}_{\mathrm{dR}}(E))$ is a filtered $F$-isocrystal on $(\fX,\cM_{\fX})$;
    \item the isomorphism 
    \[
    \vartheta\otimes_{\bB_{\mathrm{crys}}} \bB_{\mathrm{dR}}\colon \bL\otimes_{\bZ_{p}} \bB_{\mathrm{dR}}\cong \bB_{\mathrm{dR}}(\cE)
    \]
    is a filtered isomorphism, where the filtration on $\bB_{\mathrm{dR}}(\cE)$ is the one constructed in Construction \ref{bcrys mod ass to isocrystal}.
\end{enumerate}
  
\end{prop}

\begin{proof}
Let $\fU$ be the open subset of $\fX$ on which the log structure $\cM_{\fX}$ is trivial and $U_{0}$ be the mod-$p$ fiber of $\fU$. Then $\bL|_{\fU_{\eta}}$ is weakly crystalline (with respect to $\fU$) in the sense of \cite[Definition 2.31]{gr24}, and $\bL|_{\fU_{\eta}}$ and $\cE|_{U_{0}}$ are associated. By \cite[Proposition 2.38]{gr24}, $\bL|_{\fU_{\eta}}$ is de Rham in the sense of \cite[Definition 8.3]{sch13}. Let $U$ denote the open subset of $\fX_{\eta}$ on which the log structure $\cM_{\fX_{\eta}}$ is trivial. Then $U$ contains $\fU_{\eta}$ and $\bL|_{U}$ is de Rham by \cite[Theorem 3.9 (iv)]{lz17}.
    
Set $\cO\bB_{\mathrm{dR,log}}(\cE)\coloneqq \bB_{\mathrm{dR}}(\cE)\otimes_{\bB_{\mathrm{dR}}} \cO\bB_{\mathrm{dR,log}}$, which is isomorphic to $E\otimes_{\cO_{\fX_{\eta}}} \cO\bB_{\mathrm{dR,log}}$ as $\cO\bB_{\mathrm{dR,log}}$-vector bundles with connection by Proposition \ref{two conn over obdr is eq}. Let $\mu\colon (\fX,\cM_{\fX})_{\eta,\mathrm{prok\et}}\to \fX_{\eta,\mathrm{an}}$ be the canonical map of sites. We have isomorphisms
\[
\mu_{*}(\bL\otimes_{\bZ_{p}} \cO\bB_{\mathrm{dR,log}})\cong \mu_{*}(\bB_{\mathrm{dR}}(\cE)\otimes_{\bB_{\mathrm{dR}}} \cO\bB_{\mathrm{dR,log}})\cong \mu_{*}(E\otimes_{\cO_{\fX_{\eta}}} \cO\bB_{\mathrm{dR,log}})\cong E.
\]
By \cite[Theorem 3.2.7 (1)(2)]{dllz23a} and what we proved in the previous paragraph, the direct image of the filtration $\bL\otimes_{\bZ_{p}} \mathrm{Fil}^{\bullet}\cO\bB_{\mathrm{dR,log}}$ gives a filtration $\mathrm{Fil}^{\bullet}(E)$ on $E$ with locally free grading pieces such that $(\cE,\mathrm{Fil}^{\bullet}(E))$ is a filtered $F$-isocrystal, and the given isomorphism is refined to filtered isomorphisms
\begin{align*}
&(\bL\otimes_{\bZ_{p}} \cO\bB_{\mathrm{dR,log}},\bL\otimes_{\bZ_{p}} \mathrm{Fil}^{\bullet}\cO\bB_{\mathrm{dR,log}}) \\
&\cong (E\otimes_{\cO_{\fX_{\eta}}} \cO\bB_{\mathrm{dR,log}},\mathrm{Fil}^{\bullet}E\otimes_{\cO_{\fX_{\eta}}} \mathrm{Fil}^{\bullet}\cO\bB_{\mathrm{dR,log}}) \\
&\cong (\bB_{\mathrm{dR}}(\cE)\otimes_{\bB_{\mathrm{dR}}} \cO\bB_{\mathrm{dR,log}},\mathrm{Fil}^{\bullet}\bB_{\mathrm{dR}}(\cE)\otimes_{\bB_{\mathrm{dR}}} \mathrm{Fil}^{\bullet}\cO\bB_{\mathrm{dR,log}})
\end{align*}
that are compatible with connections. Taking horizontal sections in both sides, we see that $\vartheta\otimes_{\bB_{\mathrm{crys}}} \bB_{\mathrm{dR}}$ gives a filtered isomorphism
\[
(\bL\otimes_{\bZ_{p}} \bB_{\mathrm{dR}},\bL\otimes_{\bZ_{p}} \mathrm{Fil}^{\bullet}\bB_{\mathrm{dR}})\cong (\bB_{\mathrm{dR}}(\cE),\mathrm{Fil}^{\bullet}\bB_{\mathrm{dR}}(\cE))
\]
by \cite[Corollary 2.4.2 (2)]{dllz23a}. This proves the assertion for the existence of the desired filtration, and the uniqueness of it follows from the above argument.
\end{proof}

\begin{prop}\label{two obst mod}
Let $(\fX,\cM_{\fX})=(\mathrm{Spf}(R),\cM_{R})$ be a horizontally semi-stable log formal scheme over $\cO_{K}$ with a fixed unramified model $(\fX_{0},\cM_{\fX_{0}})=(\mathrm{Spf}(R_{0}),\cM_{R_{0}})$. Let $\cE$ be an object of $\mathrm{Isoc}^{\varphi}_{\mathrm{lftf}}((X_{k},\cM_{X_{k}}\oplus 0^{\bN})^{a}_{\mathrm{crys}})$ and $(E_{0},\nabla_{E_{0}},\varphi_{E_{0}},N_{E_{0}})$ be the object of the category $\mathrm{Vect}^{\varphi,\nabla,N}(R_{0})$ associated with $\mathrm{Res}(\cE)$ (see Proposition \ref{Fisoc on log crys site and Fisoc with monodromy ope}). Then there is a natural isomorphism of the following $\cO\bB_{\mathrm{st}}$-modules with connection, Frobenius structures, and monodromy operators:
    \begin{itemize}
        \item $(\bB_{\mathrm{crys}}(\cE)\otimes_{\bB_{\mathrm{crys}}} \cO\bB_{\mathrm{st},(\fX_{0},\cM_{\fX_{0}})},\mathrm{id}\otimes \nabla_{\cO\bB_{\mathrm{st},(\fX_{0},\cM_{\fX_{0}})}},\varphi_{\bB_{\mathrm{crys}}(\cE)}\otimes \varphi_{\cO\bB_{\mathrm{st}}},\mathrm{id}\otimes N_{\cO\bB_{\mathrm{st},(\fX_{0},\cM_{\fX_{0}})}})$
        \item $(E_{0}\otimes_{R_{0}[1/p]} \cO\bB_{\mathrm{st},(\fX_{0},\cM_{\fX_{0}})},\varphi_{E_{0}}\otimes \varphi_{\cO\bB_{\mathrm{st},(\fX_{0},\cM_{\fX_{0}})}},\nabla_{E_{0}}\otimes \nabla_{\cO\bB_{\mathrm{st},(\fX_{0},\cM_{\fX_{0}})}},N_{E_{0}}\otimes \mathrm{id}+\mathrm{id}\otimes N_{\cO\bB_{\mathrm{st},(\fX_{0},\cM_{\fX_{0}})}})$.
    \end{itemize}
\end{prop}

\begin{proof}
Since the two $\cO\widehat{\bB}_{\mathrm{st},(\fX_{0},\cM_{\fX_{0}})}$-modules on $(\fX,\cM_{\fX})_{\eta,\mathrm{prok\et}}$ in the assertion is obtained as the restriction of the corresponding $\cO\widehat{\bB}_{\mathrm{st}}$-modules on $(\fX_{0},\cM_{\fX_{0}})_{\eta,\mathrm{prok\et}}$, we may assume that $\cO_{K}=W$. It is enough to prove that there is natural isomorphisms of the following $\cO\widehat{\bB}_{\mathrm{st}}$-modules with connection, Frobenius structures, and monodromy operators:
    \begin{itemize}
        \item $(\bB_{\mathrm{crys}}(\cE)\otimes_{\bB_{\mathrm{crys}}} \cO\widehat{\bB}_{\mathrm{st}},\mathrm{id}\otimes \nabla_{\cO\widehat{\bB}_{\mathrm{st}}},\varphi_{\bB_{\mathrm{crys}}(\cE)}\otimes \varphi_{\cO\widehat{\bB}_{\mathrm{st}}},\mathrm{id}\otimes N_{\cO\widehat{\bB}_{\mathrm{st}}})$
        \item $(E_{0}\otimes_{R_{0}[1/p]} \cO\widehat{\bB}_{\mathrm{st}},\varphi_{E_{0}}\otimes \varphi_{\cO\widehat{\bB}_{\mathrm{st}}},\nabla_{E_{0}}\otimes \nabla_{\cO\widehat{\bB}_{\mathrm{st}}},N_{E_{0}}\otimes \mathrm{id}+\mathrm{id}\otimes N_{\cO\widehat{\bB}_{\mathrm{st}}})$.
    \end{itemize}
Indeed, Lemma \ref{N is nilp} and Corollary \ref{property of crys and semist period sheaf}(3) imply that taking nilpotent sections for monodromy operators on both sides provides the desired isomorphism. 

Let $U$ be a big log affinoid perfectoid in $(\fX,\cM_{\fX})_{\eta,\mathrm{prok\et}}$ with the associated affinoid perfectoid space $\widehat{U}=\mathrm{Spa}(S,S^{+})$. Let $\cM_{S^{+}}$, $\cM_{A_{\mathrm{crys}(S^{+})}}$, and $(\mathrm{Spf}(R\langle u\rangle),\cM_{R\langle u\rangle})$ be as in Construction \ref{hat semistable period sheaf}. Let $\cM_{R\{u\}}$ be the pullback log structure of $\cM_{R\langle u\rangle}$. The strict closed immersion
\[
(X_{k},\cM_{X_{k}}\oplus 0^{\bN})^{a}\to (\mathrm{Spf}(R\{u\}),\cM_{R\{u\}})
\]
defined by $u\mapsto 0$ belongs to $(X_{k},\cM_{X_{k}}\oplus 0^{\bN})^{a}_{\mathrm{crys}}$.
Let $(\mathrm{Spf}(D^{+}(1)),\cM_{D^{+}(1)})$ be the $p$-completed log PD-envelope of the diagonal map 
\[
(\mathrm{Spf}(R/p),\cM_{R/p}\oplus 0^{\bN})^{a}\to (\mathrm{Spf}(R\langle u_{1}\rangle),\cM_{R\langle u_{1}\rangle})\times_{W} (\mathrm{Spf}(R\langle u_{2}\rangle),\cM_{R\langle u_{2}\rangle}),
\]
given by $u_{1},u_{2}\mapsto 0$. We define $(\mathrm{Spf}(\cO\widehat{\bA}_{\mathrm{st}}(1)(U)),\cM_{\cO\widehat{\bA}_{\mathrm{st}}(1)(U)})$ as the $p$-completed log PD-envelope of the closed immersion
\[
(\mathrm{Spf}(S^{+}/p),\cM_{S^{+}/p})\hookrightarrow (\mathrm{Spf}(A_{\mathrm{crys}}(S^{+})\widehat{\otimes}_{W} R\langle u_{1}\rangle \widehat{\otimes}_{W} R\langle u_{2}\rangle),\cM_{A_{\mathrm{crys}}(S^{+})}\oplus \cM_{R\langle u_{1}\rangle}\oplus \cM_{R\langle u_{2}\rangle})^{a}
\]
Then we have the following commutative diagram of $p$-adic log PD-thickenings equipped with Frobenius lists in $(X_{k},\cM_{X_{k}})_{\mathrm{crys}}$. (We omit log structures and PD-ideals.) 
\[
  \begin{tikzcd}
      D^{+}(1) \ar[r] & \cO\widehat{\bA}_{\mathrm{st}}(1)(U) \\
      R\{u\} \ar[u,shift left=0.8ex,"p_{1}"] \ar[u,shift right=0.8ex,"p_{2}"'] \ar[r] & \cO\widehat{\bA}_{\mathrm{st}}(U) \ar[u,shift left=0.8ex,"q_{1}"] \ar[u,shift right=0.8ex,"q_{2}"'] \\
      & A_{\mathrm{crys}}(S^{+}) \ar[u]. 
  \end{tikzcd}
\]
The crystal property of $\cE$ gives isomorphisms of $\cO\widehat{\bB}^{+}_{\mathrm{st}}(U)$-modules with Frobenius structures
\[
\cE_{A_{\mathrm{crys}}(S^{+})}\otimes \cO\widehat{\bA}_{\mathrm{st}}(U) \cong \cE_{\cO\widehat{\bA}_{\mathrm{st}}(U)} \cong \cE_{R\langle u\rangle}\otimes_{R\langle u\rangle} \cO\widehat{\bA}_{\mathrm{st}}(U).
\]
By Lemma \ref{equivalence from phinabla to phinablaN}, we have an isomorphism
\begin{align*}
    &(\cE_{R\{u\}},\varphi_{\cE_{R\{u\}}},\nabla_{\cE_{R\{u\}}}^{\mathrm{hol}},N_{\cE_{R\{u\}}}) \\
    \cong &(E_{0}\widehat{\otimes}_{W} W\{u\},\varphi_{E_{0}}\otimes \varphi_{W\{u\}},\nabla_{E_{0}}\otimes \mathrm{id},N_{E_{0}}\otimes \mathrm{id}+\mathrm{id}\otimes N_{W\{u\}}).
\end{align*}
Combining these isomorphisms gives an isomorphism of $\cO\widehat{\bB}^{+}_{\mathrm{st}}(U)$-modules with Frobenius structures
\[
\cE_{A_{\mathrm{crys}}(S^{+})}\otimes_{A_{\mathrm{crys}}(S^{+})} \cO\widehat{\bA}_{\mathrm{st}}(U) \cong E_{0}\otimes_{R} \cO\widehat{\bA}_{\mathrm{st}}(U).
\]
This gives a desired isomorphism except for the compatibility with connections and monodromy operators. 

To prove the remaining, it is enough to show that the connection and the monodromy operator on $\cE_{\widehat{\bA}_{\mathrm{st}}(U)}$ induced from $\nabla_{E_{0}}\otimes \nabla_{\cO\widehat{\bA}_{\mathrm{st}}(U)}$ and $N_{E_{0}}\otimes \mathrm{id}+\mathrm{id}\otimes N_{\cO\widehat{\bA}_{\mathrm{st}}(U)}$, denoted by $\nabla'$ and $N'$, kills $\cE_{A_{\mathrm{crys}}(S^{+})}$. When we fix a framing on $(\fX,\cM_{\fX})$, there is an isomorphism
\[
R\{u\}\{\xi_{1},\dots,\xi_{s},\xi\}\cong D^{+}(1)
\]
given by $\xi\mapsto p_{1}(t_{i})p_{2}(t_{i})^{-1}-1$ and $\xi\mapsto p_{1}(u)p_{2}(u)^{-1}-1$, and we have an equation 
\begin{align*}
    p_{1}^{*}(x)-p_{2}^{*}(x)&=\sum_{i=1}^{s} \nabla_{\cE_{R\{u\}},t_{i}\frac{\partial}{\partial t_{i}}}(x)\xi_{i}+N_{\cE_{R\{u\}}}(x)\xi \\
    &=\sum_{i=1}^{s} \nabla_{\cE_{0},t_{i}\frac{\partial}{\partial t_{i}}}(y)\xi_{i}+N_{E_{0}}(y)\xi
\end{align*}
for $x=y\otimes 1\in \cE_{R\{u\}}=E_{0}\widehat{\otimes}_{W} W\{u\}$
after taking modulo $(\xi_{1},\dots,\xi_{s},\xi)^{[\geq 2]}$. The square in the above diagram is cocartesian, and there is an induced isomorphism
\[
\cO\widehat{\bA}_{\mathrm{st}}(U)\{\xi_{1},\dots,\xi_{s},\xi\}\cong \cO\widehat{\bA}_{\mathrm{st}}(U)(1).
\]
Hence, we obtain the following equation in $\cE_{\cO\widehat{\bA}_{\mathrm{st}}(U)}\otimes_{\cO\widehat{\bA}_{\mathrm{st}}(U),q_{2}} \cO\widehat{\bA}_{\mathrm{st}}(U)(1)/(\xi_{1},\dots,\xi)^{[\geq 2]}$:
\begin{align*}
    q_{1}^{*}(x)-q_{2}^{*}(x)&=(q_{1}(f)-q_{2}(f))y+f(p_{1}^{*}(y\otimes 1)-p_{2}^{*}(y\otimes 1)) \\
    &=(\sum_{i=1}^{s}\nabla_{\cO\widehat{\bA}_{\mathrm{st}}(U),t_{i}\frac{\partial}{\partial t_{i}}}(f)\xi_{i}+N_{\cO\widehat{\bA}_{\mathrm{st}}(U)}(f)\xi)\cdot y \\
    &+f\cdot (\sum_{i=1}^{s} \nabla_{\cE_{0},t_{i}\frac{\partial}{\partial t_{i}}}(y)\xi_{i}+N_{E_{0}}(y)\xi)
\end{align*}
for $x=y\otimes f\in \cE_{\cO\widehat{\bA}_{\mathrm{st}}(U)}=E_{0}\otimes_{R} \cO\widehat{\bA}_{\mathrm{st}}(U)$ after taking modulo $(\xi_{1},\dots,\xi_{s},\xi)^{[\geq 2]}$, which implies that $\nabla'(x)$ and $N'(x)$ occur in coefficients of the Taylor expansion of $q_{1}^{*}(x)-q_{2}^{*}(x)$. Therefore, $\nabla'$ and $N'$ kill $\cE_{A_{\mathrm{crys}(S^{+})}}$ by the crystal property of $\cE$ for the right column in the above diagram.
\end{proof}

\begin{prop}\label{bcrys functor is fully faithful}
    Let $(\fX,\cM_{\fX})$ be a semi-stable log formal scheme over $\cO_{K}$ which admits framings with $n\leq 1$ \'{e}tale locally. Then the functor 
    \[ \mathrm{Isoc}^{\varphi}_{\mathrm{lftf}}((X_{p=0},\cM_{X_{p=0}})_{\mathrm{crys}})\to \mathrm{Vect}^{\varphi}(\bB_{\mathrm{crys}})
    \]
    sending $\cE$ to $\bB_{\mathrm{crys}}(\cE)$ is fully faithful.
\end{prop}

\begin{proof}
    By working \'{e}tale locally on $\fX$, we may assume that $(\fX,\cM_{\fX})=(\mathrm{Spf}(R),\cM_{R})$ is a small affine log formal scheme over $\cO_{K}$ with a fixed framing such that $n\leq 1$. Let $(\fX_{0},\cM_{\fX_{0}})=(\mathrm{Spf}(R_{0}),\cM_{R_{0}})$ denote the unramified model of $(\fX,\cM_{\fX})$ defined by the fixed framing. 

    We may assume that $n=1$ because, for a horizontally semi-stable log formal scheme $(\fY,\cM_{\fY})$ over $\cO_{K}$, the natural functor
    \[
    \mathrm{Isoc}^{\varphi}((Y_{p=0},\cM_{Y_{p=0}})_{\mathrm{crys}})\to \mathrm{Isoc}^{\varphi}((Y_{p=0},\cM_{Y_{p=0}}\oplus \pi^{\bN})^{a}_{\mathrm{crys}})
    \]
    is fully faithful by Lemma \ref{dwork trick}(1) and Corollary \ref{N=0 log isoc}. In this case, Proposition \ref{two obst mod} gives the following commutative diagram:
    \[
    \begin{tikzcd}
        \mathrm{Vect}^{\varphi}(\bB_{\mathrm{crys}}) \ar[rr,"-\otimes_{\bB_{\mathrm{crys}}} \cO\bB_{\mathrm{st},(\fX_{0},\cM_{\fX_{0}})}"] & & \mathrm{Vect}^{\varphi,\nabla,N}(\cO\bB_{\mathrm{st},(\fX_{0},\cM_{\fX_{0}})}) \\ \mathrm{Isoc}^{\varphi}_{\mathrm{lftf}}((X_{p=0},\cM_{X_{p=0}})_{\mathrm{crys}}) \ar[u,"\cE\mapsto \bB_{\mathrm{crys}}(\cE)"] \ar[rr,"\sim"] & & \mathrm{Vect}^{\varphi,\nabla,N}(R_{0}) \ar[u,"-\otimes_{R_{0}} \cO\bB_{\mathrm{st},(\fX_{0},\cM_{\fX_{0}})}"']
    \end{tikzcd}
    \]
    The bottom horizontal functor is the composition of the equivalence in Proposition \ref{Fisoc on log crys site and Fisoc with monodromy ope} with the equivalence $\mathrm{Vect}^{\varphi,\nabla,N}(R_{0})\simeq \mathrm{Isoc}^{\varphi,N}_{\mathrm{lf}}((X_{k},\cM_{X_{k}})_{\mathrm{crys}})$. Since a problem on morphisms between vector bundles can be reduced to one on global sections of them by taking internal homomorphisms, the right vertical functor is fully faithful by Proposition \ref{zeroth coh of obcrys}, and the upper horizontal functor is fully faithful by $\bB_{\mathrm{crys}}=\cO\bB_{\mathrm{st}}^{\nabla=N=0}$ (Corollary \ref{property of crys and semist period sheaf}(1)(2)). Therefore, the left vertical functor is also fully faithful.
\end{proof}

\begin{rem}
    When $(\fX,\cM_{\fX})$ is horizontally semi-stable, we can use the method of \cite[Proposition 2.29]{gr24}, and the above proposition is proved without $F$-structures. However, we do not use it in this paper.
\end{rem}

\section{Log prismatic crystals}\label{section pris crys}

\subsection{Log prismatic sites}

First, let us review some basics of prisms introduced by \cite{bs22}. A \emph {$\delta$-ring} is a ring $A$ with a map $\delta:A\to A$ satisfying the following conditions:
\begin{enumerate}
    \item $\delta(1)=0$
    \item $\delta(xy)=\delta(x)y^{p}+x^{p}\delta(y)+p\delta(x)\delta(y)$;
    \item $\delta(x+y)=\delta(x)+\delta(y)-\sum_{i=1}^{p-1} \frac{(p-1)!}{i!(p-i)!}x^{i}y^{p-i}$.
\end{enumerate}
For a $\delta$-ring $A$, a map $\phi_{A}:A\to A$ mapping $a$ to $a^{p}+p\delta(a)$ is a ring map lifting the Frobenius map on $A/p$. A \emph{prism} is a pair $(A,I)$ consisting of a derived $(p,I)$-complete ring $A$ and a locally principal ideal $I$ of $A$ such that $p\in I+\phi_{A}(I)A$. We say that a prism $(A,I)$ is \emph{bounded} if $A/I$ has bounded $p^{\infty}$-torsion. A prism $(A,I)$ is called \emph{orientable} if $I$ is a principal ideal. Throughout this paper, we use the following lemma with no reference. 

\begin{lem}
    Let $(A,I)$ be a bounded prism and $B$ be a $(p,I)$-adically complete $A$-algebra. Then the map $A\to B$ is $(p,I)$-completely (faithfully) flat if and only if it is $(p,I)$-adically (faithfully) flat.
\end{lem}

\begin{proof}
    This follows from results of Yektieli. For the proof, see \cite[Lemma 1.3]{iky24} for example.
\end{proof}

For a bounded $p$-adic formal scheme $\fX$, the \emph{absolute prismatic site} $\fX_{\Prism}$ is the category of prisms $(A,I)$ equipped with maps $\mathrm{Spf}(A/I)\to \fX$. We equip $\fX_{\Prism}$ with completely flat topology. We simply write $R_{\Prism}$ for $\mathrm{Spf}(R)_{\Prism}$. Let $\cO_{\Prism}$ (resp. $\cI_{\Prism}$) denote the sheaf on $\fX_{\Prism}$ given by $(A,I)\mapsto A$ (resp. $I$). The Frobenius lift $\phi_{A}$ for a prism $(A,I)$ induces a map $\phi:\cO_{\Prism}\to \cO_{\Prism}$. A \emph{prismatic crystal} on $\fX$ is a vector bundle on the ringed site $(\fX_{\Prism},\cO_{\Prism})$, and a \emph{prismatic $F$-crystal} on $\fX$ is a pair $(\cE,\varphi_{\cE})$ consisting of a prismatic crystal $\cE$ and an isomorphism $\varphi_{\cE}\colon (\phi^{*}\cE)[1/\cI_{\Prism}]\isom \cE[1/\cI_{\Prism}]$. When no confusion occurs, we simply write $\cE$ for $(\cE,\varphi_{\cE})$. Let $\mathrm{Vect}(\fX_{\Prism})$ (resp. $\mathrm{Vect}^{\varphi}(\fX_{\Prism})$) denote the category of prismatic crystals (resp. prismatic $F$-crystals) on $\fX$.

\begin{lem}[Quasi-syntomic descent for prismatic crystals]\label{qsyn descent for nonlog pris crys}
    Let $R\to S$ be a quasi-syntomic cover of bounded $p$-complete rings and $S^{(\bullet)}$ be the \v{C}ech nerve of $S$ over $R$. Then the natural functor
    \[
    \mathrm{Vect}(R_{\Prism})\to \varprojlim_{\bullet\in \Delta}\mathrm{Vect}(S^{(\bullet)}_{\Prism})
    \]
    is bi-exact equivalent. Analogous assertions for $\mathrm{Vect}^{\varphi}$ and $\mathrm{Vect}^{\varphi}_{[a,b]}$ also hold.
\end{lem}

\begin{proof}
    This follows from \cite[Proposition 7.11]{bs22}.
\end{proof}

Now, we recall some basics of log prisms introduced in \cite{kos22}. A \emph{$\delta_{\mathrm{log}}$-ring} is a prelog ring $(A,M)$ with a $\delta$-ring structure on $A$ and a map $\delta_{\mathrm{log}}:M\to A$ satisfying the following conditions:
\begin{enumerate}
    \item $\delta_{\mathrm{log}}(e)=0$;
    \item $\alpha(m)^{p}\delta_{\mathrm{log}}(m)=\delta(\alpha(m))$;
    \item $\delta_{\mathrm{log}}(mm')=\delta_{\mathrm{log}}(m)+\delta_{\mathrm{log}}(m')+p\delta_{\mathrm{log}}(m)\delta_{\mathrm{log}}(m')$,
\end{enumerate}
where $\alpha:M\to A$ is the structure map. A \emph{bounded prelog prism} $(A,I,M)$ (resp. \emph{bounded log prism} $(A,I,\cM_{A})$) is a bounded prism $(A,I)$ equipped with a prelog structure $M\to A$ (resp. log structure $\cM_{A}$ on $\mathrm{Spf}(A)$) and a $\delta_{\mathrm{log}}$-structure on the prelog ring $(A,M)$ (resp. the log formal scheme $(\mathrm{Spf}(A),\cM_{A})$). Here, a $\delta_{\mathrm{log}}$-structure on the log formal scheme $(\mathrm{Spf}(A),\cM_{A})$ is a sheaf map $\cM_{A}\to \cO_{\mathrm{Spf}(A)}$ that defines a $\delta_{\mathrm{log}}$-structure on each section. We can attach to a prelog prism $(A,I,M)$ a log prism $(A,I,M)^{a}\coloneqq (A,I,\cM_{A}\coloneqq M^{a})$. For a log prism $(A,I,\cM_{A})$, the ring map $\phi_{A}$ and the map of log structures $\phi_{A}^{*}\cM_{A}\to \cM_{A}$ given by $m\mapsto m^{p}(1+p\delta_{\mathrm{log}}(m))$ for $m\in \phi_{A}^{-1}\cM_{A}$ (here, $1+p\delta_{\mathrm{log}}(m)\in \cO_{\mathrm{Spf}(A)}^{\times}$ is regarded as a section of $\cM_{A}$) define a morphism of log formal schemes $\phi_{(A,\cM_{A})}:(\mathrm{Spf}(A),\cM_{A})\to (\mathrm{Spf}(A),\cM_{A})$. When no confusion occurs, we write simply $\phi_{A}$ for $\phi_{(A,\cM_{A})}$.

\begin{dfn}[Absolute log prismatic sites, {\cite[Remark 4.6]{kos22}}] \label{def of log pris site}

Let $(\fX,\cM_{\fX})$ be a bounded $p$-adic log formal scheme. Let $(\fX,\cM_{\fX})_{\Prism}$ be the site with objects described as follows: an object consists of

\begin{itemize}
 \item a log prism $(A,I,\cM_{A})$ such that $(\mathrm{Spf}(A),\cM_{A})$ is quasi-coherent,
 \item a map of log formal schemes $(\Spf (A/I), \cM_{A/I}) \to (\fX,\cM_{\fX})$, where $\cM_{A/I}$ is the restriction of $\cM_{A}$.
\end{itemize}    
We simply write $(A,I,\cM_{A})$ for an object when there exists no confusion. A morphism $(A,I,\cM_{A})\to (B,J,\cM_{B})$ is the following commutative diagram:
\[
\begin{tikzcd}
    (\mathrm{Spf}(B),\cM_{B}) \ar[d] & (\mathrm{Spf}(B/J),\cM_{B/J}) \ar[l,hook'] \ar[d] \ar[rd]& \\
    (\mathrm{Spf}(A),\cM_{A}) & (\mathrm{Spf}(A/I),\cM_{A/I}) \ar[l,hook'] \ar[r] & (\fX,\cM_{\fX}).
\end{tikzcd}
\]
A morphism $(A,I,\cM_{A})\to (B,J,\cM_{B})$ in $(\fX,\cM_{\fX})_{\Prism}$ is a covering if the corresponding map $(\mathrm{Spf}(B),\cM_{B})\to (\mathrm{Spf}(A),\cM_{A})$ is strict $(p,I)$-completely faithfully flat. The resulting site $(\fX,\cM_{\fX})_{\Prism}$ is called the \emph{(absolute) log prismatic site} of $(\fX,\cM_{\fX})$.

The structure sheaf $\cO_{\Prism}$ is defined by $(A,I,\cM_{A})\mapsto A$. Similarly, the sheaf $\overline{\cO}_{\Prism}$ is defined by $(A,I,\cM_{A})\mapsto A/I$, and $\cI_{\Prism}$ by $(A, I, \cM_{A})\mapsto I$. The functor defined by $(A,I,\cM_{A})\mapsto A/(p,I)$ is also a sheaf and coincides with the sheaf $\cO_{\Prism}/(p,\cI_{\Prism})$, denoted by $\widetilde{\cO}_{\Prism}$. The ring map $\cO_{\Prism}\to \cO_{\Prism}$ induced by $\phi_{A}$ for log prisms $(A,I,\cM_{A})$ is denoted by $\phi$.

Let $(\fX,\cM_{\fX})^{\mathrm{str}}_{\Prism}$ denote the full subcategory of $(\fX,\cM_{\fX})_{\Prism}$ consisting of objects whose structure morphisms $(\mathrm{Spf}(A/I),\cM_{A/I})\to (\fX,\cM_{\fX})$ are strict. The site $(\fX,\cM_{\fX})^{\mathrm{str}}_{\Prism}$ is called the \emph{(absolute) strict log prismatic site} of $(\fX,\cM_{\fX})$.
\end{dfn}

\begin{rem}\label{rem on def of log pris site}
    The definition of log prismatic sites of \cite[Remark 4.6]{kos22} is different from the one given above in that the following two assumptions are added in \emph{loc.cit.}:
    \begin{itemize}
        \item a log prism $(A,I,\cM_{A})$ admits a global chart;
        \item the structure morphism $(\mathrm{Spf}(A/I),\cM_{A/I})\to (\fX,\cM_{\fX})$ admits a chart \'{e}tale locally.
    \end{itemize}
    The first additional assumption does not change the associated topoi. On the other hand, the second additional one changes the associated topoi. However, the category of prismatic crystals is unchanged under certain mild assumptions (see Lemma \ref{category of crystals is unchanged}).
\end{rem}

\begin{rem}\label{str log prism is qcoh}
    Let $(\fX,\cM_{\fX})$ be a bounded $p$-adic quasi-coherent and integral log formal scheme. For the definition of $(\fX,\cM_{\fX})_{\Prism}^{\mathrm{str}}$, the condition that $(A,\cM_{A})$ is quasi-coherent is redundant. To check this, it suffices to show that, for a log prism $(A,I,\cM_{A})$ with a strict morphism $(\mathrm{Spf}(A/I),\cM_{A/I})\to (\fX,\cM_{\fX})$, the log formal scheme  $(\mathrm{Spf}(A),\cM_{A})$ is quasi-coherent. By working \'{e}tale locally on $\mathrm{Spf}(A)$, we may assume that there exists an integral chart $P\to \cM_{\fX}$. We put
    \[
    \widetilde{P}\coloneqq P\times_{\Gamma(\mathrm{Spf}(A/I),\cM_{A/I})} \Gamma(\mathrm{Spf}(A),\cM_{A}).
    \]
    Since $\cM_{A}$ is integral, $\widetilde{P}\to P$ is a $1+I$-torsor by \cite[Lemma 3.8]{kos22}. Therefore, a natural map $\widetilde{P}\to \cM_{A}$ is a chart.
\end{rem}

\begin{dfn}[Log prismatic crystals]\label{def of log pris crys}
    Let $(\fX,\cM_{\fX})$ be a bounded $p$-adic log formal scheme. We let $\mathrm{Vect}((\fX, \cM_{\fX})_{\Prism})$ (resp. $\mathrm{Vect}((\fX, \cM_{\fX})^{\mathrm{str}}_{\Prism})$) denote the category of vector bundles on the ringed site on $((\fX,\cM_{\fX})_{\Prism},\cO_{\Prism})$ (resp. $((\fX, \cM_{\fX})^{\mathrm{str}}_{\Prism},\cO_{\Prism})$). Objects of $\mathrm{Vect}((\fX,\cM_{\fX})_{\Prism})$ are called \emph{log prismatic crystals}. 
    
    We let $\mathrm{Vect}^{\varphi}((\fX,\cM_{\fX})_{\Prism})$ denote the category of pairs $(\cE, \varphi_{\cE})$ where $\cE$ is an object of $\mathrm{Vect}((\fX,\cM_{\fX})_{\Prism})$ and $\varphi_{\cE}$ is an isomorphism $(\phi^{\ast}\cE)[1/\cI_{\Prism}]\isom \cE[1/\cI_{\Prism}]$. Objects of $\mathrm{Vect}^{\varphi}((\fX,\cM_{\fX})_{\Prism})$ are called \emph{log prismatic F-crystals}. Similarly, we define variants $ \mathrm{Vect}^{\varphi}((\fX,\cM_{\fX})_{\Prism}^{\mathrm{str}})$. All of these categories are canonically equipped with the structure of an exact category.
\end{dfn}

\begin{rem}\label{pris crys as lim}
    For $\star\in \{\mathrm{str}, \emptyset\}$, we have the following canonical bi-exact equivalences.
    \begin{align*}
    \mathrm{Vect}((\fX,\cM_{\fX})_{\Prism}^{\star})&\simeq \varprojlim_{(A,I,\cM_{A})\in (\fX,\cM_{\fX})_{\Prism}^{\star}} \mathrm{Vect}(A) \\
    \mathrm{Vect}^{\varphi}((\fX,\cM_{\fX})^{\star}_{\Prism})&\simeq \varprojlim_{(A,I,\cM_{A})\in (\fX,\cM_{\fX})^{\star}_{\Prism}} \mathrm{Vect}^{\varphi}(A,I)
    \end{align*}
    Here, $\mathrm{Vect}(R)$ is the category of finite projective $R$-modules for a ring $R$, and $\mathrm{Vect}^{\varphi}(A,I)$ is the category of pairs $(N,\varphi)$ of a finite projective $A$-module $N$ and an isomorphism $\varphi: (\phi_{A}^{\ast}N)[1/I]\isom N[1/I]$ for a prism $(A,I)$. For an object $\cE$ of a category on the left side in the above equivalences, the image of $\cE$ via the projection  with respect to $(A,I,\cM_{A})$ is denoted by $\cE_{(A,I,\cM_{A})}$. When no confusion occurs, we simply write $\cE_{A}$ for this.
\end{rem}

\begin{dfn}[Analytic log prismatic \texorpdfstring{$F$}--crystals, {\cite{gr24}}, {\cite{dlms24}}]\noindent

    Let $(A,I)$ be a prism. We let $\mathrm{Vect}^{\mathrm{an}}(A,I)$ denote the exact category of vector bundles on a scheme $U(A,I)\coloneqq\Spec(A)\backslash V(p,I)$, and $\mathrm{Vect}^{\mathrm{an},\varphi}(A,I)$ denote the exact category of pairs $(\cE,\varphi_{\cE})$ consisting of $\cE\in \mathrm{Vect}^{\mathrm
    {an}}(A,I)$ and an isomorphism $\varphi_{\cE}\colon (\phi_{A}^{*}\cE)[1/I]\isom \cE[1/I]$. 

    For a bounded $p$-adic formal scheme $\fX$, we define the exact categories $\mathrm{Vect}^{\mathrm{an}}(\fX_{\Prism})$ and $\mathrm{Vect}^{\mathrm{an},\varphi}(\fX_{\Prism})$ by
    \begin{align*}
    \mathrm{Vect}^{\mathrm{an}}(\fX_{\Prism})&\coloneqq \varprojlim_{(A,I,\cM_{A})\in \fX_{\Prism}} \mathrm{Vect}^{\mathrm{an}}(A,I), \\
    \mathrm{Vect}^{\mathrm{an},\varphi}(\fX_{\Prism})&\coloneqq \varprojlim_{(A,I,\cM_{A})\in \fX_{\Prism}} \mathrm{Vect}^{\mathrm{an},\varphi}(A,I).
    \end{align*}
    
    For a bounded $p$-adic log formal scheme $(\fX,\cM_{\fX})$, we define the exact categories $\mathrm{Vect}^{\mathrm{an}}((\fX,\cM_{\fX})_{\Prism})$ and $\mathrm{Vect}^{\mathrm{an},\varphi}((\fX,\cM_{\fX})_{\Prism})$ by
    \begin{align*}
    \mathrm{Vect}^{\mathrm{an}}((\fX,\cM_{\fX})_{\Prism})&\coloneqq \varprojlim_{(A,I,\cM_{A})\in (\fX,\cM_{\fX})_{\Prism}} \mathrm{Vect}^{\mathrm{an}}(A,I), \\
    \mathrm{Vect}^{\mathrm{an},\varphi}((\fX,\cM_{\fX})_{\Prism})&\coloneqq \varprojlim_{(A,I,\cM_{A})\in (\fX,\cM_{\fX})_{\Prism}} \mathrm{Vect}^{\mathrm{an},\varphi}(A,I).
    \end{align*}
    Similarly, we define the exact categories $\mathrm{Vect}^{\mathrm{an}}((\fX,\cM_{\fX})^{\mathrm{str}}_{\Prism})$ and  $\mathrm{Vect}^{\mathrm{an},\varphi}((\fX,\cM_{\fX})^{\mathrm{str}}_{\Prism})$ by
    \begin{align*}
    \mathrm{Vect}^{\mathrm{an}}((\fX,\cM_{\fX})^{\mathrm{str}}_{\Prism})&\coloneqq \varprojlim_{(A,I,\cM_{A})\in (\fX,\cM_{\fX})^{\mathrm{str}}_{\Prism}} \mathrm{Vect}^{\mathrm{an}}(A,I), \\
    \mathrm{Vect}^{\mathrm{an},\varphi}((\fX,\cM_{\fX})^{\mathrm{str}}_{\Prism})&\coloneqq \varprojlim_{(A,I,\cM_{A})\in (\fX,\cM_{\fX})^{\mathrm{str}}_{\Prism}} \mathrm{Vect}^{\mathrm{an},\varphi}(A,I).
    \end{align*}
    Objects of the category $\mathrm{Vect}^{\mathrm{an}}((\fX,\cM_{\fX})_{\Prism})$ (resp. $\mathrm{Vect}^{\mathrm{an},\varphi}((\fX,\cM_{\fX})_{\Prism})$) are called \emph{analytic log prismatic crystals} (resp. \emph{analytic log prismatic $F$-crystals}) on $(\fX,\cM_{\fX})$. 
\end{dfn}

\begin{lem}\label{category of crystals is unchanged}
    Let $(\fX,\cM_{\fX})$ be a quasi-coherent and integral bounded $p$-adic formal scheme. Then the inclusion functor $(\fX, \cM_{\fX})^{\mathrm{str}}_{\Prism}\hookrightarrow (\fX, \cM_{\fX})_{\Prism}$ has a right adjoint functor
    \[
    (\fX, \cM_{\fX})_{\Prism}\to (\fX, \cM_{\fX})^{\mathrm{str}}_{\Prism} \ \ \ \  ((A,I,\cM_{A})\mapsto (A,I,\cN_{A})).
    \]
    In particular, $(\fX, \cM_{\fX})^{\mathrm{str}}_{\Prism}$ is cofinal in $(\fX, \cM_{\fX})_{\Prism}$, and, for any $\star\in \{\emptyset,\varphi,\mathrm{an},(\mathrm{an},\varphi)\}$, the following restriction functors give bi-exact equivalences.
    \[
    \mathrm{Vect}^{\star}((\fX, \cM_{\fX})_{\Prism})\to \mathrm{Vect}^{\star}((\fX, \cM_{\fX})^{\mathrm{str}}_{\Prism}) 
    \]
\end{lem}

\begin{proof}
    Let $(A,I,\cM_{A})\in (\fX, \cM_{\fX})_{\Prism}$. We let $f\colon (\mathrm{Spf}(A/I),\cM_{A/I})\to (\fX,\cM_{\fX})$ denote the structure morphism. Under the identification $\mathrm{Spf}(A)_{\et}\simeq \mathrm{Spf}(A/I)_{\et}$, we define the log structure $\cN_{A}$ on $\mathrm{Spf}(A)$ by $\cN_{A}\coloneqq\cM_{A}\times_{\cM_{A/I}} f^{*}\cM_{\fX}$. By restricting the $\delta_{\mathrm{log}}$-structure to $\cN_{A}$, we obtain a log prism $(A,I,\cN_{A})$, which belongs to $(\fX,\cM_{\fX})_{\Prism}^{\mathrm{str}}$ by Remark \ref{str log prism is qcoh}. Then the functor $(\fX, \cM_{\fX})_{\Prism}\to (\fX, \cM_{\fX})^{\mathrm{str}}_{\Prism}$ sending $(A,I,\cM_{A})$ to $(A,I,\cN_{A})$ is the desired one by construction.
\end{proof}

\begin{construction}[Breuil-Kisin log prisms]\label{construction of bk log prism}

Let $(\fX,\cM_{\fX})=(\mathrm{Spf}(R),\cM_{R})$ be a small affine log formal scheme over $\cO_{K}$ and fix a framing $t$. We shall construct a log prism $(\fS_{R,t},(E),\cM_{\fS_{R,t}})=(\fS_{R,t},(E),\bN^{r})^{a}\in (\fX,\cM_{\fX})_{\Prism}^{\mathrm{str}}$. Let 
    \[
    \fS_{R^{0}}\coloneqq W(k)\langle x^{\pm1}_{1},\dots,x^{\pm1}_{l},y_{1},\dots,y_{m},z_{1},\dots,z_{n} \rangle \llbracket t \rrbracket/(\prod^{n}_{k=1}z_{k}-t).
    \]
     We define a prelog structure $\bN^{r}\to \fS_{R^{0}}$ by
     \[
    e_{i}\mapsto 
    \begin{cases}
        y_{i} \ (1\leq i\leq m) \\
        z_{i-m} \ (m+1\leq i\leq r),
    \end{cases}
     \]
     and define a $\delta_{\mathrm{log}}$-structure on this prelog ring by 
     $\delta(x_{i})=\delta(y_{i})=\delta(z_{i})=0$ and $\delta_{\mathrm{log}}(e_{i})=0$ for all $i$. Then we have an isomorphism $\fS_{R^{0}}/(E)\cong R^{0}$, and a log prism $(\fS_{R^{0}},(E),\bN^{r})^{\mathrm{a}}$ is an object of $((\mathrm{Spf}(R^{0}),\bN^{r})^{a})^{\mathrm{str}}_{\Prism}$. The $p$-adically \'{e}tale $R^{0}$-algebra $R$ lifts uniquely to a $(p,E)$-adically \'{e}tale $\fS_{R^{0}}$-algebra $\fS_{R,t}$. By \cite[Lemma 2.13]{kos22}, there exists a unique prelog prism structure on $(\fS_{R,t},(E),\bN^{r})$ such that $(\fS_{R^{0}},(E),\bN^{r})\to (\fS_{R,t},(E),\bN^{r})$ is a map of prelog prisms. Then $(\fS_{R,t},(E),\cM_{\fS_{R,t}})\coloneqq (\fS_{R,t},(E),\bN^{r})^{a}$ is an object of $(\mathrm{Spf}(R),\cM_{R})_{\Prism}^{\mathrm{str}}$. When no confusion occurs, we simply write $(\fS_{R},(E),\cM_{\fS_{R}})$ for $(\fS_{R,t},(E),\cM_{\fS_{R,t}})$.

     By Lemma \ref{divisible monoid and log pris site}, there exists a unique log structure $\cM_{A_{\mathrm{inf}}(\widetilde{R}_{\infty,\alpha})}$ on $\mathrm{Spf}(A_{\mathrm{inf}}(\widetilde{R}_{\infty,\alpha}))$ such that $(A_{\mathrm{inf}}(\widetilde{R}_{\infty,\alpha}),(\xi),\cM_{A_{\mathrm{inf}}(\widetilde{R}_{\infty,\alpha})})$ is a log prism in $(\mathrm{Spf}(\widetilde{R}_{\infty,\alpha}),\cM_{\widetilde{R}_{\infty,\alpha}})_{\Prism}^{\mathrm{str}}$ up to a unique isomorphism. By the proof of \emph{loc. cit.}, $\cM_{A_{\mathrm{inf}}(\widetilde{R}_{\infty,\alpha})}$ is isomorphic to the associated log structure with the prelog structure $\bQ_{\geq 0}^{r}\to A_{\mathrm{inf}}(\widetilde{R}_{\infty,\alpha})$ defined by
     \[
     se_{i}\mapsto
     \begin{cases}
         [(y_{i}^{s})^{\flat}] \ \ (1\leq i\leq m) \\
         [(z_{i-m}^{s})^{\flat}] \ \ (m+1\leq i\leq r). 
     \end{cases}
     \]
     
     We can define a log prism map $(\fS_{R},(E),\cM_{\fS_{R}})\to (A_{\mathrm{inf}}(\widetilde{R}_{\infty,\alpha}),(\xi),\cM_{A_{\mathrm{inf}}(\widetilde{R}_{\infty,\alpha})})$ as follows. Consider a map of prelog rings $(\fS_{R^{0}},\bN^{r})\to (A_{\mathrm{inf}}(\widetilde{R}_{\infty,\alpha}),\bQ_{\geq 0}^{r})$ defined by
     \begin{align*}
         x_{i}&\mapsto [x_{i}^{\flat}] \\
         y_{j}&\mapsto [y_{j}^{\flat}] \\
         z_{k}&\mapsto [z_{k}^{\flat}] 
     \end{align*}
     and the natural inclusion $\bN^{r}\to \bQ_{\geq 0}^{r}$. Since $\fS_{R^{0}}\to \fS_{R}$ is \'{e}tale, the above map lifts uniquely to a map of prelog rings $(\fS_{R},\bN^{r})\to (A_{\mathrm{inf}}(\widetilde{R}_{\infty,\alpha}),\bQ_{\geq 0}^{r})$, which gives a log prism map
     \[
     (\fS_{R},(E),\cM_{\fS_{R}})\to (A_{\mathrm{inf}}(\widetilde{R}_{\infty,\alpha}),(\xi),\cM_{A_{\mathrm{inf}}(\widetilde{R}_{\infty,\alpha})})
     \]
     by taking the associated log structures.
\end{construction}

\begin{rem}\label{rem on another construction of bkprism}
    Let $(\mathrm{Spf}(R),\cM_{R})$ be a small affine log formal scheme over $\cO_{K}$ admitting a fixed framing with $n=0$. Then $(\mathrm{Spf}(R),\cM_{R}\oplus \pi^{\bN})^{a}$ is also small affine and admits an induced framing with $n=1$. Let $(\mathrm{Spf}(R_{0}),\cM_{R_{0}})$ be the unramified model of $(\mathrm{Spf}(R),\cM_{R})$ defined by the fixed framing and $\varphi_{R_{0}}$ be the Frobenius lift on $(\mathrm{Spf}(R_{0}),\cM_{R_{0}})$ defined in Definition \ref{def of unr model}. 

    Consider a log formal scheme $(\mathrm{Spf}(R_{0}[[t]]),\cM_{R_{0}[[t]]})\coloneqq (\mathrm{Spf}(R_{0}[[t]]),\cM_{R}\oplus t^{\bN})^{a}$. The map $\varphi_{R_{0}}$ induces a Frobenius lift $\varphi_{R_{0}[[t]]}$ on $(\mathrm{Spf}(R_{0}[[t]]),\cM_{R_{0}[[t]]})$ given by $t\mapsto t^{p}$, which induces a $\delta$-ring structure on $R_{0}[[t]]$. This is refined to a $\delta_{\mathrm{log}}$-ring structure on $(\mathrm{Spf}(R_{0}[[t]]),\cM_{R_{0}[[t]]})$ defined by $\delta_{\mathrm{log}}(y_{j})=\delta_{\mathrm{log}}(t)=0$. Then a triple $(R_{0}[[t]],(E_{t}),\cM_{R_{0}[[t]]})$ is a log prism which is isomorphic to the Breuil-Kisin log prism $(\fS_{R},(E),\cM_{\fS_{R}})$ associated with $(\mathrm{Spf}(R),\cM_{R}\oplus \pi^{\bN})^{a}$. If we replace $\cM_{R}\oplus t^{\bN}$ with $\cM_{R}$ at the first point, we can obtain the Breuil-Kisin log prism associated with $(\mathrm{Spf}(R),\cM_{R})$.
\end{rem}

\begin{lem}[cf. {\cite[Lemma 2.7 and Lemma 2.8]{dlms24}}]\label{bk prism for semi-stable reduction}

Let $(\fX,\cM_{\fX})=(\mathrm{Spf}(R),\cM_{R})$ be a small affine log formal scheme over $\cO_{K}$ with a fixed framing, and $(A,I,\cM_{A})$ be an orientable log prism in $(\fX,\cM_{\fX})_{\Prism}^{\mathrm{str}}$.
Then there exists a coproduct in $(\fX,\cM_{\fX})_{\Prism}$
\[
(B,J,\cM_{B})\coloneqq (\fS_{R},(E),\cM_{\fS_{R}})\sqcup  (A,I,\cM_{A})
\]
belonging to $(\fX,\cM_{\fX})_{\Prism}^{\mathrm{str}}$, and a map $(A,I,\cM_{A})\to (B,J,\cM_{B})$ is a flat cover. In particular, $(\fS_{R},(E),\cM_{\fS_{R}})$ covers the final object of the topos associated to $(\fX,\cM_{\fX})_{\Prism}^{\mathrm{str}}$. 
\end{lem}

\begin{proof}
    By \cite[Lemma 3.8]{kos22}, $\Gamma(\mathrm{Spf}(A),\cM_{A})\to \Gamma(\mathrm{Spf}(A/I),\cM_{A/I})$ is surjective. Hence, a map $\bN^{r}\to \Gamma(\fX,\cM_{\fX})\to \Gamma(\mathrm{Spf}(A/I),\cM_{A/I})$ admits a lift $\bN^{r}\to \Gamma(\mathrm{Spf}(A),\cM_{A})$. Fix such a lift and let $(A,M_{A})$ denote the prelog ring associated with the fixed lift. Here, $M_{A}=\bN^{r}$ as monoids. Then $(A,I,\cM_{A})=(A,I,M_{A})^{a}$. Let
    \[ C\coloneqq A[s_{1}^{\pm},\dots,s_{l}^{\pm},t_{1},\dots,t_{m},u_{1},\dots,u_{n}][v]/(\prod_{k=1}^{n}u_{k}-v)
    \]
    Let $(C,M_{C})$ the associated prelog ring where $M_{C}\coloneqq M_{A}\oplus \bN^{r}\to C$ is given by $M_{A}\to A\to C$ and 
    \[
    e_{i}\mapsto 
    \begin{cases}
        t_{i} \ \ (1\leq i\leq m)\\
        u_{i-m} \ \ (m+1\leq i\leq r).
    \end{cases}
    \]
    The prelog ring $(C,M_{C})$ is equipped the unique $\delta_{\mathrm{log}}$-structure such that the natural map $(A,M_{A})\to (C,M_{C})$ is compatible with $\delta_{\mathrm{log}}$-structures and $\delta(s_{i})=\delta(t_{j})=\delta(u_{k})=\delta_{\mathrm{log}}(e_{i})=0$ for every $i,j,k$.
    Consider the surjection $(C,M_{C})\to (A/I,M_{A})$ induced by $s_{i}\mapsto x_{i}$, $t_{j}\mapsto y_{j}$, and $u_{k}\mapsto z_{k}$, and let
    \[
    (C,M_{C})\to (C',M_{C'})\to (A/I,M_{A})
    \]
    be the exactification of it. Explicitly, we have
    \begin{align*}
        C'&\coloneqq C[(\widetilde{y}_{1}/t_{1})^{\pm 1},\dots,(\widetilde{y}_{m}/t_{m})^{\pm 1},(\widetilde{z}_{1}/u_{1})^{\pm 1},(\widetilde{z}_{n}/u_{n})^{\pm 1}] \\ &\coloneqq C[T_{1}^{\pm 1},\dots,T_{m}^{\pm 1},U_{1}^{\pm 1},\dots,U_{n}^{\pm 1}]/(t_{j}T_{j}-\widetilde{y}_{j},u_{k}U_{k}-\widetilde{z}_{k})_{j,k},
    \end{align*}
    where $\widetilde{y}_{j}$ (resp. $\widetilde{z}_{k}$) is the image of $e_{j}$ (resp. $e_{m+k}$) via $\bN^{r}=M_{A}\to A$ for every $j,k$. The $\delta_{\mathrm{log}}$-structure on $(C,M_{C})$ uniquely extends to $(C',M_{C'})$ by \cite[Proposition 2.16]{kos22}. Fix a lift $\widetilde{x}_{i}\in A$ of $x_{i}\in A/I$ for each $i$. Let $J$ denote the kernel of the surjection $C'\to A/I$. Then $J$ is generated by $I$ and $(s_{i}-\widetilde{x}_{i},\widetilde{y}_{j}/t_{j}-1,\widetilde{z}_{k}/u_{k}-1)_{i,j,k}$. Let $(C')^{\wedge}$ denote the derived $(p,I)$-adic completion of $C'$, which coincides with the classical $(p,I)$-completion because $C'$ is flat over $A$ (cf.~the proof of \cite[Proposition 3.9]{kos22}), and $(B,IB)$ denote the prismatic envelope of the $\delta$-pair $((C')^{\wedge},J(C')^{\wedge})$ over $(A,I)$. Since $(s_{i}-\widetilde{x}_{i},\widetilde{y}_{j}/t_{j}-1,\widetilde{z}_{k}/u_{k}-1)_{i,j,k}$ is a regular sequence in $(C')^{\wedge}$ and the natural map of derived quotient rings
    \[
    A/^{\bL}(p,d)\to (C')^{\wedge}/^{\bL}(p,d,s_{i}-\widetilde{x}_{i},\widetilde{y}_{j}/t_{j}-1,\widetilde{z}_{k}/u_{k}-1)_{i,j,k}
    \]
    is an isomorphism, the prism map $(A,I)\to (B,IB)$ is a flat cover by \cite[Proposition 3.13 (1)]{bs22} or \cite[Proposition 2.6.6 (1)]{ito24}. 

    We define a prelog ring map $(\fS_{R^{0}},(E),\bN^{r})\to (B,IB,M_{C'})$ by $x_{i}\mapsto s_{i}$, $y_{j}\mapsto t_{J}$, $z_{k}\mapsto u_{k}$, and $t\mapsto v$. Since $E(v)\in J$ implies that $v^{e}\in (p,I)B$, this map is well-defined. Using the \'{e}taleness of $\fS_{R^{0}}\to \fS_{R}$, we obtain a unique lift $(\fS_{R},(E),\bN^{r})\to (B,IB,M_{C'})$ inducing the natural map $(R,\bN^{r})\to (B/IB,M_{A})$. Then $(B,IB,M_{C'})\in (\mathrm{Spf}(R),\cM_{R})_{\Prism}^{\mathrm{str}}$ is the coproduct of $(\fS_{R},(E),\bN^{r})^{a}$ and $(A,I,\cM_{A})$ in $(\mathrm{Spf}(R),\cM_{R})_{\Prism}$.
\end{proof}

For a small affine log formal scheme $(\fX,\cM_{\fX})=(\mathrm{Spf}(R),\cM_{R})$ over $\cO_{K}$ with a fixed framing, the $n+1$--fold self coproduct of $(\fS_{R},(E),\cM_{\fS_{R}})$ in $(\fX,\cM_{\fX})_{\Prism}$ is denoted by $(\fS_{R}^{(n)},(E),\cM_{\fS_{R}^{(n)}})(\in (\fX,\cM_{\fX})_{\Prism}^{\mathrm{str}})$ for each $n\geq 0$.

\begin{lem}\label{fully faithfulness from pris crys to anal pris crys}
    Let $(\fX,\cM_{\fX})$ be a semi-stable log formal scheme over $\cO_{K}$. Then the restriction functor
    \[
    \mathrm{Vect}^{\star}((\fX,\cM_{\fX})_{\Prism})\to \mathrm{Vect}^{\mathrm{an},\star}((\fX,\cM_{\fX})_{\Prism})
    \]
    is fully faithful for $\star\in \{\emptyset,\varphi\}$.
\end{lem}

\begin{proof}
    By working \'{e}tale locally on $\fX$, we may assume that $(\fX,\cM_{\fX})=(\mathrm{Spf}(R),\cM_{R})$ is small affine with a fixed framing. The natural functors
    \begin{align*}
        \mathrm{Vect}^{\star}((\fX,\cM_{\fX})_{\Prism})&\to \varprojlim_{\bullet\in \Delta} \mathrm{Vect}^{\star}(\fS_{R}^{(\bullet)},(E)), \\
        \mathrm{Vect}^{\mathrm{an},\star}((\fX,\cM_{\fX})_{\Prism})&\to \varprojlim_{\bullet\in \Delta} \mathrm{Vect}^{\mathrm{an},\star}(\fS_{R}^{(\bullet)},(E))
    \end{align*}
    give equivalences by Lemma \ref{bk prism for semi-stable reduction}. For each $n\geq 0$, the natural functor
    \[
    \mathrm{Vect}^{\star}(\fS_{R}^{(n)},(E))\to \mathrm{Vect}^{\mathrm{an},\star}(\fS_{R}^{(n)},(E))
    \]
    is fully faithful by \cite[0G7P]{sp24}. Therefore, the functor in the statement is fully faithful.
\end{proof}

\begin{dfn}(\cite[Definition 3.7]{dlms24})
    Let $\mathrm{DD}_{\fS_{R}}$ be the category of pairs $(\fM,f)$ where
    \begin{itemize}
        \item $\fM$ is a finitely generated $\fS_{R}$-module such that, $\fM[1/p]$ (resp. $\fM[1/E]$) is finite projective over $\fS_{R}[1/p]$ (resp. $\fS_{R}[1/E]$), and $\fM=\fM[1/p]\cap \fM[1/E]$;
        \item $f:\fM\otimes_{\fS_{R},p_{1}} \fS_{R}^{(1)}\isom \fM\otimes_{\fS_{R},p_{2}} \fS_{R}^{(1)}$ is an isomorphism of $\fS_{R}^{(1)}$-modules satisfying the cocycle condition over $\fS_{R}^{(2)}$.
    \end{itemize}
    Let $\mathrm{DD}^{\varphi}_{\fS_{R}}$ be the category of pairs $(\fM,\varphi_{\fM},f)$ where
    \begin{itemize}
        \item $(\fM,f)$ belongs to $\mathrm{DD}_{\fS_{R}}$;
        \item $\varphi_{\fM}\colon (\phi^{*}\fM)[1/E]\isom \fM[1/E]$ is an isomorphism of $\fS_{R}[1/E]$-modules such that $f$ is compatible with $\varphi_{\fM}$.
    \end{itemize}
\end{dfn}

\begin{lem}\label{anal pris crys via bk prism}
    For $\star\in \{\emptyset,\varphi\}$, we have natural bi-exact equivalences
    \[
    \mathrm{Vect}^{\mathrm{an},\star}((\fX,\cM_{\fX})_{\Prism})\isom \mathrm{DD}^{\star}_{\fS_{R}}
    \]
    sending $\cE$ to $\Gamma(U(\fS_{R},(E)),\cE_{(\fS_{R},(E),\cM_{\fS_{R}})})$.
\end{lem}

\begin{proof}
    This follows from the same argument as \cite[Lemma 3.8]{dlms24}.
\end{proof}

\subsection{``Kummer quasi-syntomic descent'' for analytic log prismatic \texorpdfstring{$F$}--crystals}

\begin{lem}\label{p-div implies deltalog is zero}
    Let $(A,I,M)$ be a bounded prelog prism. For $m\in M$ admitting a $p^{n}$-th root for every $n\geq 1$, we have $\delta_{\mathrm{log}}(m)=0$.
\end{lem}

\begin{proof}
    For $x\in M$, we have a formula
    \[
    \delta_{\mathrm{log}}(x^{p})=\sum_{i=1}^{p}\binom{p}{i}p^{i-1}\delta_{\mathrm{log}}(x)^{i},
    \]
    which can be checked by reducing the problem to the case that $A$ is a free $\delta_{\mathrm{log}}$-ring (cf. \cite[Lemma 2.11]{kos22}). By this formula, $\delta_{\mathrm{log}}(x)\in p^{i}A$ implies $\delta_{\mathrm{log}}(x^{p})\in p^{i+1}A$. Hence, $\delta_{\mathrm{log}}(m)$ is contained in $\displaystyle \bigcap_{i\geq 1}p^{i}A$ by the assumption. Since $A$ is $p$-adically separated by the boundedness, we obtain $\delta_{\mathrm{log}}(m)=0$.
\end{proof}

\begin{lem}\label{divisible monoid and log pris site}
    Let $(\fX,\cM_{\fX})$ be a bounded $p$-adic integral log formal scheme. Assume that $(\fX,\cM_{\fX})$ admits a chart $M\to \cM_{\fX}$ where $M$ is uniquely $p$-divisible (see Lemma \ref{def of div monoid} for the definition). Then the natural functor
    \[
    (\fX,\cM_{\fX})^{\mathrm{str}}_{\Prism}\to \fX_{\Prism}
    \]
    sending $(A,I,\cM_{A})$ to $(A,I)$ is an equivalence.

    In particular, for $\star\in \{\emptyset,\varphi,\mathrm{an},(\mathrm{an},\varphi)\}$, the following functors give bi-exact equivalences:
    \[
    \mathrm{Vect}^{\star}(\fX_{\Prism})\to \mathrm{Vect}^{\star}((\fX,\cM_{\fX})_{\Prism}).
    \]
\end{lem}

\begin{proof}
    The latter assertion follows from the former and Lemma \ref{category of crystals is unchanged}. 
    
    We shall prove the former. We construct a quasi-inverse functor. Let $(A,I)\in \fX_{\Prism}$. We let $f$ denote the map 
    \[
    M\to \Gamma(\fX,\cM_{\fX})\to \Gamma(\mathrm{Spf}(A/I),\cM_{A/I})\to A/I.
    \]
    Since $M$ is uniquely $p$-divisible, there exists a unique lift $\widetilde{f}\colon M\to A$ of $f$ by Lemma \ref{lift of map from div monoid}. By setting $\delta_{\mathrm{log}}=0$, the map $\widetilde{f}$ defines a prelog prism $(A,I,M)$, and $(A,I,M)^{a}$ belongs to $(\fX, \cM_{\fX})^{\mathrm{str}}_{\Prism}$. 

    We shall check the functor sending $(A,I)$ to $(A,I,M)^{a}$ constructed above is indeed a quasi-inverse functor. As a non-trivial direction, it suffices to prove that, for $(A,I,\cM_{A})\in (\fX, \cM_{\fX})^{\mathrm{str}}_{\Prism}$, there exists a natural isomorphism $(A,I,M)^{a}\isom (A,I,\cM_{A})$, where $(A,I,M)$ is a prelog prism defined above. By Lemma \ref{lift of div log str} (1), we have a unique map of log structure $M^{a}\to \cM_{A}$ lifting the identity map on $\cM_{A/I}$. This map $M^{a}\to \cM_{A}$ is an isomorphism because so is the restriction to $\mathrm{Spf}(A/I)$. It follows from Lemma \ref{p-div implies deltalog is zero} that this map is compatible with $\delta_{\mathrm{log}}$-structures.
\end{proof}

To state ``Kummer quasi-syntomic descent'' for analytic log prismatic $F$-crystals, we clarify our setup. 

\begin{setting}\label{setting for kqsyn descent for anal pris crys}
    Let $(\fX,\cM_{\fX})=(\mathrm{Spf}(R),\cM_{R})$ be a small affine log formal scheme over $\cO_{K}$ with a fixed framing. Let $(\fX_{\infty,\alpha},\cM_{\fX_{\infty,\alpha}})\coloneqq (\mathrm{Spf}(R_{\infty,\alpha}),\cM_{R_{\infty,\alpha}})$.

Consider a bounded $p$-adic fs log formal scheme $(\fY,\cM_{\fY})$ and a  strict map $(\fY,\cM_{\fY})\to  (\fX_{\infty,\alpha},\cM_{\fX_{\infty,\alpha}})$. Let $(\fY^{(\bullet)},\cM_{\fY^{(\bullet)}})$ be the \v{C}ech nerve of $(\fY,\cM_{\fY})$ over $(\fX,\cM_{\fX})$ in the category of $p$-adic saturated log formal schemes. Then the pullback functor along $(\fY,\cM_{\fY})\to (\fX,\cM_{\fX})$ gives a functor
\begin{align*}
    \mathrm{Vect}^{\mathrm{an},\varphi}((\fX,\cM_{\fX})_{\Prism})&\to \varprojlim_{\bullet\in \Delta} \mathrm{Vect}^{\mathrm{an},\varphi}((\fY^{(\bullet)},\cM_{\fY^{(\bullet)}})_{\Prism}) \\
    &\simeq \varprojlim_{\bullet\in \Delta} \mathrm{Vect}^{\mathrm{an},\varphi}((\fY^{(\bullet)},\cM_{\fY^{(\bullet)}})^{\mathrm{str}}_{\Prism}) \\
    &\simeq \varprojlim_{\bullet\in \Delta}\mathrm{Vect}^{\mathrm{an},\varphi}(\fY^{(\bullet)}_{\Prism})
\end{align*}
by Lemma\ref{category of crystals is unchanged} and Lemma \ref{divisible monoid and log pris site}.
\end{setting}

\begin{prop}\label{kqsyn descent for anal pris crys}
    Consider the setting as in Setting \ref{setting for kqsyn descent for anal pris crys}. Suppose that $(\fY,\cM_{\fY})\to (\fX_{\infty,\alpha},\cM_{\fX_{\infty,\alpha}})$ is strict quasi-syntomic and that $(\fY,\cM_{\fY})\to (\fX,\cM_{\fX})$ is surjective. Then the functor constructed in  Setting \ref{setting for kqsyn descent for anal pris crys}
    \[
    \mathrm{Vect}^{\mathrm{an},\varphi}((\fX,\cM_{\fX})_{\Prism})\to \varprojlim_{\bullet\in \Delta} \mathrm{Vect}^{\mathrm{an},\varphi}((\fY^{(\bullet)})_{\Prism})
    \]
    is fully faithful and gives a bi-exact equivalence to the essential image.
\end{prop}

\begin{proof}
First, we shall prove fully faithfulness. Let $(A,I,\cM_{A})\in (\fX,\cM_{\fX})_{\Prism}^{\mathrm{str}}$. In the same way as the proof of \cite[Proposition 4.18]{ino25}, replacing $(A,I,\cM_{A})$ with a flat cover allows us to assume that the following conditions hold:
    \begin{itemize}
        \item $(A,I,\cM_{A})$ is orientable;
        \item there is a chart $\widetilde{\alpha}\colon \bN^{r}\to \cM_{A}$ lifting a map $\bN^{r}\to \cM_{A/I}$ induced from the given chart;
        \item there exists a log prism $(B,IB,\cM_{B})\in (\fY,\cM_{\fY})_{\Prism}^{\mathrm{str}}$ with log prism maps
    \[
    (A,I,\cM_{A})\to (A_{\infty,\widetilde{\alpha}},IA_{\infty,\widetilde{\alpha}},\cM_{A_{\infty,\widetilde{\alpha}}})\to (B,IB,\cM_{B})
    \]
    in $(\fX,\cM_{\fX})_{\Prism}$ such that $(\mathrm{Spf}(B),\cM_{B})\to (\mathrm{Spf}(A_{\infty,\widetilde{\alpha}}),\cM_{A_{\infty,\widetilde{\alpha}}})$ is strict flat and $\mathrm{Spf}(B)\to (\mathrm{Spf}(A))$ is surjective.
    \end{itemize}
    To make the same argument as the proof of \cite[Proposition 4.18]{ino25} work, it suffices to prove that the natural functor
    \[
    \mathrm{Vect}^{\mathrm{an},\varphi}(A,I)\to \varprojlim_{\bullet\in \Delta} \mathrm{Vect}^{\mathrm{an},\varphi}(B^{(\bullet)},IB^{(\bullet)})
    \]
    is fully faithful. Let $I=(d)$. Considering internal homomorphisms, we reduce the problem to showing that, for $\cE\in \mathrm{Vect}^{\mathrm{an},\varphi}(A,I)$, the sequence
    \[
    0\to\Gamma(U(A,I),\cE)\to \Gamma(U(B^{(0)},IB^{(0)}),\cE_{B^{(0)}})\to \Gamma(U(B^{(1)},IB^{(1)}),\cE_{B^{(1)}})
    \]
    is exact, where $\cE_{B^{(n)}}$ denotes the pullback of $\cE$ along $U(B^{(n)},IB^{(n)})\to U(A,I)$ for each $n\geq 0$. Since $U(A,I)$ has a covering $\mathrm{Spec}(A[1/p])\cap \mathrm{Spec}(A[1/d])$ with the intersection $\mathrm{Spec}(A[1/pd])$, it suffices to prove that the sequence
    \[
    0\to \Gamma(\mathrm{Spec}(A[1/f]),\cE)\to \Gamma(\mathrm{Spec}(B^{(0)}[1/f]),\cE_{B^{(0)}})\to \Gamma(\mathrm{Spec}(B^{(1)}[1/f]),\cE_{B^{(1)}})
    \]
    is exact for $f\in \{p,d,pd\}$. By \cite[Proposition 2.26]{ino25}, we have an exact sequence
    \[
    0\to A\to B^{(0)}\to B^{(1)}.
    \]
    Taking tensor products with $\Gamma(\mathrm{Spec}(A[1/f]),\cE)$, we obtain an exact sequence 
    \[
    0\to \Gamma(\mathrm{Spec}(A[1/f]),\cE)\to \Gamma(\mathrm{Spec}(B^{(0)}[1/f]),\cE_{B^{(0)}})\to \Gamma(\mathrm{Spec}(B^{(1)}[1/f]),\cE_{B^{(1)}}).
    \]
    This proves the fully faithfulness.

    Next, we shall prove bi-exactness. Lemma \ref{bk prism for semi-stable reduction} implies that we can check the exactness of a sequence of analytic log prismatic $F$-crystals by evaluating at the Breuil-Kisin log prism $(\fS_{R},(E),\cM_{\fS_{R}})$. Hence, it suffices to prove that the natural functor
    \[
    \mathrm{Vect}^{\mathrm{an},\varphi}(A,I)\to \varprojlim_{\bullet\in \Delta} \mathrm{Vect}^{\mathrm{an},\varphi}(B^{(\bullet)},IB^{(\bullet)})
    \]
    gives a bi-exact equivalence to the essential image for $(A,I,\cM_{A})\coloneqq(\fS_{R},(E),\cM_{\fS_{R}})$. Since $A$ is noetherian, the $(p,I)$-completely faithfully flat map $A\to B^{(0)}$ is a faithfully flat ring map. Therefore, the morphism of schemes $U(B^{(0)},IB^{(0)})\to U(A,I)$ is also faithfully flat. This proves the assertion.
    \end{proof}

\section{Realization functors for log prismatic crystals}\label{section realization functor}

In this section, we define \'{e}tale, crystalline, and de Rham realization functors for analytic log prismatic $F$-crystals on a semi-stable log formal scheme over $\cO_{K}$ and study their properties.

\subsection{The \'{e}tale realization functor for log prismatic \texorpdfstring{$F$}--crystals}

\begin{dfn}[Laurent $F$-crystals, {\cite[Definition 7.35]{ky23}}] 
Let $(\fX,\cM_{\fX})$ be a bounded $p$-adic log formal scheme. Let $\cO_{\Prism}[1/\cI_{\Prism}]^{\wedge}_{p}$ denote the sheaf on $(\fX,\cM_{\fX})_{\Prism}$ given by $(A,I,\cM_{A})\mapsto A[1/I]^{\wedge}_{p}$ and $\phi\colon \cO_{\Prism}[1/\cI_{\Prism}]^{\wedge}_{p}\to \cO_{\Prism}[1/\cI_{\Prism}]^{\wedge}_{p}$ denote the ring map induced by a ring map $\phi_{A}$ for each log prism $(A,I,\cM_{A})\in (\fX,\cM_{\fX})_{\Prism}$. A vector bundle $\cE$ on the ringed site $((\fX,\cM_{\fX})_{\Prism},\cO_{\Prism}[1/\cI_{\Prism}]^{\wedge}_{p})$ equipped with an isomorphism $\varphi_{\cE}\colon \phi^{*}\cE\isom \cE$ is called a \emph{Laurent $F$-crystal} on $(\fX,\cM_{\fX})$.
Let $\mathrm{Vect}^{\varphi}((\fX,\cM_{\fX})_{\Prism},\cO_{\Prism}[1/\cI_{\Prism}]^{\wedge}_{p})$ denote the category of Laurent $F$-crystals on $(\fX,\cM_{\fX})$.
\end{dfn}

\begin{rem}\label{Laurent crys as lim}
    For a prism $(A,I)$, we let $\mathrm{Vect}^{\varphi}(A[1/I]^{\wedge}_{p})$ denote the category of a finite projective $A[1/I]^{\wedge}_{p}$-module $M$ with an isomorphism $\phi_{A}^{*}M\isom M$, where the ring map $\phi_{A}\colon A[1/I]^{\wedge}_{p}\to A[1/I]^{\wedge}_{p}$ is induced by the ring map $\phi_{A}\colon A\to A$. By the same argument as in \cite[Proposition 2.7]{bs23}, there exists a natural equivalence
    \[
    \mathrm{Vect}^{\varphi}((\fX,\cM_{\fX})_{\Prism},\cO_{\Prism}[1/\cI_{\Prism}]^{\wedge}_{p})\simeq \varprojlim_{(A,I,\cM_{A})\in (\fX,\cM_{\fX})_{\Prism}} \mathrm{Vect}^{\varphi}(A[1/I]^{\wedge}_{p}).
    \]
    For a Laurent $F$-crystal $\cE$ on $(\fX,\cM_{\fX})$ and a log prism $(A,I,\cM_{A})\in (\fX,\cM_{\fX})_{\Prism}$, the projection of $\cE$ to $\mathrm{Vect}^{\varphi}(A[1/I]^{\wedge}_{p})$ via the above functor is denoted by $\cE_{(A,I,\cM_{A})}$. When no confusion occurs, we write $\cE_{A}$ for this.
\end{rem}

\begin{thm}[\cite{ky23}, Theorem 7.36]\label{etale realization} 
    For a bounded $p$-adic fs log formal scheme $(\fX,\cM_{\fX})$, there exists a natural equivalence
    \[
    T_{\et}\colon \mathrm{Vect}^{\varphi}((\fX,\cM_{\fX})_{\Prism},\cO_{\Prism}[1/\cI_{\Prism}]^{\wedge}_{p})\isom \mathrm{Loc}_{\bZ_{p}}((\fX,\cM_{\fX})_{\eta,\mathrm{qprok\et}}^{\Diamond}),
    \]
    where $(\fX,\cM_{\fX})_{\eta}^{\Diamond}$ is the log diamond of the generic fiber of $(\fX,\cM_{\fX})$ ( \cite[Example 7.6]{ky23}). 
\end{thm}

\begin{proof}
    See \cite[Theorem 7.36]{ky23}. Although the definition of absolute prismatic sites in \cite[Definition 7.33]{ky23} is slightly different from Definition \ref{def of log pris site}, the category of Laurent $F$-crystals is unchanged under this modification by Lemma \ref{category of crystals is unchanged} (see also Remark \ref{rem on def of log pris site}).
\end{proof}

\begin{rem}\label{Laurent isom}
    Let $(\fX,\cM_{\fX})$ be a bounded $p$-adic fs log formal scheme and $\cE$ be a Laurent $F$-crystal on $(\fX,\cM_{\fX})$. By the proof of \cite[Theorem 7.36]{ky23}, there exists a functorial isomorphism of $A[1/I]^{\wedge}_{p}$-modules
    \[
    \cE_{A}\cong T_{\et}(\cE)|_{(\mathrm{Spd}(S,S^{+}),\cM_{S})}\otimes_{\bZ_{p}} A[1/I]^{\wedge}_{p}
    \]
    for a perfect log prism $(A,I,\cM_{A})\in (\fX,\cM_{\fX})_{\Prism}$ associated with a perfect indiscrete log prism in the sense of \cite[Definition 2.36 and Definition 2.39]{ky23} satisfying the following conditions:
    \begin{itemize}
        \item the pullback of $T_{\et}(\cE)$ to $(\mathrm{Spd}(S,S^{+}),\cM_{S})\coloneqq (\mathrm{Spf}(A/I),\cM_{A/I})^{\Diamond}_{\eta}$ is constant (we simply write $ T_{\et}(\cE)|_{(\mathrm{Spd}(S,S^{+}),\cM_{S})}$ for the corresponding $\bZ_{p}$-module);
        \item $\cM_{S}/\cM_{S}^{\times}$ is divisible.
    \end{itemize}
\end{rem}

\begin{dfn}
    Let $(\fX,\cM_{\fX})$ be a bounded $p$-adic fs log formal scheme. By Remark \ref{pris crys as lim} and Remark \ref{Laurent crys as lim}, we have a sequence of exact functors
    \[
    \mathrm{Vect}^{\varphi}((\fX,\cM_{\fX})_{\Prism})\to \mathrm{Vect}^{\mathrm{an},\varphi}((\fX,\cM_{\fX})_{\Prism})\to \mathrm{Vect}^{\varphi}((\fX,\cM_{\fX})_{\Prism},\cO_{\Prism}[1/\cI_{\Prism}]^{\wedge}_{p}).
    \]
    The exact functors
    \begin{align*}
        \mathrm{Vect}^{\mathrm{an},\varphi}((\fX,\cM_{\fX})_{\Prism})&\to \mathrm{Loc}_{\bZ_{p}}((\fX,\cM_{\fX})_{\eta,\mathrm{qprok\et}}^{\Diamond}) \\
        \mathrm{Vect}^{\varphi}((\fX,\cM_{\fX})_{\Prism})&\to \mathrm{Loc}_{\bZ_{p}}((\fX,\cM_{\fX})_{\eta,\mathrm{qprok\et}}^{\Diamond})
    \end{align*}
    obtained as the composition of the above functor and the equivalence in Proposition \ref{etale realization} are also denoted by $T_{\et}$ and called \emph{\'{e}tale realization functors}. 
\end{dfn}

\begin{rem}\label{et real without diamond}
    When $(\fX,\cM_{\fX})_{\eta}$ is a locally noetherian fs log adic space, the target category of the \'{e}tale realization functor can be identified with the category of $\bZ_{p}$-local systems on $(\fX,\cM_{\fX})_{\eta,\mathrm{k\et}}$ (or $(\fX,\cM_{\fX})_{\eta,\mathrm{prok\et}}$) by Proposition \ref{ket vs qproket}.
\end{rem}

\subsection{The crystalline realization functor for log prismatic \texorpdfstring{$F$}--crystals}\label{subsec of crys crys}

\begin{construction}[Breuil log prism]\label{construction of breuil log prism}

Let $(\fX,\cM_{\fX})=(\mathrm{Spf}(R),\cM_{R})$ be a small affine $p$-adic fs log formal scheme and fix a framing $t$. We have a strict closed immersion
\[
(\mathrm{Spf}(R),\cM_{R})\hookrightarrow (\mathrm{Spf}(\fS_{R,t}),\cM_{\fS_{R,t}}).
\]
Let $(\mathrm{Spf}(S_{R,t}),\cM_{S_{R,t}})$ denote a $p$-completed log PD-envelope of it. More generally, for an integer $i\geq 0$, let $(\mathrm{Spf}(S^{(i)}_{R,t}),\cM_{S^{(i)}_{R,t}})$ denote a $p$-completed log PD-envelope of a closed immersion 
\[
(\mathrm{Spf}(R),\cM_{R})\hookrightarrow (\mathrm{Spf}(\fS_{R,t}),\cM_{\fS_{R,t}})^{(i)/W},
\]
where $(\mathrm{Spf}(\fS_{R,t}),\cM_{\fS_{R,t}})^{(i)/W}$ is the $i+1$--fold self product of $(\mathrm{Spf}(\fS_{R,t}),\cM_{\fS_{R,t}})$ over $W$. The $p$-adic log PD-thickening
$(S^{(i)}_{R,t}\twoheadrightarrow R/p,\cM_{S^{(i)}_{R,t}})$ belongs to $(X_{p=0},\cM_{X_{p=0}})_{\mathrm{crys}}^{\mathrm{str}}$. By construction, $(S_{R,t}\twoheadrightarrow R/p,\cM_{S_{R,t}})$ covers the final object of the topos associated with $(X_{p=0},\cM_{X_{p=0}})_{\mathrm{crys}}$, and $(S_{R,t}^{(i)},\cM_{S_{R,t}^{(i)}})$ is the $i$-th fold self product of $(S_{R,t},\cM_{S_{R,t}})$ in $(X_{p=0},\cM_{X_{p=0}})$. When no confusion occurs, we omit subscripts $t$.

Since $\mathrm{Ker}(\fS_{R}\twoheadrightarrow R)$ is generated by a single regular element $E(t)$, the ring $S_{R}^{(i)}$ is also $p$-torsion free for each $i\geq 0$. Hence, the Frobenius lift $\phi_{S_{R}^{(i)}}$ on $(\mathrm{Spf}(S_{R}^{(i)}),\cM_{S_{R}^{(i)}})$ induced from $\phi_{\fS_{R}^{(i)}}$ uniquely gives a structure of a $\delta$-ring on $S_{R}^{(i)}$. By \cite[Proposition 2.16]{kos22}, this $\delta$-structure uniquely refines to a $\delta_{\mathrm{log}}$-structure on $(\mathrm{Spf}(S_{R}^{(i)}),\cM_{S_{R}^{(i)}})$ such that the natural map 
\[
(\mathrm{Spf}(S_{R}^{(i)}),\cM_{S_{R}^{(i)}})\to (\mathrm{Spf}(\fS_{R}),\cM_{\fS_{R}})^{(i)/W}
\]
is compatible with $\delta_{\mathrm{log}}$-structures. As a result, we obtain a log prism $(S^{(i)}_{R},(p),\phi^{*}_{S_{R}^{(i)}}\cM_{S^{(i)}_{R}})\in (X_{p=0},\cM_{X_{p=0}})_{\Prism}^{\mathrm{str}}$, where the $\delta_{\mathrm{log}}$-structure on $(S^{(i)}_{R},\phi^{*}\cM_{S^{(i)}_{R}})$ is induced from that on $(S^{(i)}_{R},\cM_{S^{(i)}_{R}})$, and the ring map $\fS_{R}\stackrel{\phi_{\fS_{R}}}{\to} \fS_{R}\to S_{R}$ induces a log prism map in $(\fX,\cM_{\fX})_{\Prism}^{\mathrm{str}}$
\[
\phi\colon (\fS_{R}^{(i)},(E),\cM_{\fS_{R}^{(i)}})\to (S^{(i)}_{R},(p),\phi^{*}_{S_{R}^{(i)}}\cM_{S^{(i)}_{R}})
\]
for each $i\geq 0$. 
\end{construction}

\begin{rem}\label{rem on another construction of breuil ring}
Let $(\mathrm{Spf}(R),\cM_{R})$ be a small affine log formal scheme over $\cO_{K}$ admitting a fixed framing with $n=0$. Then $(\mathrm{Spf}(R),\cM_{R}\oplus \pi^{\bN})^{a}$ is also small affine and admits an induced framing with $n=1$. Let $(\mathrm{Spf}(R_{0}),\cM_{R_{0}})$ be the unramified model of $(\mathrm{Spf}(R),\cM_{R})$ defined by the fixed framing and $\varphi_{R_{0}}$ be the Frobenius lift on $(\mathrm{Spf}(R_{0}),\cM_{R_{0}})$ defined in Definition \ref{def of unr model}. 

Let $(\mathrm{Spf}(R_{0}[[t]]),\cM_{R_{0}[[t]]})$ as in Remark \ref{rem on another construction of bkprism}. Consider a strict closed immersion
\[
(\mathrm{Spf}(R),\cM_{R})\hookrightarrow (\mathrm{Spf}(R_{0}[[t]]),\cM_{R_{0}[[t]]})
\]
induced from $(\mathrm{Spf}(R),\cM_{R})\to (\mathrm{Spf}(R_{0}),\cM_{R_{0}})$ and $t\mapsto \pi$. By Remark \ref{rem on another construction of bkprism}, the $p$-completed log PD-envelope of it is nothing but the Breuil ring $(\mathrm{S_{R}},\cM_{S_{R}})$ associated with $(\mathrm{Spf}(R),\cM_{R}\oplus \pi^{\bN})^{a}$. Explicitly, we have an isomorphism $S_{R}\cong R_{0}[t][\{\frac{t^{en}}{(en)!}\}_{n\geq 1}]^{\wedge}$, where $e\coloneqq [K:K_{0}]$. We can also obtain the Breuil ring associated with $(\mathrm{Spf}(R),\cM_{R})$ in a similar way as in Remark \ref{rem on another construction of bkprism}.
\end{rem}

By using Breuil ring, we can prove the following.

\begin{prop}\label{local HK isom}
    Let $(\mathrm{Spf}(R),\cM_{R})$ be a small affine log formal scheme over $\cO_{K}$ admitting a fixed framing with $n=0$. Let $(\mathrm{Spf}(R_{0}),\cM_{R_{0}})$ be the unramified model defined by the fixed framing. Consider the Breuil ring $(\mathrm{Spf}(S_{R}),\cM_{S_{R}})$ associated with $(\mathrm{R},\cM_{R}\oplus \pi^{\bN})$ equipped with the induced framing. Then, for a locally free $F$-isocrystal $\cE$ on $(\mathrm{Spec}(R/p),\cM_{R/p}\oplus \pi^{\bN})$, there is a natural isomorphism
    \[
    \cE_{(S_{R}\twoheadrightarrow R/p,\cM_{S_{R}})}\cong \cE_{(R_{0}\twoheadrightarrow R_{0}/p,\cM_{R_{0}}\oplus 0^{\bN})^{a}}\otimes_{R_{0}[1/p]} S_{R}[1/p]
    \]
    that is compatible with Frobenius isomorphisms on both sides. In particular, there is a natural isomorphism
    \[
    \cE_{(R\twoheadrightarrow R/p,\cM_{R}\oplus \pi^{\bN})^{a}} \cong \cE_{(R_{0}\twoheadrightarrow R_{0}/p,\cM_{R_{0}}\oplus 0^{\bN})^{a}}\otimes_{R_{0}[1/p]} R[1/p].
    \]
\end{prop}

\begin{proof}
    The strict closed immersion
    \[
    (\mathrm{Spf}(R_{0}),\cM_{R_{0}}\oplus 0^{\bN})^{a}\hookrightarrow (\mathrm{Spf}(R_{0}[[t]]),\cM_{R_{0}}\oplus t^{\bN})
    \]
    given by $t\mapsto 0$ induces a strict closed immersion
    \[
    (\mathrm{Spf}(R_{0}),\cM_{R_{0}}\oplus 0^{\bN})^{a}\hookrightarrow (\mathrm{Spf}(S_{R}),\cM_{S_{R}})
    \]
    by the universal property of PD-envelopes and Remark \ref{rem on another construction of breuil ring}. The crystal property of $\cE$ along the above map gives an isomorphism
    \[
    \cE_{(S_{R}\twoheadrightarrow R/p,\cM_{S_{R}})}\otimes_{S_{R}[1/p]} R_{0}[1/p] \cong \cE_{(R_{0}\twoheadrightarrow R_{0}/p,\cM_{R_{0}}\oplus 0^{\bN})^{a}}.
    \]
    Therefore, by applying Lemma \ref{beilinson trick} to ring maps $R_{0}\to S_{R}\twoheadrightarrow R_{0}$, we obtain the former assertion. The latter one is proved by taking the former isomorphism along $S_{R}\twoheadrightarrow R$. 
\end{proof}

\begin{dfn}
    Let $\mathrm{DD}_{S_{R}}$ denote the category of pairs $(M,\epsilon)$ where
\begin{itemize}
    \item $M$ is a finitely presented $S_{R}$-module such that $M[1/p]$ is finite projective over $S_{R}[1/p]$;
    \item $\epsilon\colon M\otimes_{S_{R},p_{1}} S_{R}^{(1)}\isom M\otimes_{S_{R},p_{2}} S_{R}^{(1)}$ satisfying the cocycle condition over $S_{R}^{(2)}$.
\end{itemize}
Morphisms $(M,\epsilon)\to (M',\epsilon')$ in $\mathrm{DD}_{S_{R}}$ are $S_{R}$-linear morphisms $M\to M'$ compatible with $\epsilon$ and $\epsilon$'. We let $\mathrm{DD}^{\varphi}_{S_{R}}$ denote the category of triples $(M,\varphi_{M},\epsilon)$ where
\begin{itemize}
    \item $(M,\epsilon)$ belongs to $\mathrm{DD}_{S_{R}}$;
    \item $\varphi_{M}\colon (\phi_{S_{R}}^{*}M)[1/p]\isom M[1/p]$ is an $S_{R}[1/p]$-linear isomorphism that is compatible with $\epsilon$.
\end{itemize}
Morphisms $(M,\varphi_{M},\epsilon)\to (M',\varphi_{M'},\epsilon')$ in $\mathrm{DD}^{\varphi}_{S_{R}}$ are $S_{R}$-linear morphisms $M\to M'$ compatible with $\epsilon$, $\epsilon$', $\varphi_{M}$, and $\varphi_{N}$.
\end{dfn}

\begin{lem}\label{crys crys via breuil prism}
Let $(\fX,\cM_{\fX})=(\mathrm{Spf}(R),\cM_{R})$ be a small affine log formal scheme over $\cO_{K}$ with a fixed framing. Then there exist natural functors
\begin{align*}
    \mathrm{DD}_{S_{R}}&\to \mathrm{Crys}_{\mathrm{tf}}((X_{p=0},\cM_{X_{p=0}})_{\mathrm{crys}}), \\
    \mathrm{DD}^{\varphi}_{S_{R}}&\to \mathrm{Isoc}^{\varphi}_{\mathrm{lftf}}((X_{p=0},\cM_{X_{p=0}})_{\mathrm{crys}}).
\end{align*}
\end{lem}

\begin{proof}
It is enough to construct the former functor. Let $(M,\epsilon)\in \mathrm{DD}_{S_{R}}$. Let $\overline{M}$ denote the maximal $p$-torsion free quotient of $M$. Then $\overline{M}$ is $p$-complete by Lemma \ref{completeness is automatically}, and $\epsilon$ induces a $S_{R}^{(1)}$-linear isomorphism
\[
\overline{\epsilon}\colon \overline{M}\widehat{\otimes}_{S_{R},p_{1}} S_{R}^{(1)}\isom \overline{M}\widehat{\otimes}_{S_{R},p_{2}} S_{R}^{(1)}
\]
satisfying the cocycle condition over $S_{R}^{(2)}$ by Lemma \ref{lem for tf quot and completion} below. Since the $p$-adic log PD-thickening $(S_{R}\twoheadrightarrow R/p,\cM_{S_{R}})$ covers the final object of the topos associated with $(X_{p=0},\cM_{X_{p=0}})_{\mathrm{crys}}$, the pair $(\overline{M},\overline{\epsilon})$ gives an object of $\mathrm{Crys}_{\mathrm{tf}}((X_{p=0},\cM_{X_{p=0}})_{\mathrm{crys}})$. This construction gives the former functor in the statement. 
\end{proof}

\begin{lem}\label{lem for tf quot and completion}
    Let $A\to B$ be a $p$-adically flat ring map of $p$-complete and $p$-torsion free rings. Let $M$ be a finitely generated $A$-module and $\overline{M}$ be the maximal $p$-torsion free quotient of $M$ (which is $p$-complete by Lemma \ref{completeness is automatically}). Suppose that $M[1/p]$ is finite projective over $A[1/p]$. Then $\overline{M}\widehat{\otimes}_{A} B$ is isomorphic to the maximal $p$-torsion free quotient of $M\otimes_{A} B$.
\end{lem}

\begin{proof}
    Since $\overline{M}\otimes_{A} B$ is a finitely generated $B$-module, the natural map $\overline{M}\otimes_{A} B\to \overline{M}\widehat{\otimes}_{A} B$ is surjective. Since $A\to B$ is $p$-adically flat, $\overline{M}\widehat{\otimes}_{A} B$ is $p$-torsion free. By Lemma \ref{lemma for complete tensor and inverting p}, we have a natural isomorphism $(\overline{M}\widehat{\otimes}_{A} B)[1/p]\cong \overline{M}[1/p]\otimes_{A} B$. As a summary, we obtain the following commutative diagram:
    \[
    \begin{tikzcd}
        M\otimes_{A} B \ar[d] \ar[r,twoheadrightarrow] & \overline{M}\otimes_{A} B \ar[r,twoheadrightarrow] & \overline{M}\widehat{\otimes}_{A} B \ar[d,hook] \\
        M[1/p]\otimes_{A} B \ar[r,"\sim"] & \overline{M}[1/p]\otimes_{A} B \ar[r,"\sim"] & (\overline{M}\widehat{\otimes}_{A} B)[1/p].
    \end{tikzcd}
    \]
    This diagram implies the assertion.
\end{proof}

\begin{construction}[Crystalline realization]\label{crys real by using bk prisms}\noindent

Let $(\fX,\cM_{\fX})=(\mathrm{Spf}(R),\cM_{R})$ be a semi-stable log formal scheme over $\cO_{K}$. First, we shall construct a functor
\[
\mathrm{Vect}^{\mathrm{an}}((\fX,\cM_{\fX})_{\Prism})\to \mathrm{Crys}_{\mathrm{tf}}((X_{p=0},\cM_{X_{p=0}})_{\mathrm{crys}}).
\]
To do this, we may work \'{e}tale locally on $\fX$ and assume that $(\fX,\cM_{\fX})=(\mathrm{Spf}(R),\cM_{R})$ is a small affine log formal scheme with a fixed framing. We use notations in Construction \ref{construction of bk log prism} and Construction \ref{construction of breuil log prism}. By Lemma \ref{anal pris crys via bk prism}, there is an equivalence
\[
\mathrm{Vect}^{\mathrm{an}}((\fX,\cM_{\fX})_{\Prism})\to \mathrm{DD}_{\fS_{R}}.
\]
Taking the base change along log prism maps $\phi\colon (\fS_{R}^{(i)},(E),\cM_{\fS_{R}^{(i)}})\to (S_{R}^{(i)},(p),\cM_{S_{R}^{(i)}})$ gives a functor $\mathrm{DD}_{\fS_{R}}\to \mathrm{DD}_{S_{R}}$. Then, by composing these functors with the functor in Lemma \ref{crys crys via breuil prism}, we obtain a functor
\[
\mathrm{Vect}^{\mathrm{an}}((\fX,\cM_{\fX})_{\Prism})\to \mathrm{Crys}_{\mathrm{tf}}((X_{p=0},\cM_{X_{p=0}})_{\mathrm{crys}}),
\]
denoted by $T_{\mathrm{crys}}$. To globalize this construction, we need to check that this functor is independent of the choice of a framing. We can prove this in the same way as \cite[Lemma 3.31]{dlms24}.

Next, we shall construct a functor
\[
\mathrm{Vect}^{\mathrm{an},\varphi}((\fX,\cM_{\fX})_{\Prism})\to \mathrm{Isoc}^{\varphi}_{\mathrm{lftf}}((X_{p=0},\cM_{X_{p=0}})_{\mathrm{crys}}).
\]
For $\cE\in \mathrm{Vect}^{\mathrm{an},\varphi}((\fX,\cM_{\fX})_{\Prism})$, we have $T_{\mathrm{crys}}(\cE)\in \mathrm{Crys}_{\mathrm{tf}}((X_{p=0},\cM_{X_{p=0}})_{\mathrm{crys}})$. By construction, $T_{\mathrm{crys}}(\cE)$ is locally free. Hence, it is enough to construct a Frobenius structure $F^{*}T_{\mathrm{crys}}(\cE)[1/p]\isom T_{\mathrm{crys}}(\cE)[1/p]$. To do this, we may work \'{e}tale locally on $\fX$, and so we may assume that $(\fX,\cM_{\fX})=(\mathrm{Spf}(R),\cM_{R})$ is a small affine log formal scheme with a fixed framing. Taking the base change along log prism maps $\phi:(\fS_{R}^{(i)},(E),\cM_{\fS_{R}^{(i)}})\to (S_{R}^{(i)},(p),\phi^{*}\cM_{S_{R}^{(i)}})$ in Construction \ref{construction of breuil log prism} gives a functor
\[
\mathrm{DD}^{\varphi}_{\fS_{R}} \to \mathrm{DD}^{\varphi}_{S_{R}}.
\]
Then, in the same way as the construction of $T_{\mathrm{crys}}$, we can define a functor
\[
\mathrm{Vect}^{\mathrm{an},\varphi}((\fX,\cM_{\fX})_{\Prism})\to \mathrm{Isoc}^{\varphi}_{\mathrm{lftf}}((X_{p=0},\cM_{X_{p=0}})_{\mathrm{crys}}),
\]
denoted by $T_{\mathrm{isoc}}$. Since the functor $T_{\mathrm{isoc}}$ is compatible with $T_{\mathrm{crys}}$ in the sense that $T_{\mathrm{isoc}}(\cE)$ is isomorphic to $T_{\mathrm{crys}}(\cE)$ as isocrystals, this gives a Frobenius structure on $T_{\mathrm{crys}}(\cE)[1/p]$. 
\end{construction}

\begin{dfn}[Crystalline realization]
For a semi-stable log formal scheme $(\fX,\cM_{\fX})$ over $\cO_{K}$, we constructed functors
\begin{align*}
T_{\mathrm{crys}}\colon \mathrm{Vect}^{\mathrm{an}}((\fX,\cM_{\fX})_{\Prism})&\to \mathrm{Crys}_{\mathrm{tf}}((X_{p=0},\cM_{X_{p=0}})_{\mathrm{crys}}) \\
T_{\mathrm{isoc}}\colon \mathrm{Vect}^{\mathrm{an},\varphi}((\fX,\cM_{\fX})_{\Prism})&\to \mathrm{Isoc}^{\varphi}_{\mathrm{lftf}}((X_{p=0},\cM_{X_{p=0}})_{\mathrm{crys}})
\end{align*}
in Construction \ref{crys real by using bk prisms}. These are called \emph{crystalline realization functors}. 
\end{dfn}

\subsection{The relationship between \'{e}tale and crystalline realization}

Let $(\fX,\cM_{\fX})$ be a semi-stable log formal scheme over $\cO_{K}$.

\begin{setting}\label{setting for big integral perfd}
    Let $(\fX,\cM_{\fX})$ be a semi-stable log formal scheme over $\cO_{K}$ and $\cE$ be an analytic log prismatic $F$-crystal on $(\fX,\cM_{\fX})$. Consider a $p$-adic saturated log formal scheme $(\mathrm{Spf}(S),\cM_{S})$ over $(\fX,\cM_{\fX})$ satisfying the following conditions:
    \begin{itemize}
        \item $S$ is an integral perfectoid $\cO_{C}$-algebra;
        \item $T_{\et}(\cE)|_{(\mathrm{Spf}(S),\cM_{S})_{\eta}}$ is a constant $\bZ_{p}$-local system;
        \item the structure map $(\mathrm{Spf}(S),\cM_{S})\to (\fX,\cM_{\fX})$ factors as
        \[
        (\mathrm{Spf}(S),\cM_{S})\to (\mathrm{Spf}(\widetilde{R}_{\infty,\alpha}),\cM_{\widetilde{R}_{\infty,\alpha}})\to (\mathrm{Spf}(R),\cM_{R})\to (\fX,\cM_{\fX}),
        \]
        where $(\mathrm{Spf}(R),\cM_{R})$ is small affine with a fixed framing, $(\mathrm{Spf}(R),\cM_{R})\to (\fX,\cM_{\fX})$ is strict \'{e}tale, and $(\mathrm{Spf}(S),\cM_{S})\to (\mathrm{Spf}(\widetilde{R}_{\infty,\alpha}),\cM_{\widetilde{R}_{\infty,\alpha}})$ is a strict map over $\cO_{C}$.
    \end{itemize}
    In this situation, there exists a unique log structure $\cM_{A_{\mathrm{inf}}(S)}$ on $\mathrm{Spf}(A_{\mathrm{inf}}(S))$ which restricts to $\cM_{S}$ by Lemma \ref{lift of div log str}. We have a log prism $(A_{\mathrm{inf}}(S),(\xi),\cM_{A_{\mathrm{inf}}(S)})$.
\end{setting}

\begin{prop}\label{Laurent isom after inverting mu}
    Let $\cE$ be an analytic log prismatic $F$-crystal on $(\fX,\cM_{\fX})$. Let $(\mathrm{Spf}(S),\cM_{S})$ be a saturated $p$-adic log formal scheme over $(\fX,\cM_{\fX})$ as in Setting \ref{setting for big integral perfd}. Then there exists a natural isomorphism
    \[
    \cE_{A_{\mathrm{inf}}(S)}[1/\phi^{-1}(\mu)]\cong T_{\et}(\cE)|_{(\mathrm{Spf}(S),\cM_{S})_{\eta}}\otimes_{\bZ_{p}} A_{\mathrm{inf}}(S)[1/\phi^{-1}(\mu)].
    \]
    Here, $T_{\et}(\cE)$ is regarded as a finite free $\bZ_{p}$-module.
\end{prop}

\begin{proof}
    By Remark \ref{Laurent isom}, we have a natural isomorphism
    \[
    \cE_{A_{\mathrm{inf}}(S)}[1/\phi^{-1}(\mu)]\cong T_{\et}(\cE)|_{(\mathrm{Spf}(S),\cM_{S})_{\eta}}\otimes_{\bZ_{p}} A_{\mathrm{inf}}(S)[1/\phi^{-1}(\mu)].
    \]
    It suffices to prove that this restricts to the isomorphism in the statement. This follows from the argument in the proof of \cite[Lemma 3.13]{gr24}, which reduces the problem to \cite[Lemma 4.26]{bms18} via arc descent. Note that, in \cite{gr24},  the definition of Breuil-Kisin-Fargues modules over $\cO_{C}$ in \cite[Definition 4.22]{bms18} is adopted. This object naturally corresponds to analytic prismatic $F$-crystals over $(A_{\mathrm{inf}}(\cO_{C}),(\phi(\xi)))$ (not over $(A_{\mathrm{inf}}(\cO_{C}),\xi)$).
\end{proof}

\begin{thm}\label{semistableness of et real}
    Let $\cE$ be an analytic log prismatic $F$-crystal $\cE$ on $(\fX,\cM_{\fX})$. Then the $\bQ_{p}$-local system $T_{\et}(\cE)[1/p]$ on $(\fX,\cM_{\fX})_{\eta}$ and the $F$-isocrystal $T_{\mathrm{isoc}}(\cE)$ on $(X,\cM_{X})$ are associated. In particular, $T_{\et}(\cE)[1/p]$ is a semi-stable local system.
\end{thm}

\begin{proof}
    It suffices to show that there exists a natural isomorphism
    \[
    (T_{\et}(\cE)\otimes_{\bZ_{p}} \bB_{\mathrm{crys}})(U)\cong \bB_{\mathrm{crys}}(T_{\mathrm{crys}}(\cE))(U)
    \]
    that is compatible with Frobenius
    for a big log affinoid perfectoid $U$ in $(\fX,\cM_{\fX})_{\eta,\mathrm{prok\et}}$ with $T_{\et}(\cE)|_{U}$ being constant. If we set $\widehat{U}\coloneqq \mathrm{Spa}(S,S^{+})$ to be the associated affinoid perfectoid space with $U$, the natural map $(\mathrm{Spf}(S^{+}),\cM_{S^{+}})\to (\fX,\cM_{\fX})$ satisfies the conditions in Setting \ref{setting for big integral perfd}. Taking the base change of the isomorphism in Proposition \ref{Laurent isom after inverting mu} along a ring map 
    \[
    A_{\mathrm{inf}}(S^{+})[1/\phi^{-1}(\mu)]\stackrel{\phi}{\to} A_{\mathrm{inf}}(S^{+})[1/\mu]\to B_{\mathrm{crys}}(S^{+}),
    \]
    we obtain isomorphisms
    \begin{align*}
    (T_{\et}(\cE)\otimes_{\bZ_{p}} \bB_{\mathrm{crys}})(U)&\cong T_{\et}(\cE)|_{U}\otimes_{\bZ_{p}} B_{\mathrm{crys}}(S^{+}) \\
    &\cong \cE_{A_{\mathrm{inf}}(S^{+})}[1/\phi^{-1}(\mu)]\otimes_{A_{\mathrm{inf}}(S^{+})[1/\phi^{-1}(\mu)],\phi} B_{\mathrm{crys}}(S^{+})    
    \end{align*}
    which are compatible with Frobenius. 
    
    For the right hand side, we have isomorphisms which are compatible with Frobenius isomorphisms
    \begin{align*}
        &\cE_{A_{\mathrm{inf}}(S^{+})}[1/\phi^{-1}(\mu)]\otimes_{A_{\mathrm{inf}}(S^{+})[1/\phi^{-1}(\mu)],\phi} B_{\mathrm{crys}}(S^{+}) \\
        &\cong \cE_{A_{\mathrm{crys}}(S^{+})}\otimes_{B^{+}_{\mathrm{crys}}(S^{+})} B_{\mathrm{crys}}(S^{+}) \\
        &\cong T_{\mathrm{crys}}(\cE)_{(A_{\mathrm{crys}}(S^{+})\twoheadrightarrow S^{+}/p,\cM_{A_{\mathrm{crys}}(S^{+})})}\otimes_{B^{+}_{\mathrm{crys}}(S^{+})} B_{\mathrm{crys}}(S^{+}) \\
        &\cong \bB_{\mathrm{crys}}(T_{\mathrm{crys}}(\cE))(U),
    \end{align*}
    where the first isomorphism is defined by the crystal property of $\cE$ for a log prism map $\phi\colon (A_{\mathrm{inf}}(S^{+}),(\xi),\cM_{A_{\mathrm{inf}}(S^{+})})\to (A_{\mathrm{crys}}(S^{+}),(p),\cM_{A_{\mathrm{crys}}(S^{+})})$, and the second isomorphism comes from the definition of crystalline realization functor. By composing this with the isomorphism in the previous paragraph, we obtain the desired isomorphism.
\end{proof}

\begin{thm}\label{existence of filfisoc real}
    The crystalline realization functor $T_{\mathrm{isoc}}$ is uniquely enhanced to a functor
    \[
    T_{\mathrm{fisoc}}\colon \mathrm{Vect}^{\mathrm{an},\varphi}((\fX,\cM_{\fX})_{\Prism})\to \mathrm{FilIsoc}^{\varphi}(\fX,\cM_{\fX})
    \]
    such that, for every $\cE\in \mathrm{Vect}^{\mathrm{an},\varphi}((\fX,\cM_{\fX})_{\Prism})$, $T_{\et}(\cE)$ and $T_{\mathrm{fisoc}}(\cE)$ are associated in the sense of Definition \ref{def of semist loc sys}(2).
\end{thm}

\begin{proof}
    By Proposition \ref{weakly semi-stable implies semi-stable} and Theorem \ref{semistableness of et real}, there exists a unique filtration $\mathrm{Fil}^{\bullet}_{\mathrm{dR}}$ on $T_{\mathrm{isoc}}(\cE)_{(X_{p=0},\fX,\cM_{\fX})}$ such that $T_{\et}(\cE)$ and $T_{\mathrm{fisoc}}(\cE)\coloneqq (T_{\mathrm{isoc}}(\cE)_{(X_{p=0},\fX,\cM_{\fX})},\mathrm{Fil}^{\bullet}_{\mathrm{dR}})$ are associated in the sense of Definition \ref{def of semist loc sys} (2). Then the functor $T_{\mathrm{fisoc}}$ defined in this way is the desired one.
\end{proof}

\begin{dfn}[De Rham realization]
    Let $\mathrm{FilVect}(\fX_{\eta})$ be the category of filtered vector bundles on $\fX_{\eta}$. The functor
    \[
    \mathrm{Vect}^{\mathrm{an},\varphi}((\fX,\cM_{\fX})_{\Prism})\to \mathrm{FilVect}(\fX_{\eta})
    \]
    sending $\cE$ to $(T_{\mathrm{isoc}}(\cE)_{X_{p=0},\fX,\cM_{\fX}},\mathrm{Fil}^{\bullet}_{\mathrm{dR}})$ is called a \emph{de Rham realization functor}, denoted by $T_{\mathrm{dR}}$.
\end{dfn}

In the next subsection, we give a description of $T_{\mathrm{dR}}$ using the Nygaard filtration. To do this, we make some preparations.

Let $(\mathrm{Spf}(S),\cM_{S})$ be a $p$-adic saturated log formal scheme over $(\fX,\cM_{\fX})$ as in Setting \ref{setting for big integral perfd}. We put a filtration $\mathrm{Fil}^{\bullet}_{\xi}$ on $T_{\et}(\cE)|_{(\mathrm{Spf}(S),\cM_{S})_{\eta}}\otimes_{\bZ_{p}} B_{\mathrm{dR}}(S)$ (resp. $\cE_{A_\mathrm{inf}(S)}\otimes_{A_\mathrm{inf}(S)} B_{\mathrm{dR}}(S)$) by
\begin{align*}
\mathrm{Fil}^{\bullet}_{\xi}&\coloneqq T_{\et}(\cE)|_{(\mathrm{Spf}(S),\cM_{S})_{\eta}}\otimes_{\bZ_{p}} \mathrm{Fil}^{\bullet}B_{\mathrm{dR}}(S) \\(\text{resp.}~\mathrm{Fil}^{\bullet}_{\xi}&\coloneqq \cE_{A_{\mathrm{inf}}(S)}\otimes_{A_{\mathrm{inf}}(S)} \mathrm{Fil}^{\bullet}B_{\mathrm{dR}}(S)).    
\end{align*}
Moreover we define a filtration $\mathrm{Fil}^{\bullet}_{\mathrm{dR},B_{\mathrm{dR}}(S)}$ on 
\[
\cE_{A_{\mathrm{inf}}(S)}\otimes_{A_{\mathrm{inf}}(S),\phi} B_{\mathrm{dR}}(S)\cong T_{\mathrm{crys}}(\cE)_{(A_{\mathrm{crys}}(S)\twoheadrightarrow S/p,\cM_{A_{\mathrm{crys}}(S)})}\otimes_{A_{\mathrm{crys}}(S)} B_{\mathrm{dR}}(S)
\]
in the same way as Construction \ref{bcrys mod ass to isocrystal}. 

\begin{lem}\label{nyg filt eq to de rham filt over BdR}
We have the following commutative diagram of isomorphisms of filtered $B_{\mathrm{dR}}(S)$-modules:
    \[
    \begin{tikzcd}
        (T_{\et}(\cE)|_{(\mathrm{Spf}(S),\cM_{S})_{\eta}}\otimes_{\bZ_{p}} B_{\mathrm{dR}}(S),\mathrm{Fil}^{\bullet}_{\xi}) \ar[r,phantom, "\cong"] \ar[d,phantom, "=" sloped] & (\cE_{A_{\mathrm{inf}}(S)}\otimes_{A_{\mathrm{inf}}(S),\phi} B_{\mathrm{dR}}(S),\mathrm{Fil}^{\bullet}_{\mathrm{dR},B_{\mathrm{dR}}(S)}) \ar[d, "\varphi_{\cE}\otimes \mathrm{id}","\sim"'  sloped] \\
        (T_{\et}(\cE)|_{(\mathrm{Spf}(S),\cM_{S})_{\eta}}\otimes_{\bZ_{p}} B_{\mathrm{dR}}(S),\mathrm{Fil}^{\bullet}_{\xi}) \ar[r,phantom, "\cong"] & (\cE_{A_{\mathrm{inf}}(S)}\otimes_{A_{\mathrm{inf}}(S)} B_{\mathrm{dR}}(S),\mathrm{Fil}^{\bullet}_{\xi}).
    \end{tikzcd}
    \]
\end{lem}

\begin{proof}
    Since the isomorphism in Proposition \ref{Laurent isom after inverting mu} is obtained as the restriction of the isomorphism in Remark \ref{Laurent isom} which is compatible with Frobenius, we have a commutative diagram
    \[
    \begin{tikzcd}
        T_{\et}(\cE)|_{(\mathrm{Spf}(S),\cM_{S})_{\eta}}\otimes_{\bZ_{p}} A_{\mathrm{inf}}(S)[1/\phi^{-1}(\mu)] \ar[r,phantom, "\cong"] \ar[d,"\mathrm{id}\otimes \phi","\sim"' sloped] & \cE_{A_{\mathrm{inf}}(S)}[1/\phi^{-1}(\mu)] \ar[d,"\varphi_{\cE}","\sim"' sloped] \\
        T_{\et}(\cE)|_{(\mathrm{Spf}(S),\cM_{S})_{\eta}} \otimes_{\bZ_{p}} A_{\mathrm{inf}}(S)[1/\mu] \ar[r,phantom, "\cong"] & \cE_{A_{\mathrm{inf}}(S)}[1/\mu],
    \end{tikzcd}
    \]
    where the upper horizontal map is the isomorphism in Proposition \ref{Laurent isom after inverting mu}, the bottom horizontal map is the localization of the upper horizontal isomorphism. Taking the base change of this diagram along ring maps
    \[
    A_{\mathrm{inf}}(S)[1/\phi^{-1}(\mu)]\stackrel{\phi}{\isom} A_{\mathrm{inf}}(S)[1/\mu]\to B_{\mathrm{dR}}(S),
    \]
    we obtain the commutative diagram in the statement.

    We shall prove that every isomorphism in the square in the statement is a filtered isomorphism. Since $\phi^{-1}(\mu)$ is invertible in $B^{+}_{\mathrm{dR}}(S)$, we have a ring map $A_{\mathrm{inf}}(S)[1/\phi^{-1}(\mu)]\to B^{+}_{\mathrm{dR}}(S)$, and the bottom horizontal isomorphism is obtained as the base change of the isomorphism in Proposition \ref{Laurent isom after inverting mu} along 
    \[
    A_{\mathrm{inf}}(S)[1/\phi^{-1}(\mu)]\to B^{+}_{\mathrm{dR}}(S)\to B_{\mathrm{dR}}(S).
    \]
    This implies that the bottom horizontal isomorphism in the diagram in the statement is a filtered isomorphism.  
    
    For the upper horizontal isomorphism, note that $(\mathrm{Spf}(S),\cM_{S})_{\eta}$ does not necessarily live in $(\fX,\cM_{\fX})_{\eta,\mathrm{prok\et}}$. However, when we take a factorization of $(\mathrm{Spf}(S),\cM_{S})\to (\fX,\cM_{\fX})$ as in Setting \ref{setting for big integral perfd}, by taking the filtered base change of the filtered isomorphism
    \begin{align*}
    &(T_{\et}(\cE)|_{(\mathrm{Spf}(\widetilde{R}_{\infty,\alpha}),\cM_{\widetilde{R}_{\infty,\alpha}})_{\eta}}\otimes_{\bZ_{p}} B_{\mathrm{dR}}(\widetilde{R}_{\infty,\alpha}),\mathrm{Fil}^{\bullet}_{\xi}) \\
    &\cong (\cE_{A_{\mathrm{inf}}(\widetilde{R}_{\infty,\alpha})}\otimes_{A_{\mathrm{inf}}(\widetilde{R}_{\infty,\alpha}),\phi} B_{\mathrm{dR}}(\widetilde{R}_{\infty,\alpha}),\mathrm{Fil}^{\bullet}_{\mathrm{dR},B_{\mathrm{dR}}(\widetilde{R}_{\infty,\alpha})})
    \end{align*}
    given by Proposition \ref{weakly semi-stable implies semi-stable} along $B_{\mathrm{dR}}(\widetilde{R}_{\infty,\alpha})\to B_{\mathrm{dR}}(S)$, we see that the upper horizontal isomorphism is a filtered isomorphism. Therefore, the right vertical morphism is also a filtered isomorphism.
\end{proof}

\subsection{The de Rham realization functor for log prismatic \texorpdfstring{$F$}--crystals}\noindent

Let $(\fX,\cM_{\fX})=(\mathrm{Spf}(R),\cM_{R})$ be a small affine log formal scheme over $\cO_{K}$ with a fixed framing $t$. Let $\cE$ be an analytic log prismatic $F$-crystal on $(\fX,\cM_{\fX})$. The goal of this subsection is to give the description of the filtered vector bundle $(T_{\mathrm{isoc}}(\cE)_{(X_{p=0},\fX,\cM_{\fX})},\mathrm{Fil}^{\bullet}_{\mathrm{dR}})$ by using Breuil-Kisin log prisms and Nygaard filtrations (Proposition \ref{nyg filt and de rham filt are eq}). 

For a log prism $(A,I,\cM_{A})\in (\fX,\cM_{\fX})_{\Prism}$, we set
\begin{align*}
\cE[1/p]_{A}&\coloneqq \Gamma(\mathrm{Spec}(A[1/p]),\cE_{A}), \\
(\phi^{*}\cE)[1/p]_{A}&\coloneqq \Gamma(\mathrm{Spec}(A[1/p]),\phi_{A}^{*}\cE_{A}), \\
\overline{\cE[1/p]}_{A}&\coloneqq \cE_{A}[1/p]\otimes_{A[1/p]} A/I[1/p], \\
\overline{(\phi^{*}\cE)[1/p]}_{A}&\coloneqq (\phi^{*}\cE)[1/p]_{A}\otimes_{A[1/p]} A/I[1/p].
\end{align*}
The map $\varphi_{\cE}:(\phi^{*}\cE)[1/\cI_{\Prism}]\isom \cE[1/\cI_{\Prism}]$ induces a map 
\[
\varphi_{\cE_{A}}:(\phi^{*}\cE)_{A}[1/pI]\isom \cE_{A}[1/pI]
\]
of $A[1/pI]$-modules. We define a Nygaard filtration  $\mathrm{Fil}^{\bullet}_{\cN,A}$ on $\overline{(\phi^{*}\cE)[1/p]}_{A}$ by
\[
\mathrm{Fil}^{n}_{\cN,A}\coloneqq \mathrm{Im}(\varphi_{\cE_{A}}^{-1}(I^{n}\cE_{A}[1/p])\cap (\phi^{*}\cE)[1/p]_{A}\to \overline{(\phi^{*}\cE)[1/p]}_{A}).
\]

\begin{lem}\label{crys de Rham comparison}
    There exists a natural isomorphism of $R[1/p]$-modules
    \[
    T_{\mathrm{isoc}}(\cE)_{(X_{p=0},\fX,\cM_{\fX})}\cong  \overline{(\phi^{*}\cE)[1/p]}_{\fS_{R,t}}.
    \]
\end{lem}

\begin{proof}
The crystal property of $T_{\mathrm{isoc}}(\cE)$ for the map 
\[
(S_{R,t}\twoheadrightarrow R/p,\cM_{S_{R,t}})\to (R\twoheadrightarrow R/p,\cM_{R})
\]
in $(X_{p=0},\cM_{X_{p=0}})_{\mathrm{crys}}^{\mathrm{str}}$ gives an isomorphism of $R[1/p]$-modules
\[
T_{\mathrm{isoc}}(\cE)_{(X_{p=0},\fX,\cM_{\fX})}\cong T_{\mathrm{isoc}}(\cE)_{(S_{R,t}\twoheadrightarrow R/p,\cM_{S_{R,t}})}\otimes_{S_{R,t}[1/p]} R[1/p].
\]
By the definition of $T_{\mathrm{isoc}}(\cE)$ and the crystal property of $\cE$ of the map
\[
\phi\colon (\fS_{R,t},(E),\cM_{\fS_{R,t}})\to (S_{R,t},(p),\phi_{S_{R,t}}^{*}\cM_{S_{R,t}})
\]
in $(\fX,\cM_{\fX})_{\Prism}^{\mathrm{str}}$, we have isomorphisms of $R[1/p]$-modules
\begin{align*}
    &T_{\mathrm{isoc}}(\cE)_{(S_{R,t}\twoheadrightarrow R/p,\cM_{S_{R,t}})}\otimes_{S_{R,t}[1/p]} R[1/p]\\
    \cong &(\Gamma(\mathrm{Spec}(\fS_{R,t}[1/p]),\cE_{\fS_{R,t}})\otimes_{\fS_{R,t}[1/p],\phi} S_{R,t}[1/p])\otimes_{S_{R,t}[1/p]} R[1/p] \\
    \cong &\Gamma(\mathrm{Spec}(\fS_{R,t}[1/p]),\phi_{\fS_{R,t}}^{*}\cE_{\fS_{R,t}})\otimes_{\fS_{R,t}[1/p]} R[1/p] \\
    \cong &\overline{(\phi^{*}\cE)[1/p]}_{\fS_{R,t}}.
\end{align*}
The combination of these isomorphisms give the desired isomorphism.
\end{proof}

\begin{lem}\label{absolute int clos}
    There exists a $p$-adic saturated log formal scheme $(\mathrm{Spf}(S),\cM_{S})$ over $(\fX,\cM_{\fX})$ satisfying the following conditions:
    \begin{itemize}
        \item $S$ is an integral perfectoid $\cO_{C}$-algebra, and $R\to S$ is faithfully flat;
        \item $\bL|_{(\mathrm{Spf}(S),\cM_{S})_{\eta}}$ is constant for every $\bZ_{p}$-local system $\bL$ on $(\fX,\cM_{\fX})_{\eta,\mathrm{k\et}}$;
        \item the structure map $(\mathrm{Spf}(S),\cM_{S})\to (\fX,\cM_{\fX})$ factors as
        \[
        (\mathrm{Spf}(S),\cM_{S})\to (\mathrm{Spf}(\widetilde{R}_{\infty,\alpha}),\cM_{\widetilde{R}_{\infty,\alpha}})\to (\fX,\cM_{\fX}),
        \]
        where $(\mathrm{Spf}(S),\cM_{S})\to (\mathrm{Spf}(\widetilde{R}_{\infty,\alpha}),\cM_{\widetilde{R}_{\infty,\alpha}})$ is a strict map over $\cO_{C}$.
    \end{itemize}
\end{lem}

\begin{proof}
We may assume that $\fX$ is connected. Let $S$ denote the $p$-adic completion of the integral closure of $R$ in an algebraic closure of $\mathrm{Frac}(R)$. The ring $S$ is an integral perfectoid $\cO_{C}$-algebra by \cite[Lemma 4.20 (2)]{bha20}, and $R\to S$ is faithfully flat by \cite[Theorem 5.16]{bha20}. Since the ring map $R\to S$ admits a factorization
\[
R\to \widetilde{R}_{\infty,\alpha}\to S,
\]
we fix such a factorization and let $\cM_{S}$ denote the pullback log structure of $\cM_{\widetilde{R}_{\infty,\alpha}}$. 

It suffices to prove that $(\mathrm{Spf}(S),\cM_{S})$ satisfies the second condition in the statement. Let $\bL$ be a $\bZ_{p}$-local system on $(\fX,\cM_{\fX})_{\eta,\mathrm{k\et}}$ and let $\bL_{n}$ denote the $\bZ/p^{n}$-local system $\bL\otimes_{\bZ_{p}} \bZ/p^{n}$ for $n\geq 1$. Since $\mathrm{GL}_{s}(\bZ/p^{n+1})\to \mathrm{GL}_{s}(\bZ/p^{n})$ is surjective for $s\coloneqq \mathrm{rk}(\bL)$ and each $n\geq 1$, it is enough to show that $\bL_{n}|_{(\mathrm{Spf}(S),\cM_{S})_{\eta}}$ is constant for each $n\geq 1$.

Take a finite Kummer \'{e}tale cover $(Y_{n},\cM_{Y_{n}})\to (\fX,\cM_{\fX})_{\eta}$ such that $\bL_{n}|_{(Y_{n},\cM_{Y_{n}})}$ is constant. By \cite[Lemma 4.2.5]{dllz23b}, we may refine $(Y_{n},\cM_{Y_{n}})$ and assume that $(Y_{n},\cM_{Y_{n}})\to (\fX,\cM_{\fX})_{\eta}$ factors as
\[
(Y_{n},\cM_{Y_{n}})\to (\fX_{m,\alpha},\cM_{\fX_{m,\alpha}})_{\eta}\to  (\fX,\cM_{\fX})_{\eta}
\]
for some $m\geq 1$, where $(Y_{n},\cM_{Y_{n}})\to (\fX_{m,\alpha},\cM_{\fX_{m,\alpha}})_{\eta}$ is strict finite \'{e}tale. Here, $(\fX_{m,\alpha},\cM_{\fX_{m,\alpha}})$ is defined by
\[
(\fX_{m,\alpha},\cM_{\fX_{m,\alpha}})\coloneqq (\fX,\cM_{\fX})\times_{(\bZ[\bN^{r}],\bN^{r})^{a}} (\bZ[\frac{1}{m}\bN^{r}],\frac{1}{m}\bN^{r}).
\]
Then $(\mathrm{Spf}(S),\cM_{S})\to (\fX,\cM_{\fX})$ factors as
\[
(\mathrm{Spf}(S),\cM_{S})\to (\mathrm{Spf}(\widetilde{R}_{\infty,\alpha}),\cM_{\widetilde{R}_{\infty,\alpha}})\to (\fX_{m,\alpha},\cM_{\fX_{m,\alpha}})\to (\fX,\cM_{\fX}), 
\]
and, when we fix such a factorization, $(\mathrm{Spf}(S),\cM_{S})_{\eta}\to (\fX_{m,\alpha},\cM_{\fX_{m,\alpha}})_{\eta}$ factors as
\[
(\mathrm{Spf}(S),\cM_{S})_{\eta}\to (Y_{n},\cM_{Y_{n}})\to  (\fX_{m,\alpha},\cM_{\fX_{m,\alpha}})_{\eta}
\]
by the same argument as the forth paragraph of the proof of \cite[Lemma 4.10]{gr24}. Therefore, $\bL_{n}|_{(\mathrm{Spf}(S),\cM_{S})_{\eta}}$ is constant.
\end{proof}

\begin{prop}\label{nyg filt and de rham filt are eq}
    The isomorphism in Lemma \ref{crys de Rham comparison} is refined to an isomorphism of filtered vector bundles on $\fX_{\eta}$
    \[
    T_{\mathrm{dR}}(\cE)\cong (\overline{(\phi^{*}\cE)[1/p]}_{\fS_{R,t}},\mathrm{Fil}^{\bullet}_{\cN,\fS_{R,t}}).
    \]
    In particular, the filtration $\mathrm{Fil}^{\bullet}_{\cN,\fS_{R,t}}$ on $\overline{(\phi^{*}\cE)[1/p]}_{\fS_{R,t}}$ is a filtration with locally free graded pieces.
\end{prop}

\begin{proof}
Take a sequence of $p$-adic saturated log formal schemes
\[
(\mathrm{Spf}(S),\cM_{S})\to (\mathrm{Spf}(\widetilde{R}_{\infty,\alpha}),\cM_{\widetilde{R}_{\infty,\alpha}})\to (\fX,\cM_{\fX})
\]
as in Lemma \ref{absolute int clos}. Since $R\to S$ is faithfully flat, it is enough to show that the isomorphism in Lemma \ref{crys de Rham comparison} is a filtered isomorphism after taking the base change along $R[1/p]\to S[1/p]$.

We prepare some categories and functors as follows:
\begin{itemize}
    \item $\mathrm{Mod}^{\varphi}(\fS_{R}[1/p])$ denote the category of a pair $(M,\varphi_{M})$ consisting of a finite projective $\fS_{R}[1/p]$-module $M$ with an isomorphism $\varphi_{M}\colon (\phi_{\fS_{R}}^{*}M)[1/E]\isom M[1/E]$;
    \item $\mathrm{Mor}(B^{+}_{\mathrm{dR}}(S))$ denote the category of a triple $(M_{1},M_{2},f)$ consisting of finite projective $B^{+}_{\mathrm{dR}}(S)$-modules $M_{1}$ and $M_{2}$ with a $B_{\mathrm{dR}}(S)$-linear isomorphism 
    \[
    M_{1}\otimes_{B^{+}_{\mathrm{dR}}(S)} B_{\mathrm{dR}}(S)\isom M_{2}\otimes_{B^{+}_{\mathrm{dR}}(S)} B_{\mathrm{dR}}(S);
    \]
    \item the functor $\mathrm{Mod}^{\varphi}(\fS_{R}[1/p])\to \mathrm{Mor}(B^{+}_{\mathrm{dR}}(S))$ given by 
    \[
    (M,\varphi_{M})\mapsto (M\otimes_{\fS_{R},\phi} B^{+}_{\mathrm{dR}}(S),M\otimes_{\fS_{R}} B^{+}_{\mathrm{dR}}(S),\varphi_{M}\otimes \mathrm{id}),
    \]
    where $M\otimes_{\fS_{R},\phi} B^{+}_{\mathrm{dR}}(S)$ (resp. $M\otimes_{\fS_{R}} B^{+}_{\mathrm{dR}}(S)$) is the base change of $M$ along the ring map
    \[
    \fS_{R}\stackrel{\phi}{\to} \fS_{R}\to A_{\mathrm{inf}}(S)\to B_{\mathrm{dR}}(S) \ \ \ 
    (\text{resp.}~\fS_{R}\to A_{\mathrm{inf}}(S)\to B_{\mathrm{dR}}(S))
    \]
    (for the definition of $\fS_{R}\to A_{\mathrm{inf}}(S)$, see the last paragraph of Construction \ref{bk prism for semi-stable reduction});
    \item for a ring $A$, $\mathrm{FilMod}(A)$ denote the category of a finite projective $A$-module with a descending filtration;
    \item the functor $\mathrm{Mod}^{\varphi}(\fS_{R}[1/p])\to \mathrm{FilMod}(R[1/p])$ sending $(M,\varphi_{M})$ to the $R[1/p]$-module $M\otimes_{\fS_{R}[1/p]} R[1/p]$ equipped with the filtration defined as
    \[
    \mathrm{Im}(\phi^{*}_{\fS_{R}}M\cap \varphi_{M}^{-1}(E^{\bullet}M)\to M\otimes_{\fS_{R}[1/p]} R[1/p]);
    \]
    \item the functor $\mathrm{Mor}(B^{+}_{\mathrm{dR}}(S))\to \mathrm{FilMod}(S[1/p])$ sending $(M_{1},M_{2},f)$ to the $B^{+}_{\mathrm
    {dR}}(S)$-module $M_{1}\otimes_{B^{+}_{\mathrm
    {dR}}(S)} S[1/p]$ equipped with the filtration defined as
    \[
    \mathrm{Im}(M_{1}\cap f^{-1}(\xi^{\bullet}M_{2})\to M_{1}\otimes_{B^{+}_{\mathrm
    {dR}}(S)} S[1/p]).
    \]
\end{itemize}
Then, by flatness of the ring map $\fS_{R}\to B^{+}_{\mathrm{dR}}(S)$, the following diagram is commutative:
    \[
    \begin{tikzcd}
    \mathrm{Mod}^{\varphi}(\fS_{R}[1/p]) \ar[r] \ar[d] & \mathrm{FilMod}(R[1/p]) \ar[d,"{-\otimes_{R[1/p]} S[1/p]}"] \\
    \mathrm{Mor}(B^{+}_{\mathrm{dR}}(S)) \ar[r]  & \mathrm{FilMod}(S[1/p]).
    \end{tikzcd}
    \]
    Let 
    \[
    \varphi_{B_{\mathrm{dR}}(S)}\colon \cE_{\fS_{R}}\otimes_{\fS_{R},\phi} B_{\mathrm{dR}}(S)\isom \cE_{\fS_{R}}\otimes_{\fS_{R}} B_{\mathrm{dR}}(S)
    \]
    be the image of $\cE[1/p]_{\fS_{R}}\in \mathrm{Mod}^{\varphi}(\fS_{R}[1/p])$ via the left vertical functor. We define a filtration $\mathrm{Fil}^{\bullet}_{\cN,B_{\mathrm{dR}}(S)}$ on $\cE_{\fS_{R}}\otimes_{\fS_{R},\phi} B^{+}_{\mathrm{dR}}(S)$ by
    \[
    \mathrm{Fil}^{\bullet}_{\cN,B_{\mathrm{dR}}(S)}\coloneqq (\cE_{\fS_{R}}\otimes_{\fS_{R},\phi} B^{+}_{\mathrm{dR}}(S))\cap \varphi_{B_{\mathrm{dR}}(S)}^{-1}(\xi^{\bullet}\cdot (\cE_{\fS_{R}}\otimes_{\fS_{R},\phi} B^{+}_{\mathrm{dR}}(S))).
    \]
    Let $\overline{\mathrm{Fil}}^{\bullet}_{\cN,B_{\mathrm{dR}}(S)}$ denote the filtration on $\cE_{\fS_{R}}\otimes_{\fS_{R},\phi} B^{+}_{\mathrm{dR}}(S)\otimes_{B^{+}_{\mathrm
    {dR}}(S)} S[1/p]$ defined as the image of $\mathrm{Fil}^{\bullet}_{\cN,B_{\mathrm{dR}}(S)}$ via the surjection
    \[
    \cE_{\fS_{R}}\otimes_{\fS_{R},\phi} B^{+}_{\mathrm{dR}}(S)\twoheadrightarrow \cE_{\fS_{R}}\otimes_{\fS_{R},\phi} B^{+}_{\mathrm{dR}}(S)\otimes_{B^{+}_{\mathrm
    {dR}}(S)} S[1/p].
    \]
    Applying the commutativity of the above diagram to $\cE[1/p]_{\fS_{R}}\in \mathrm{Mod}^{\varphi}(\fS_{R}[1/p])$ gives an isomorphism
    \begin{align}
        (\overline{(\phi^{*}\cE)[1/p]}_{\fS_{R}},\mathrm{Fil}^{\bullet}_{\cN,\fS_{R}})\otimes_{R} S\cong (\cE_{A_{\mathrm{inf}}(S)}\otimes_{A_{\mathrm{inf}}(S),\phi} B^{+}_{\mathrm{dR}}(S)\otimes_{B^{+}_{\mathrm{dR}}(S)} S,\overline{\mathrm{Fil}}^{\bullet}_{\cN,B_{\mathrm{dR}}(S)}).
    \end{align}
    By Lemma \ref{nyg filt eq to de rham filt over BdR}, we have the identity of filtrations on $\cE_{A_{\mathrm{inf}}(S)}\otimes_{A_{\mathrm{inf}}(S),\phi} B^{+}_{\mathrm{dR}}(S)$
    \[
    \mathrm{Fil}^{\bullet}_{\cN,B_{\mathrm{dR}}(S)}=(\cE_{A_{\mathrm{inf}}(S)}\otimes_{A_{\mathrm{inf}}(S),\phi} B^{+}_{\mathrm{dR}}(S))\cap \mathrm{Fil}^{\bullet}_{\mathrm{dR},B_{\mathrm{dR}}(S)}.
    \]
    By the definition of $\mathrm{Fil}^{\bullet}_{\mathrm{dR}}$ and the argument in Construction \ref{bcrys mod ass to isocrystal}, when we choose $s\colon (R,\cM_{R})\to (A_{\mathrm{crys},K}(S),\cM_{A_{\mathrm{crys},K}(S)})$ as in \emph{loc. cit.}, there is an isomorphism 
    \[
    \cE_{A_{\mathrm{inf}}(S)}\otimes_{A_{\mathrm{inf}}(S),\phi} B^{+}_{\mathrm{dR}}(S)\cong T_{\mathrm{isoc}}(\cE)_{(X_{p=0},\fX,\cM_{\fX})}\otimes_{R[1/p],s} B^{+}_{\mathrm{dR}}(S)
    \] 
    such that the induced isomorphism
    \[
    \cE_{A_{\mathrm{inf}}(S)}\otimes_{A_{\mathrm{inf}}(S),\phi} B_{\mathrm{dR}}(S)\cong T_{\mathrm{isoc}}(\cE)_{(X_{p=0},\fX,\cM_{\fX})}\otimes_{R[1/p],s} B_{\mathrm{dR}}(S)
    \] 
    sends the filtration on the left hand side $\mathrm{Fil}^{\bullet}_{\mathrm{dR},B_{\mathrm{dR}}(S)}$ to the product filtration $\mathrm{Fil}^{\bullet}_{\mathrm{dR}}\otimes_{R[1/p],s} \mathrm{Fil}^{\bullet}B_{\mathrm{dR}}(S)$ on the right hand side. Therefore, we have filtered isomorphisms
    \begin{align*}
    &(\cE_{\fS_{R}}\otimes_{\fS_{R},\phi} B^{+}_{\mathrm{dR}}(S),\mathrm{Fil}^{\bullet}_{\cN,B_{\mathrm{dR}}(S)}) \\
    = &(\cE_{\fS_{R}}\otimes_{\fS_{R},\phi} B^{+}_{\mathrm{dR}}(S),(\cE_{\fS_{R}}\otimes_{\fS_{R},\phi} B^{+}_{\mathrm{dR}}(S))\cap \mathrm{Fil}^{\bullet}_{\mathrm{dR},B_{\mathrm{dR}}(S)}) \\
    \cong &(T_{\mathrm{isoc}}(\cE)_{(R\twoheadrightarrow R/p,\cM_{R})}\otimes_{R[1/p]} B^{+}_{\mathrm{dR}}(S),\bigoplus_{n\geq 0} (\mathrm{Fil}^{\bullet-n}_{\mathrm{dR}}\otimes_{R[1/p]} \mathrm{Fil}^{n}B_{\mathrm{dR}}(S))).
    \end{align*}
    Taking the mod-$\xi$ reductions with induced filtrations, we get a filtered isomorphism
    \begin{align}
        (\cE_{\fS_{R}}\otimes_{\fS_{R},\phi} B^{+}_{\mathrm{dR}}(S)\otimes_{B^{+}_{\mathrm{dR}}(S)} S[1/p],\overline{\mathrm{Fil}}^{\bullet}_{\cN,B^{+}_{\mathrm{dR}}(S)})\cong T_{\mathrm{dR}}(\cE)\otimes_{R[1/p]} S[1/p].
    \end{align}
    Then the combination of (1) and (2) proves the assertion.
\end{proof}

\subsection{The proof of fully faithfulness part of Theorem \ref{mainthm for et real}}

\begin{lem}\label{faithfulness of et real}
Let $(\mathrm{Spf}(S),\cM_{S})$ be a $p$-adic saturated log formal scheme. Suppose that $S$ is a $p$-torsion free qrsp ring, and we are given a chart $P\to \cM_{S}$ with $P$ being and uniquely $n$-divisible for every $n\geq 1$ and saturated. Then the \'{e}tale realization functor
\[
T_{\et}\colon \mathrm{Vect}^{\mathrm{an},\varphi}((\mathrm{Spf}(S),\cM_{S})_{\Prism})\to \mathrm{Loc}_{\bZ_{p}}((\mathrm{Spf}(S),\cM_{S})^{\Diamond}_{\eta,\mathrm{qprok\et}})
\]
is faithful.
\end{lem}

\begin{proof}
We see that the \'{e}tale realization functor
\[
T_{\et}\colon \mathrm{Vect}^{\mathrm{an},\varphi}(S_{\Prism})\to \mathrm{Loc}_{\bZ_{p}}(\mathrm{Spf}(S)_{\eta,\mathrm{pro\et}})
\]
is faithful by the same argument as the proof of \cite[Lemma 4.1]{gr24}. By Lemma \ref{divisible monoid and log pris site}, the source category is equivalent to $\mathrm{Vect}^{\mathrm{an},\varphi}((\mathrm{Spf}(S),\cM_{S})_{\Prism})$. On the other hand, Lemma \ref{qproket map to div log diamond is str qproet} gives an equivalence
\[
\mathrm{Loc}_{\bZ_{p}}((\mathrm{Spf}(S),\cM_{S})^{\Diamond}_{\eta,\mathrm{qprok\et}})\simeq \mathrm{Loc}_{\bZ_{p}}(\mathrm{Spf}(S)^{\Diamond}_{\eta,\mathrm{qpro\et}})\simeq \mathrm{Loc}_{\bZ_{p}}(\mathrm{Spf}(S)_{\eta,\mathrm{pro\et}}).
\]
This proves the assertion.
\end{proof}

\begin{construction}[{\cite[Construction 4.12]{gr24}}]\label{construction of latt real}

Let $S$ be a $p$-torsion free integral perfectoid $\cO_{C}$-algebra and set $\fX\coloneqq \mathrm{Spf}(S)$. The \emph{lattice realization functor} is a functor
\begin{align*}
    \mathrm{Vect}^{\mathrm{an},\varphi}(A_{\mathrm{inf}}(S),(\xi))&\to \bB^{+}_\mathrm{dR}\text{-}\mathrm{Latt} \\
    &\coloneqq\{(\bL,M)| \ \bL\in \mathrm{Loc}_{\bZ_{p}}(\fX_{\eta,\mathrm{pro\et}}), M(\subset \bL\otimes_{\bZ_{p}} \bB_{\mathrm{dR}}):\bB_{\mathrm{dR}}^{+}\text{-lattice} \}
\end{align*}
given on an object $(\cE,\varphi_{\cE})\in \mathrm{Vect}^{\mathrm{an},\varphi}(A_{\mathrm{inf}}(S),(\xi))$ as follows:
\begin{itemize}
    \item $\bL$ is the \'{e}tale realization of $(\cE,\varphi_{\cE})$.
    \item Taking the base change of the isomorphism in Proposition \ref{Laurent isom after inverting mu} along 
    \[
    \bA_{\mathrm{inf}}[1/\phi^{-1}(\mu)]\stackrel{\phi}{\to} \bA_{\mathrm{inf}}[1/\mu]\to \bB_{\mathrm{dR}},
    \]
    we get an isomorphism
    \[
    \phi_{\bB_{\mathrm{inf}}}^{*}\bB_{\mathrm{inf}}(\cE)\otimes_{\bB_{\mathrm{inf}}} \bB_{\mathrm{dR}}\cong \bL\otimes_{\bZ_{p}} \bB_{\mathrm{dR}}.
    \]
    and a $\bB^{+}_{\mathrm{dR}}$-lattice $M\coloneqq \phi_{\bB_{\mathrm{inf}}}^{*}\bB_{\mathrm{inf}}(\cE)\otimes_{\bB_{\mathrm{inf}}} \bB^{+}_{\mathrm{dR}}$ in $\bL\otimes_{\bZ_{p}} \bB_{\mathrm{dR}}$.
\end{itemize}
    
\end{construction}

\begin{lem}[{\cite[Lemma 4.3 and Proposition 4.4]{gr24}}]\label{fully faithfulness of latt real}
    Let $S$ be a $p$-torsion free perfectoid $\cO_{C}$-algebra. Suppose that $S/\pi$ is a free $\cO_{C}/\pi$-module. Then the following assertions hold.
    \begin{enumerate}
        \item We have the following formulas.
        \[
        \bigcap_{r\geq 1} \frac{\mu}{\phi^{-r}(\mu)}\Prism_{S}=\mu \Prism_{S}
        \]
        \[
        \bigcap_{r\geq 1} (\frac{\mu}{\phi^{-r}(\mu)}\Prism_{S}[1/p])=\mu \Prism_{S}[1/p]
        \]
        \[
        \bigcap_{r\geq 1} (\frac{\mu}{\phi^{-r}(\mu)}\Prism_{S}[1/\xi])=\mu \Prism_{S}[1/\xi]
        \]
        \item The lattice realization functor in Construction \ref{construction of latt real} is fully faithful.
    \end{enumerate}
    
\end{lem}

\begin{proof}
    \begin{enumerate}
        \item In \cite[Lemma 4.3]{gr24}, this assertion is proved for a specified integral perfectoid ring $S$. The argument in \emph{loc. cit.} works in this setting as well.
        \item This follows from the same argument as \cite[Proposition 4.4]{gr24}. We can use (1) as a key input.
    \end{enumerate}
\end{proof}

\begin{thm}\label{fully faithfulness of et real for anal pris}
    For a semi-stable log formal scheme $(\fX,\cM_{\fX})$ over $\cO_{K}$, the \'{e}tale realization functor 
    \[
    T_{\et}\colon \mathrm{Vect}^{\mathrm{an},\varphi}((\fX,\cM_{\fX})_{\Prism})\to \mathrm{Loc}_{\bZ_{p}}((\fX,\cM_{\fX})_{\eta})
    \]
    is fully faithful.
\end{thm}

\begin{proof}
    By working \'{e}tale locally on $\fX$, we may assume that $(\fX,\cM_{\fX})=(\mathrm{Spf}(R),\cM_{R})$ is small affine. Fix a framing $(\mathrm{Spf}(R),\cM_{R})\to (\mathrm{Spf}(R_{0}),\bN^{r})^{a}$. By Proposition \ref{kqsyn descent for anal pris crys} and Lemma \ref{faithfulness of et real}, it is enough to prove that $T_{\et}$ is full. Let $\cE_{1},\cE_{2}\in \mathrm{Vect}^{\mathrm{an},\varphi}((\fX,\cM_{\fX})_{\Prism})$ and $f\colon T_{\et}(\cE_{1})\to T_{\et}(\cE_{2})$ be a morphism of $\bZ_{p}$-local systems. 

    First, suppose that $n\leq 1$ (in other words, the special fiber of $\fX$ is smooth). In this case, by isomorphisms
    \begin{align*}
        \widetilde{R}/\pi &\cong (R\otimes_{\cO_{K}\langle \{x_{i}^{\pm 1}\},\{y_{j}\} \rangle} \cO_{K}\langle \{x_{i}^{\pm 1/p^{\infty}}\},\{y_{j}^{\bQ_{\geq 0}}\} \rangle \otimes_{\cO_{K}} \cO_{C})/\pi \\
        &\cong (R/\pi \otimes_{k[\{x_{i}^{\pm 1}\},\{y_{j}\}]} k[\{x_{i}^{\pm 1/p^{\infty}}\},\{y_{j}^{\bQ_{\geq 0}}\}])  \otimes_{k} \cO_{C}/\pi,
    \end{align*}
    $\widetilde{R}/\pi$ is a free $\cO_{C}/\pi$-module. Hence we can use Lemma \ref{fully faithfulness of latt real} and obtain the following diagram.
    \[
    \begin{tikzcd}
        & \mathrm{Vect}^{\mathrm{an},\varphi}((\fX,\cM_{\fX})_{\Prism}) \ar[rd,"\cE\mapsto (T_{\et}(\cE)\text{,}\bB^{+}_{\mathrm{dR}}(T_{\mathrm{crys}}(\cE)))"] \ar[ld] &  \\
        \mathrm{Vect}^{\mathrm{an},\varphi}(\widetilde{R}_{\Prism}) \ar[r,"\sim"] &
        \mathrm{Vect}^{\mathrm{an},\varphi}(\Prism_{\widetilde{R}},I) \ar[r,"\sim"]  & \bB_\mathrm{dR}(\widetilde{R})\text{-}\mathrm{Latt}
    \end{tikzcd}
    \]
    By Proposition \ref{bcrys functor is fully faithful}, there exists a unique map $T_{\mathrm{crys}}(\cE_{1})\to T_{\mathrm{crys}}(\cE_{2})$ of isocrystals which induces the map
    \[
    \bB_{\mathrm{crys}}(T_{\mathrm{crys}}(\cE_{1}))\cong T_{\et}(\cE_{1})\otimes_{\bZ_{p}} \bB_{\mathrm{crys}}\stackrel{f\otimes \mathrm{id}}{\to} T_{\et}(\cE_{2})\otimes_{\bZ_{p}} \bB_{\mathrm{crys}}\cong \bB_{\mathrm{crys}}(T_{\mathrm{crys}}(\cE_{2})),
    \]
    and so $f\otimes \mathrm{id}_{\bB_{\mathrm{dR}}}\colon T_{\et}(\cE_{1})\otimes_{\bZ_{p}} \bB_{\mathrm{dR}}\to T_{\et}(\cE_{1})\otimes_{\bZ_{p}} \bB_{\mathrm{dR}}$ maps the $\bB^{+}_{\mathrm{dR}}(T_{\mathrm{crys}}(\cE_{1}))$ into $\bB^{+}_{\mathrm{dR}}(T_{\mathrm{crys}}(\cE_{2}))$. Hence, the above diagram implies that there exists a unique morphism $g\colon \cE_{1}|_{\mathrm{Spf}(\widetilde{R})}\to \cE_{2}|_{\mathrm{Spf}(\widetilde{R})}$ which induces $f|_{\mathrm{Spf}(\widetilde{R})_{\eta}}$ via \'{e}tale realization. Lemma \ref{faithfulness of et real} implies that $\mathrm{pr}_{1}^{*}(g)=\mathrm{pr}_{2}^{*}(g)$, and so $g$ descents to a morphism $\cE_{1}\to \cE_{2}$ by Proposition \ref{kqsyn descent for anal pris crys}.

    We shall prove the assertion in general cases by induction on $n$. What we proved in the previous paragraph allows us to assume that $n\geq 2$. Let $\fU_{i}$ be the vanishing locus of $x_{i}\in R$ for $i=1,2$ and $\fU_{3}\coloneqq\fU_{1}\cap \fU_{2}$. Let $\cM_{\fU_{i}}$ be log structures obtained as restrictions of $\cM_{\fX}$ for $i=1,2,3$. By the induction hypothesis, we have morphisms $\cE_{1}|_{(\fU_{i},\cM_{\fU_{i}})}\to \cE_{2}|_{(\fU_{i},\cM_{\fU_{i}})}$ which induces $f$ via the \'{e}tale realization for $i=1,2$, and these morphisms coincide on $(\fU_{3},\cM_{\fU_{3}})$ by the faithfulness of the \'{e}tale realization. Hence, it suffices to prove that the restriction functor
    \begin{align*}
        &\mathrm{Vect}^{\mathrm{an},\varphi}((\fX,\cM_{\fX})_{\Prism}) \\
        &\to \mathrm{Vect}^{\mathrm{an},\varphi}((\fU_{1},\cM_{\fU_{1}})_{\Prism})\times_{\mathrm{Vect}^{\mathrm{an},\varphi}((\fU_{3},\cM_{\fU_{3}})_{\Prism})} \mathrm{Vect}^{\mathrm{an},\varphi}((\fU_{2},\cM_{\fU_{2}})_{\Prism})
    \end{align*}
    is fully faithful.

    Let $x_{3}\coloneqq x_{1}x_{2}$. A log prism $(\fS_{R}[1/x_{i}]^{\wedge}_{(p,E)},(E),\cM_{\fS_{R}[1/x_{i}]^{\wedge}_{(p,E)}})\in (\fU_{i},\cM_{\fU_{i}})_{\Prism}$ covers the final object of the associated topos, and the $n+1$-fold product of this log prism in $ (\fU_{i},\cM_{\fU_{i}})^{\mathrm{str}}_{\Prism}$ coincides with
    \[
    (\fS_{R}^{(n)}[1/x_{i}]^{\wedge}_{(p,E)},(E),\cM_{\fS_{R}^{(n)}[1/x_{i}]^{\wedge}_{(p,E)}}).
    \]
    It suffices to show that
    \begin{align*}
        &\mathrm{Vect}^{\mathrm{an},\varphi}(\fS_{R}^{(n)},(E)) \\
        &\to \mathrm{Vect}^{\mathrm{an},\varphi}(\fS_{R}^{(n)}[1/x_{1}]^{\wedge},(E))\times_{\mathrm{Vect}^{\mathrm{an},\varphi}(\fS_{R}^{(n)}[1/x_{3}]^{\wedge},(E))} \mathrm{Vect}^{\mathrm{an},\varphi}(\fS_{R}^{(n)}[1/x_{2}]^{\wedge},(E))
    \end{align*}
    is fully faithful for $n\geq 0$. Since $\mathrm{Spf}(A)\ V(p,d)$ has a covering $\mathrm{Spec}(A[1/p])\cap \mathrm{Spec}(A[1/d])$ with an intersection $\mathrm{Spec}(A[1/pd])$ for a prism $(A,(d))$, it is enough to prove that
    \begin{align*}
        &\mathrm{Vect}(\fS_{R}^{(n)}[1/a]) \\
        &\to \mathrm{Vect}(\fS_{R}^{(n)}[1/x_{1}]^{\wedge}[1/a])\times_{\mathrm{Vect}(\fS_{R}^{(n)}[1/x_{3}]^{\wedge}[1/a])} \mathrm{Vect}(\fS_{R}^{(n)}[1/x_{2}]^{\wedge}[1/a])
    \end{align*}
    is fully faithful for $a\in \{p,E,pE\}$. Taking internal homomorphisms, we are reduced to show that
    \[
    0\to \fS_{R}^{(n)}\to \fS_{R}^{(n)}[1/x_{1}]^{\wedge}\times \fS_{R}^{(n)}[1/x_{2}]^{\wedge}\to \fS_{R}^{(n)}[1/x_{3}]^{\wedge} \ \ \ (1)
    \]
    is exact. We have an exact sequence
    \[
    0\to R\to R[1/x_{1}]^{\wedge}_{p}\times R[1/x_{2}]^{\wedge}_{p} \to R[1/x_{3}]^{\wedge}_{p} \ \ \ (2)
    \]
    because the mod-$\pi$ reductions of this is obtained by the base change of an exact sequence
    \[
    0\to R_{0}/\pi\to (R_{0}/\pi)[1/x_{1}]\times (R_{0}/\pi)[1/x_{2}] \to (R_{0}/\pi)[1/x_{3}]
    \]
    (the exactness of which can be checked by a direct computation) along the \'{e}tale map $R_{0}/\pi\to R/\pi$. Therefore, the sequence $(1)$ is exact because the mod-$E$ reduction of the sequence $(2)$ is obtained by the base change of the sequence $(1)$ along the $p$-completely flat map $R\to \fS_{R}^{(n)}/(E)$.
\end{proof}

\begin{cor}\label{fully faithfulness of et real for log pris crys}
    For a semi-stable log formal scheme $(\fX,\cM_{\fX})$ over $\cO_{K}$, the \'{e}tale realization functor 
    \[
    T_{\et}:\mathrm{Vect}^{\varphi}((\fX,\cM_{\fX})_{\Prism})\to \mathrm{Loc}_{\bZ_{p}}((\fX,\cM_{\fX})_{\eta})
    \]
    is fully faithful.
\end{cor}

\begin{proof}
    This follows from Lemma \ref{fully faithfulness from pris crys to anal pris crys} and Theorem \ref{fully faithfulness of et real for anal pris}.
\end{proof}

\subsection{The proof of bi-exactness part of Theorem \ref{mainthm for et real}}

\begin{setting}
    Let $(\fX,\cM_{\fX})=(\mathrm{Spf}(R),\cM_{R})$ be a small affine log formal scheme over $\cO_{K}$ with a fixed framing. Let $\cE$ be an analytic log prismatic $F$-crystal on $(\fX,\cM_{\fX})$. Let $(\mathrm{Spf}(S),\cM_{S})$ be a $p$-adic saturated log formal scheme over $(\fX,\cM_{\fX})$ as in Lemma \ref{absolute int clos}.
    
    We simply write $T_{\et}(\cE)|_{(\mathrm{Spf}(S),\cM_{S})_{\eta}}$ for the $\bZ_{p}$-module corresponding to the constant $\bZ_{p}$-local system $T_{\et}(\cE)|_{(\mathrm{Spf}(S),\cM_{S})_{\eta}}$. 

    We consider the following rings:
    \begin{align*}
    B_{[0,1]}&\coloneqq A_{\mathrm{inf}}(S)[p/[p^{\flat}]]^{\wedge}_{[p^{\flat}]}[1/[p^{\flat}]] \\
    B_{[1,\infty]}&\coloneqq A_{\mathrm{inf}}(S)[[p^{\flat}]/p]^{\wedge}_{p}[1/p] \\
    B_{[1,1]}&\coloneqq A_{\mathrm{inf}}(S)[[p^{\flat}]/p,p/[p^{\flat}]]^{\wedge}_{p}[1/p[p^{\flat}]].    
    \end{align*}
\end{setting}

\begin{lem}(\cite[Theorem 3.8]{ked20}) \label{kedlaya lemma}
The functor
\[
\mathrm{Vect}(U(A_{\mathrm{inf}}(S)),(\xi))\to \mathrm{Vect}(B_{[0,1]})\times_{\mathrm{Vect}(B_{[1,1]})} \mathrm{Vect}(B_{[1,\infty]})
\]
is a bi-exact equivalence.
    
\end{lem}

\begin{lem}\label{lem for biexactness}
    We have canonical isomorphisms
    \begin{align*}
    \cE_{A_{\mathrm{inf}}(S)}\otimes_{A_{\mathrm{inf}}(S)} B_{[0,1]}&\cong T_{\et}(\cE)\otimes_{\bZ_{p}} B_{[0,1]} \\
    (\cE_{A_{\mathrm{inf}}(S)}\otimes_{A_{\mathrm{inf}}(S)} B_{[1,\infty]})[1/\xi]&\cong (\cE_{A_{\mathrm{crys}}(S)}\otimes_{A_{\mathrm{crys}}(S)} B_{[1,\infty]})[1/\xi].   
    \end{align*}
\end{lem}

\begin{proof}
    First, we construct the former isomorphism. The intersection formula in Lemma \ref{fully faithfulness of latt real} (1) implies that $\phi^{-1}(\mu)$ is invertible in $B_{[0,1]}$ (cf. \cite[footnote in p.44]{gr24}), and so we have a ring map $A_{\mathrm{inf}}[1/\phi^{-1}(\mu)]\to B_{[0,1]}$. Taking the base change of the isomorphism in Proposition \ref{Laurent isom after inverting mu} along this ring map, we obtain an isomorphism
    \[
    \cE_{A_{\mathrm{inf}}(S)}\otimes_{A_{\mathrm{inf}}(S)} B_{[0,1]}\cong T_{\et}(\cE)\otimes_{\bZ_{p}} B_{[0,1]}.
    \]
    
    The remaining is to construct the latter isomorphism. The crystal property of $\cE$ for a log prism map $\phi\colon (A_{\mathrm{inf}}(S),(\xi),\cM_{A_{\mathrm{inf}}(S)})\to (A_{\mathrm{crys}}(S),(p),\cM_{A_{\mathrm{crys}}(S)})$ induces an isomorphism of $B^{+}_{\mathrm{crys}}(S)$-modules
    \[
    \cE_{A_{\mathrm{crys}}(S)}\cong \cE_{A_{\mathrm{inf}}(S)}\otimes_{A_{\mathrm{inf}}(S),\phi} B^{+}_{\mathrm{crys}}(S).
    \]
    Therefore, we have isomorphisms
    \begin{align*}
    (\cE_{A_{\mathrm{crys}}(S)}\otimes_{A_{\mathrm{crys}}(S)} B_{[1,\infty]})[1/\xi]
    &\cong (\cE_{A_{\mathrm{inf}}(S)}\otimes_{A_{\mathrm{inf}}(S),\phi} B_{[1,\infty]})[1/\xi] \\
    &\cong (\phi^{*}\cE_{A_{\mathrm{inf}}(S)})[1/\xi]\otimes_{A_{\mathrm{inf}}(S)} B_{[1,\infty]} \\
    &\stackrel{\varphi_{\cE}}{\cong} \cE_{A_{\mathrm{inf}}(S)}[1/\xi]\otimes_{A_{\mathrm{inf}}(S)} B_{[1,\infty]} \\
    &\cong (\cE_{A_{\mathrm{inf}}(S)}\otimes_{A_{\mathrm{inf}}(S)} B_{[1,\infty]})[1/\xi].
    \end{align*}
\end{proof}

\begin{thm}\label{biexactness of et real}
    Let $(\fX,\cM_{\fX})$ be a semi-stable log formal scheme over $\cO_{K}$ which admits framings with $n\leq 1$ \'{e}tale locally. Then the \'{e}tale realization functor 
    \[
    T_{\et}:\mathrm{Vect}^{\mathrm{an},\varphi}((\fX,\cM_{\fX})_{\Prism})\to \mathrm{Loc}_{\bZ_{p}}((\fX,\cM_{\fX})_{\eta})
    \]
    gives a bi-exact equivalence to the essential image.
\end{thm}

\begin{proof}
    By working \'{e}tale locally on $\fX$, we may assume that $(\fX,\cM_{\fX})$ is a small affine log formal scheme with a fixed framing such that $n\leq 1$. We treat just the $n=1$ case because the same argument works in the $n=0$ case.
    
    Let $(\mathrm{Spf}(R),\cM_{R})$ be a small affine log formal scheme over $\cO_{K}$ admitting a fixed framing with $n=0$ such that $(\fX,\cM_{\fX})\cong (\mathrm{Spf}(R),\cM_{R}\oplus \pi^{\bN})^{a}$. Let $\cE_{1}\to \cE_{2}\to \cE_{3}$ be a sequence in $\mathrm{Vect}^{\mathrm{an},\varphi}((\fX,\cM_{\fX})_{\Prism})$ and suppose that
    \[
    0\to T_{\et}(\cE_{1})\to T_{\et}(\cE_{2})\to T_{\et}(\cE_{3})\to 0
    \] 
    is exact. The goal is to prove that
    \[
    0\to \cE_{1}\to \cE_{2}\to \cE_{3}\to 0
    \]
    is exact. To do this, it is enough to show that the sequence
    \[
    0\to \cE_{1,A_{\mathrm{inf}}(S)}\to \cE_{2,A_{\mathrm{inf}}(S)}\to \cE_{3,A_{\mathrm{inf}}(S)}\to 0
    \]
    is exact by Proposition \ref{kqsyn descent for anal pris crys}.
    
    First, we shall show that the sequence of $B^{+}_{\mathrm{crys}}(S)$-modules
    \[
    0\to \cE_{1,A_{\mathrm{crys}}(S)}\to \cE_{2,A_{\mathrm{crys}}(S)}\to \cE_{3,A_{\mathrm{crys}}(S)}\to 0
    \]
    is exact. In \cite{iky24}, this is proved by using Brinon's formulation. Since this argument does not work in this setting, we discuss in a different way.
    Taking the grading pieces of an exact sequence of filtered $B_{\mathrm{dR}}(S)$-modules
    \[
    0\to T_{\et}(\cE_{1})\otimes_{\bZ_{p}} B_{\mathrm{dR}}(S)\to T_{\et}(\cE_{2})\otimes_{\bZ_{p}} B_{\mathrm{dR}}(S)\to T_{\et}(\cE_{3})\otimes_{\bZ_{p}} B_{\mathrm{dR}}(S)\to 0,
    \] 
    we obtain an exact sequence
    \begin{align*}
        0\to \mathrm{gr}^{\bullet}T_{\mathrm{dR}}(\cE_{1})\otimes_{R[1/p]}  \mathrm{gr}^{\bullet}B_{\mathrm{dR}}(S)\to \mathrm{gr}^{\bullet}T_{\mathrm{dR}}(\cE_{2})&\otimes_{R[1/p]} \mathrm{gr}^{\bullet}B_{\mathrm{dR}}(S) \\
        &\to \mathrm{gr}^{\bullet}T_{\mathrm{dR}}(\cE_{3})\otimes_{R[1/p]}  \mathrm{gr}^{\bullet}B_{\mathrm{dR}}(S)\to 0.
    \end{align*}
    Since $\mathrm{gr}^{n}B_{\mathrm{dR}}(S)\cong S[1/p]$ is faithfully flat over $R[1/p]$ for each $n\geq 1$, we see that the sequence
    \[
    0\to T_{\mathrm{dR}}(\cE_{1})\to T_{\mathrm{dR}}(\cE_{2})\to T_{\mathrm{dR}}(\cE_{3})\to 0
    \]
    is exact. Let $(\mathrm{Spf}(R_{0}),\cM_{R_{0}})$ be the unramified model of $(\mathrm{Spf}(R),\cM_{R})$ defined by the framing. Let $(\mathrm{Spf}(S_{R}),\cM_{S_{R}})$ be the Breuil ring associated with $(\mathrm{Spf}(R),\cM_{R}\oplus \pi^{\bN})^{a}$. By Proposition \ref{local HK isom} and Proposition \ref{crys de Rham comparison}, we have isomorphisms
\begin{align}
    T_{\mathrm{dR}}(\cE_{i})&\cong T_{\mathrm{isoc}}(\cE_{i})_{(R_{0},\cM_{R_{0}}\oplus 0^{\bN})^{a}}\otimes_{K_{0}} K, \\
    T_{\mathrm{isoc}}(\cE_{i})_{(S_{R},\cM_{S_{R}})}&\cong T_{\mathrm{isoc}}(\cE_{i})_{(R_{0},\cM_{R_{0}}\oplus 0^{\bN})^{a}}\otimes_{R_{0}[1/p]} S_{R}[1/p],
\end{align}
for $i=1,2,3$. By (1), the sequence
\[
0\to T_{\mathrm{isoc}}(\cE_{1})_{(R_{0},\cM_{R_{0}}\oplus 0^{\bN})^{a}}\to T_{\mathrm{isoc}}(\cE_{2})_{(R_{0},\cM_{R_{0}}\oplus 0^{\bN})^{a}}\to T_{\mathrm{isoc}}(\cE_{3})_{(R_{0},\cM_{R_{0}}\oplus 0^{\bN})^{a}}\to 0
\]
is exact. By (2), the sequence
\[
0\to T_{\mathrm{isoc}}(\cE_{1})_{(S_{R},\cM_{S_{R}})}\to T_{\mathrm{isoc}}(\cE_{2})_{(S_{R},\cM_{S_{R}})}\to T_{\mathrm{isoc}}(\cE_{3})_{(S_{R},\cM_{S_{R}})}\to 0
\]
is exact. Since $(\mathrm{Spf}(S_{R}),\cM_{S_{R}})$ covers the final object of the topos associated with $(X_{p=0},\cM_{X_{p=0}})_{\mathrm{crys}}$, we have an exact sequence
\[
0\to T_{\mathrm{isoc}}(\cE_{1})_{(A_{\mathrm{crys}}(S),\cM_{A_{\mathrm{crys}}(S)})}\to T_{\mathrm{isoc}}(\cE_{1})_{(A_{\mathrm{crys}}(S),\cM_{A_{\mathrm{crys}}(S)})}\to T_{\mathrm{isoc}}(\cE_{1})_{(A_{\mathrm{crys}}(S),\cM_{A_{\mathrm{crys}}(S)})}\to 0.
\]
Since the crystal property of $\cE_{i}$ for a log prism map $(\fS_{R},(E),\cM_{\fS_{R}})\to (A_{\mathrm{crys}}(S),(p),\cM_{A_{\mathrm{crys}}})$ gives isomorphisms
\begin{align*}
    T_{\mathrm{isoc}}(\cE_{i})_{(A_{\mathrm{crys}}(S),\cM_{A_{\mathrm{crys}}}})&\cong T_{\mathrm{isoc}}(\cE_{i})_{S_{R},\cM_{S_{R}}}\otimes_{S_{R}[1/p]} A_{\mathrm{crys}}(S)[1/p] \\
    &\cong (\cE_{i,\fS_{R}}\otimes_{\fS_{R}} S_{R}[1/p])\otimes_{S_{R}[1/p]} A_{\mathrm{crys}}(S)[1/p] \\
    &\cong  \cE_{i, A_{\mathrm{crys}}(S)}
\end{align*}
for $i=1,2,3$, we conclude that the sequence
    \[
    0\to \cE_{1,A_{\mathrm{crys}}(S)}\to \cE_{2,A_{\mathrm{crys}}(S)}\to \cE_{3,A_{\mathrm{crys}}(S)}\to 0
    \]
    is exact. 
    
    By Lemma \ref{kedlaya lemma}, it suffices to prove that the following sequences are exact:
    \begin{align}
        0\to \cE_{1,A_{\mathrm{inf}}(S)}\otimes B_{[0,1]}\to \cE_{2,A_{\mathrm{inf}}(S)}\otimes B_{[0,1]}\to \cE_{3,A_{\mathrm{inf}}(S)}\otimes B_{[0,1]}\to 0 \\
        0\to \cE_{1,A_{\mathrm{inf}}(S)}\otimes B_{[1,\infty]}\to \cE_{2,A_{\mathrm{inf}}(S)}\otimes B_{[1,\infty]}\to \cE_{3,A_{\mathrm{inf}}(S)}\otimes B_{[1,\infty]}\to 0
    \end{align}
    The exactness of the sequence $(3)$ follows from Lemma \ref{lem for biexactness}. For the exactness of the sequence $(4)$, Beauville-Laszlo glueing reduces the problem to prove that the following sequences are exact:
    \begin{align}
        0\to \cE_{1,A_{\mathrm{inf}}(S)}\otimes B_{[1,\infty]}[1/\xi]\to \cE_{2,A_{\mathrm{inf}}(S)}\otimes B_{[1,\infty]}[1/\xi]\to \cE_{3,A_{\mathrm{inf}}(S)}\otimes B_{[1,\infty]}[1/\xi]\to 0 \\
        0\to (\cE_{1,A_{\mathrm{inf}}(S)}\otimes B_{[1,\infty]})^{\wedge}_{\xi}\to (\cE_{2,A_{\mathrm{inf}}(S)}\otimes B_{[1,\infty]})^{\wedge}_{\xi}\to (\cE_{3,A_{\mathrm{inf}}(S)}\otimes B_{[1,\infty]})^{\wedge}_{\xi}\to 0.
    \end{align}
    The exactness of $(5)$ follows from Lemma \ref{lem for biexactness} and what we proved in the previous paragraph, and the exactness of $(6)$ follows from the exactness of $(3)$ and the identification $(B_{[1,\infty]})^{\wedge}_{\xi}\cong (B_{[0,1]})^{\wedge}_{\xi}(\cong B^{+}_{\mathrm{dR}}(S))$. This completes the proof.
\end{proof}

\appendix

\section{Some lemmas on log schemes}

\begin{lem}\label{log str for thickening}
     Let $i:S\hookrightarrow T$ be a nilpotent closed immersion of affine schemes with an ideal sheaf $\cI$. Let $\cM_{T}$ be an integral log structure on $T$. We let $\cM_{S}$ denote the pullback log structure $i^{*}\cM_{T}$. Then a natural map $\Gamma(T,\cM_{T})\to \Gamma(S,\cM_{S})$ is a $\Gamma(T,1+\cI)$-torsor.
\end{lem}

\begin{proof}
    It is obvious that each fiber of the map $\Gamma(T,\cM_{T})\to \Gamma(S,\cM_{S})$ is a $\Gamma(T,1+\cI)$-orbit. Since $\cM_{T}$ is integral, the action of $\Gamma(T,1+\cI)$ on $\Gamma(T,\cM_{T})$ is faithful. Hence it suffices to prove that the map $\Gamma(T,\cM_{T})\to \Gamma(S,\cM_{S})$ is surjective. We may assume that $\cI^{2}=0$. We have an exact sequence
    \[
    1\to 1+\cI\to \cM_{T}^{\mathrm{gp}}\to \cM_{S}^{\mathrm{gp}}\to 1
    \]
    Since we have $H^{1}(T,1+\cI)\cong H^{1}(T,\cI)=0$, the natural map $\Gamma(T,\cM_{T}^{\mathrm{gp}})\to \Gamma(S,\cM_{S}^{\mathrm{gp}})$ is surjective, which implies that the natural map $\Gamma(T,\cM_{T})\to \Gamma(S,\cM_{S})$ is also surjective because $\cM_{T}$ is isomorphic to $\cM_{S}\times_{\cM_{S}^{\mathrm{gp}}} \cM_{T}^{\mathrm{gp}}$.
\end{proof}

\begin{dfn}\label{def of div monoid}
    Let $P$ be a monoid. For an integer $n\geq 1$, a monoid $P$ is called \emph{uniquely $n$-divisible} (resp. \emph{$n$-divisible}) if a monoid map $P\to P$ mapping $p$ to $p^{n}$ is an isomorphism (resp. a surjection).
\end{dfn}

\begin{lem}\label{lift of map from div monoid}
    Let $n\geq 1$ be an integer. Let $P$ be a uniquely $n$-divisible monoid and $\pi:Q\to R$ is a surjection of monoids. Assume that there exists an integer $k\geq 1$ such that, for any $p_{1},p_{2}\in P$ with $\pi(p_{1})=\pi(p_{2})$, we have $p_{1}^{n^{k}}=p_{2}^{n^{k}}$. Then the natural map
    \[
    \mathrm{Hom}(P,Q)\to \mathrm{Hom}(P,R)
    \]
    is bijective.
\end{lem}

\begin{proof}
Let $f:P\to R$ be a monoid map. We define a monoid map $\widetilde{f}:P\to Q$ lifting $f$ as follows. For $p\in P$, we take a unique $p'\in P$ such that $(p')^{n^{k}}=p$ and choose a lift $q'\in Q$ of $\pi(p')$. We put $\widetilde{f}(p)\coloneqq(q')^{n^{k}}$. For another lift $q''\in Q$ of $\pi(p')$, the assumption implies $(q')^{n^{k}}=(q'')^{n^{k}}$. Therefore $\widetilde{f}$ is a well-defined monoid map. By construction, $\widetilde{f}$ is a unique lift of $f$ 
\end{proof}

\begin{rem}
    when $\pi:Q\to R$ is either of the following types, the assumption in Lemma \ref{lift of map from div monoid} is satisfied:
    \begin{itemize}
        \item a natural surjection $\pi:A\to A/I$ where $A$ is a ring, $I$ is an ideal of $A$ with $I^{k}=0$ and $n\in I$
        \item a natural surjection $\pi:M\to M/G$ where $M$ is a monoid and $G$ is a subgroup of $M^{\times}$ with $G^{n^{k}}=\{1\}$.
    \end{itemize}
\end{rem}

\begin{lem}\label{lift of div log str}
    Let $i:S\hookrightarrow T$ be a nilpotent closed immersion of schemes with an ideal sheaf $\cI$. Suppose that an integer $n\geq 1$ is nilpotent on $S$. 
    \begin{enumerate}
        \item Let $\cM_{T},\cN_{T}$ be integral log structures on $T$. We let $\cM_{S},\cN_{S}$ denote the pullback log structures $i^{*}\cM_{T},i^{*}\cN_{T}$ respectively. Suppose that, \'{e}tale locally, $\cM_{T}$ admits a chart $P\to \cM_{T}$ for an $n$-divisible monoid $P$. Then any map of log structures $\cM_{S}\to \cN_{S}$ extends uniquely to a map of log structures $\cM_{T}\to \cN_{T}$.
        \item Let $\cM_{S}$ be an integral log structure on $S$. Suppose that, \'{e}tale locally, $\cM_{S}$ admits a chart $P\to \cM_{S}$ for a uniquely $n$-divisible monoid $P$. Then there exists a unique log structure $\cM_{T}$ on $T$ with $i^{*}\cM_{T}\cong \cM_{S}$.
    \end{enumerate}
\end{lem}

\begin{proof}
    \begin{enumerate}
        \item By working \'{e}tale locally on $T$, we may assume that $T$ is affine and  we have an $n$-divisible monoid $P$ and a chart $P\to \cM_{T}$. It is enough to prove that a natural map
        \[
        \mathrm{Hom}(P,\Gamma(T,\cM_{T}))\to \mathrm{Hom}(P,\Gamma(S,\cM_{S}))
        \] 
        is bijective. We may assume that $\cI^{2}=0$. Since we have $\cM_{T}\cong \cM_{S}\times_{\cM_{S}^{\mathrm{gp}}} \cM_{T}^{\mathrm{gp}}$, we are reduced to show that a natural map
        \[
        \mathrm{Hom}(P^{\mathrm{gp}},\Gamma(T,\cM_{T}^{\mathrm{gp}}))\to \mathrm{Hom}(P^{\mathrm{gp}},\Gamma(S,\cM_{S}^{\mathrm{gp}}))
        \] 
        is bijective. Lemma
        \ref{log str for thickening} gives an exact sequence
        \begin{align*}
            &\mathrm{Hom}(P^{\mathrm{gp}},\Gamma(T,\cI))\to \mathrm{Hom}(P^{\mathrm{gp}},\Gamma(T,\cM_{T}^{\mathrm{gp}}))\to \mathrm{Hom}(P^{\mathrm{gp}},\Gamma(S,\cM_{S}^{\mathrm{gp}}))\to \\
            \to &\mathrm{Ext}^{1}(P^{\mathrm{gp}},H^{1}(T,\cI)),
        \end{align*}
        by $1+\cI\cong \cI$. Since $P^{\mathrm{gp}}$ is $n$-divisible and $n^{k}\cI=0$ for a sufficiently large $k$, we have
        \[
        \mathrm{Hom}(P^{\mathrm{gp}},\Gamma(T,\cI))=\mathrm{Ext}^{1}(P^{\mathrm{gp}},H^{1}(T,\cI))=0.
        \]
        This proves the assertion.
        \item By working \'{e}tale locally on $T$, we may assume that $T$ is affine and we have a uniquely $n$-divisible monoid $P$ and a chart $P\to \cM_{S}$. Lemma \ref{lift of map from div monoid} gives a unique lift $P\to \cO_{T}$ of the map $P\to \cM_{S}\to \cO_{S}$. Let $\cM_{T}$ denote the associated log structure with the map $P\to \cO_{T}$. By construction, we have $i^{*}\cM_{T}\cong \cM_{S}$. To prove the uniqueness, we take another log structure $\cN_{T}$ on $T$ with $i^{*}\cN_{T}\cong \cM_{S}$. By (1), there exist a unique lift $\cM_{T}\to \cN_{T}$ of the identity map on $\cM_{S}$. This map $\cM_{T}\to \cN_{T}$ is an isomorphism because so is the pullback to $S$. This proves the uniqueness.
    \end{enumerate}
\end{proof}

\section{Log adic spaces and log diamonds}

In this section, we study the relationship between log adic spaces (introduced in \cite{dllz23b}) and associated log diamond (introduced in \cite{ky23}). 

\begin{lem}\label{qproket map to div log diamond is str qproet}
    Let $f\colon (X,\cM_{X})\to (Y,\cM_{Y})$ be a map of log diamonds. Suppose that we are given a divisible saturated monoid $P$ and a chart $P\to \cM_{Y}$. Then $f$ is quasi-pro-Kummer-\'{e}tale if and only if it is strict quasi-pro-\'{e}tale.
\end{lem}

\begin{proof}
    Since the "if" part is trivial, we shall prove the "only if" part. Take a strictly totally disconnected perfectoid space $Z$ and a quasi-pro-\'{e}tale cover $Z\to Y$. Let $\cM_{Z}$ denote the log structure on $Z$ obtained as the pullback of $\cM_{Y}$. Consider the following Cartesian diagram:
    \[
    \begin{tikzcd}
    (X,\cM_{X})\times_{(Y,\cM_{Y})} (Z,\cM_{Z})\ar[r] \ar[d] & (Z,\cM_{Z})  \ar[d] \\
    (X,\cM_{X}) \ar[r,"f"] & (Y,\cM_{Y}).
    \end{tikzcd}
    \]
    In this diagram, the upper horizontal map is strict quasi-pro-\'{e}tale by the definition of quasi-pro-Kummer-\'{e}tale maps, and both of vertical maps are strict quasi-pro-\'{e}tale covers. Therefore, $f$ is strict quasi-pro-\'{e}tale. 
\end{proof}

\begin{dfn}
    Let $(X,\cM_{X})$ be an fs log adic space over $\mathrm{Spa}(\bQ_{p},\bZ_{p})$. We have the associated locally spatial diamond $X^{\Diamond}$. there exists a natural morphism of ringed sites
    \[ 
    \nu\colon (X^{\Diamond}_{\mathrm{qpro\et}},\widehat{\cO}_{X^{\Diamond}})\to (X_{\et},\cO_{X}).
    \]
    Let $\cM_{X^{\Diamond}}\to \widehat{\cO}_{X^{\Diamond}}$ be the associated log structure with the prelog structure 
    \[
    \nu^{-1}\cM_{X}\to \nu^{-1}\cO_{X}\to \widehat{\cO}_{X^{\Diamond}}.
    \]
    The locally spatial log diamond $(X^{\Diamond},\cM_{X^{\Diamond}})$ is called \emph{the associated log diamond} with $(X,\cM_{X})$ and denoted by $(X,\cM_{X})^{\Diamond}$. By definition, a chart $P\to \cM_{X}$ induces a chart $P\to \cM_{X^{\Diamond}}$.
\end{dfn}

\begin{lem}
    Let $(X,\cM_{X})\to (Y,\cM_{Y})$ be a Kummer log \'{e}tale morphism of locally noetherian fs log adic spaces over $\mathrm{Spa}(\bQ_{p},\bZ_{p})$. Then the induced morphism of log diamonds $(X,\cM_{X})^{\Diamond} \to (Y,\cM_{Y})^{\Diamond}$ is Kummer log \'{e}tale.
\end{lem}

\begin{proof}
    By working \'{e}tale locally on $Y$, we may assume that $Y$ is quasi-compact and that $(Y,\cM_{Y})$ admits an fs chart $P\to \cM_{Y}$ with $P$ being a torsion free. We set 
    \[
    (Y_{n},\cM_{Y_{n}}):=(Y,\cM_{Y})\times_{(\mathrm{Spa}(\bQ_{p}\langle P \rangle,\bZ_{p}\langle P \rangle),P)^{a}} (\mathrm{Spa}(\bQ_{p}\langle P^{1/n} \rangle,\bZ_{p}\langle P^{1/n} \rangle),P)^{a}
    \]
    and let $(X_{n},\cM_{X_{n}})$ denote the saturated fiber product of $(Y_{n},\cM_{Y_{n}})$ and $(X,\cM_{X})$ over $(Y,\cM_{Y})$ for an integer. By \cite[Lemma 4.2.5]{dllz23b}, there exists an integer $n\geq 1$ such that the projection map $(X_{n},\cM_{X_{n}})\to (Y_{n},\cM_{Y_{n}})$ is strict \'{e}tale.
    
    Let $f\colon (Z,\cM_{Z})\to (Y,\cM_{Y})^{\Diamond}$ be a map from strictly totally disconnected log perfectoid space $(Z,\cM_{Z})$. Since $\Gamma(Z,\cM_{Z})$ is divisible by \cite[Definition 7.12]{ky23}, the map $P\to \cM_{Z}$ extends to a map $P_{\bQ_{\geq 0}}\to \cM_{Z}$. Hence, the map $f$ factors through $(Y_{n},\cM_{Y_{n}})$ and the base-changed map
    \[
    g\colon (Z,\cM_{Z})\times_{(Y,\cM_{Y})^{\Diamond}} (X,\cM_{X})^{\Diamond}\to (Z,\cM_{Z})
    \]
    can be written as the base change of the strict \'{e}tale map $(X_{n},\cM_{X_{n}})^{\Diamond}\to (Y_{n},\cM_{Y_{n}})^{\Diamond}$. Therefore, $g$ is also strict \'{e}tale, and so the map $(X,\cM_{X})^{\Diamond}\to (Y,\cM_{Y})^{\Diamond}$ is Kummer \'{e}tale.
\end{proof}

\begin{prop}\label{ket vs qproket}
    Let $(X,\cM_{X})$ be a locally noetherian fs log adic space. For a site $\sC$, we let $\mathrm{Loc}_{\bZ/p^{n}}(\sC)$ denote the category of $\bZ/p^{n}$-local systems on $\sC$. Then the natural functor
    \[
    \mathrm{Loc}_{\bZ/p^{n}}((X,\cM_{X})_{\mathrm{k\et}})\to \mathrm{Loc}_{\bZ/p^{n}}((X,\cM_{X})^{\Diamond}_{\mathrm{qprok\et}})
    \]
    is an equivalence.
\end{prop}

\begin{proof}
    By working \'{e}tale locally on $X$, we may assume that $(X,\cM_{X})$ admits a chart $P\to \cM_{X}$ and that $X$ is qcqs. We define log adic spaces 
    \[ (X_{m},\cM_{X_{m}})\coloneqq(X,\cM_{X})\times_{(\mathrm{Spa}(\bQ_{p}\langle P \rangle,\bZ_{p}\langle P \rangle),P)^{a}} (\mathrm{Spa}(\bQ_{p}\langle P^{1/m} \rangle,\bZ_{p}\langle P^{1/m} \rangle),P^{1/m})^{a}
    \]
    and a log diamond
    \[ (X_{\infty},\cM_{X_{\infty}})\coloneqq(X,\cM_{X})^{\Diamond}\times_{(\mathrm{Spd}(\bQ_{p}\langle P \rangle,\bZ_{p}\langle P \rangle),P)^{a}} (\mathrm{Spd}(\bQ_{p}\langle P_{\bQ_{\geq 0}} \rangle,\bZ_{p}\langle P_{\bQ_{\geq 0}} \rangle),P_{\bQ_{\geq 0}})^{a}.
    \]
    Then the natural map $(X_{\infty},\cM_{X_{\infty}})\to (X,\cM_{X})^{\Diamond}$ is a quasi-pro-Kummer \'{e}tale cover by \cite[Proposition 7.20 (2)]{ky23}, and we have the following isomorphisms:
    \[
    X_{\infty}\cong \varprojlim_{m} X_{m}^{\Diamond}, \ \ \  \ \ 
    (X_{\infty},\cM_{X_{\infty}})\cong \varprojlim_{m} (X_{m},\cM_{X_{m}})^{\Diamond}.
    \]
    Let $(X_{m}^{(\bullet)},\cM_{X_{m}^{(\bullet)}})$ (resp. $(X_{\infty}^{(\bullet)},\cM_{X_{\infty}^{(\bullet)}})$) is the \v{C}ech nerve of $(X_{m},\cM_{X_{m}})\to (X,\cM_{X})$ (resp. $(X_{\infty},\cM_{X_{\infty}})\to (X,\cM_{X})^{\Diamond}$) in the category of fs log adic spaces (resp. saturated log diamonds).
    
    By \cite[Lemma 4.2.5 and Theorem 4.4.15]{dllz23b}, every object of $\mathrm{Loc}_{\bZ/p^{n}}((X,\cM_{X})_{\mathrm{k\et}})$ comes from an object of $\mathrm{Loc}_{\bZ/p^{n}}((X_{m})_{\mathrm{\et}})$ after being pulled back to $(X_{m},\cM_{X_{m}})$ for some $m\geq 1$. In other words, we have a natural equivalence
    \[
    \mathrm{Loc}_{\bZ/p^{n}}((X,\cM_{X})_{\mathrm{k\et}})\simeq \varinjlim_{m} \varprojlim_{\bullet\in \Delta} \mathrm{Loc}_{\bZ/p^{n}}((X_{m}^{(\bullet)})_{\mathrm{\et}}).
    \]
    On the other hand, we have natural equivalences
    \begin{align*}
        \mathrm{Loc}_{\bZ/p^{n}}((X,\cM_{X})^{\Diamond}_{\mathrm{qprok\et}}) &\simeq \varprojlim_{\bullet\in \Delta} \mathrm{Loc}_{\bZ/p^{n}}((X_{\infty}^{(\bullet)},\cM_{X_{\infty}^{(\bullet)}})_{\mathrm{qprok\et}}) \\
        &\simeq \varprojlim_{\bullet\in \Delta} \mathrm{Loc}_{\bZ/p^{n}}((X_{\infty}^{(\bullet)})_{\mathrm{qpro\et}}) \\
        &\simeq \varinjlim_{m} \varprojlim_{\bullet\in \Delta} \mathrm{Loc}_{\bZ/p^{n}}((X_{m}^{(\bullet)})^{\Diamond}_{\mathrm{qpro\et}}),
    \end{align*}
    where the first equivalence is quasi-pro-Kummer \'{e}tale descent, the second one comes from Lemma \ref{qproket map to div log diamond is str qproet}, and the third one is the limit argument \cite[Proposition 11.23 (i)]{sch22}. Moreover, \cite[Lemma 15.6]{sch22} and \cite[Proposition 3.7]{mw23} imply that there exists an equivalence
    \[
    \mathrm{Loc}_{\bZ/p^{n}}((X_{m}^{(k)})_{\mathrm{\et}})\simeq \mathrm{Loc}_{\bZ/p^{n}}((X_{m}^{(k)})^{\Diamond}_{\mathrm{\et}})\simeq \mathrm{Loc}_{\bZ/p^{n}}((X_{m}^{(k)})^{\Diamond}_{\mathrm{qpro\et}})
    \]
    for each $m\geq 1$ and $k\geq 0$. This proves the assertion.
\end{proof}

\end{document}